\documentclass[journal]{IEEEtran}
\ifCLASSINFOpdf
\else
\fi
\usepackage{url}

\usepackage{graphics} 
\usepackage{graphicx}
\usepackage{times} 
\usepackage{amsmath} 
\usepackage{amssymb}  
\usepackage{algorithm}
\usepackage{xcolor}

\usepackage{verbatim}

\usepackage{amsthm}

\usepackage{mathtools}
\usepackage{bbm}

\usepackage{subfigure}

\newtheorem{theorem}{\bf Theorem} \newtheorem{definition}{\bf Definition} 
\newtheorem{lemma}{\bf Lemma} \newtheorem{remark}{\bf Remark}
 \newtheorem{corollary}{\bf Corollary} \newtheorem{proposition}{\bf Proposition} 
\newtheorem{assumption}{\bf Assumption}  
\newtheorem{Algorithm}{\bf Algorithm}

\newcommand\bbr{\mathbb{R}}

\newcommand\bbone{\mathbbm{1}}

\newcommand\calD{\mathcal{D}}
\newcommand\calZ{\mathcal{Z}}

\newcommand\rma{\mathrm{a}}
\newcommand\rmb{\mathrm{b}}
\newcommand\rmc{\mathrm{c}}
\newcommand\rmd{\mathrm{d}}
\newcommand\rme{\mathrm{e}}

\newcommand\rml{\mathrm{l}}

\newcommand\rmr{\mathrm{r}}
\newcommand\rms{\mathrm{s}}
\newcommand\rmu{\mathrm{u}}

\newcommand\rmx{\mathrm{x}}
\newcommand\rmy{\mathrm{y}}

\newcommand\rmH{\mathrm{H}}
\newcommand\rmJ{\mathrm{J}}
\newcommand\rmL{\mathrm{L}}

\newcommand\rmV{\mathrm{V}}

\newcommand\rmeq{\mathrm{eq}}
\newcommand\rmlin{\mathrm{Lin}}

\begin{document}
%
\title{Linear tracking MPC for nonlinear systems\\Part II: The data-driven case}
%
%
%

\author{Julian Berberich$^1$, 
		Johannes K\"ohler$^{1,2}$, 
		Matthias A. M\"uller$^3$, 
		and Frank Allg\"ower$^1$.
		\thanks{Funded by Deutsche Forschungsgemeinschaft (DFG, German Research Foundation) under Germany's Excellence Strategy - EXC 2075 - 390740016 and under grant 468094890. 
		We acknowledge the support by the Stuttgart Center for Simulation Science (SimTech).
This project has received funding from the European Research Council (ERC) under the European Union’s Horizon 2020 research and innovation programme (grant agreement No 948679).
The authors thank the International Max Planck Research School for Intelligent Systems (IMPRS-IS)
for supporting Julian Berberich.}
\thanks{$^1$University of Stuttgart, Institute for Systems Theory and Automatic Control, 70550 Stuttgart, Germany (email:$\{$julian.berberich, frank.allgower$\}$@ist.uni-stuttgart.de)}
\thanks{$^2$Institute for Dynamical Systems and Control, ETH Zurich, ZH-8092, Switzerland (email:jkoehle@ethz.ch)}
\thanks{$^3$Leibniz University Hannover, Institute of Automatic Control, 30167 Hannover, Germany (e-mail:mueller@irt.uni-hannover.de)}}

%
%

\markboth{}%
{}
%



\IEEEoverridecommandlockouts

\IEEEpubid{\begin{minipage}{\textwidth}\ \\[12pt] \\ \\
\copyright 2021 IEEE. Personal use of this material is permitted. Permission from IEEE must be obtained for all other uses, in any current or future media, including reprinting/republishing this material for advertising or promotional purposes, creating new collective works, for resale or redistribution to servers or lists, or reuse of any copyrighted component of this work in other works.
\end{minipage}}

\maketitle

\begin{abstract}
We present a novel data-driven model predictive control (MPC) approach to control unknown nonlinear systems using only measured input-output data with closed-loop stability guarantees.
Our scheme relies on the data-driven system parametrization provided by the Fundamental Lemma of Willems et al.
We use new input-output measurements online to update the data, exploiting local linear approximations of the underlying system.
We prove that our MPC scheme, which only requires solving strictly convex quadratic programs online, ensures that the closed loop (practically) converges to the (unknown) optimal reachable equilibrium that tracks a desired output reference while satisfying polytopic input constraints.
As intermediate results of independent interest, we extend the Fundamental Lemma to affine systems and we derive novel robustness bounds w.r.t.\ noisy data for the open-loop optimal control problem, which are directly transferable to other data-driven MPC schemes in the literature.
The applicability of our approach is illustrated with a numerical application to a continuous stirred tank reactor.
\end{abstract}


%
\IEEEpeerreviewmaketitle
\section{Introduction}
Data-driven control has received significant attention in recent years due to the abundance of available data, the potential difficulties in obtaining accurate models, and the simplicity of data-driven approaches, see~\cite{hou2013model} for an overview.
Our paper relies on the Fundamental Lemma by Willems et al.~\cite{willems2005note} which shows that one persistently exciting input-output trajectory can be used to parametrize all trajectories of a linear time-invariant (LTI) system.
This provides a promising foundation for data-driven control of LTI systems and it can, e.g., be used to design data-driven model predictive control (MPC) schemes~\cite{yang2015data,coulson2019deepc}.
Different contributions have analyzed such schemes in the presence of noise with regard to open-loop robustness~\cite{coulson2021distributionally,xue2021data,yin2021maximum,furieri2021near} or closed-loop stability/robustness both with~\cite{berberich2021guarantees} and without~\cite{bongard2021robust} terminal ingredients.
Although different successful applications to complex nonlinear systems have been reported in the literature, see, e.g.,~\cite{elokda2021quadcopters,huang2019power}, providing theoretical guarantees of data-driven MPC for nonlinear systems remains a widely open research problem.
The literature contains various extensions and variations of~\cite{willems2005note} for specific classes of nonlinear systems such as Hammerstein and Wiener systems~\cite{berberich2020trajectory}, Volterra systems~\cite{rueda2020data}, polynomial systems~\cite{martin2021dissipativity,guo2021data}, systems with rational dynamics~\cite{straesser2021data}, flat systems~\cite{alsalti2021data}, and linear parameter-varying systems~\cite{verhoek2021fundamental}.
However, all of these works assume that the system is linearly parametrized in known basis functions, which restricts their practical applicability.
In summary, although data-driven MPC is well-explored for LTI systems, there exists no unifying framework for nonlinear data-driven control.

In this paper, we propose an MPC approach to control unknown nonlinear systems with closed-loop stability guarantees by updating the data used in the data-driven system parametrization of~\cite{willems2005note} online, thereby exploiting that nonlinear systems can be approximated locally via linearization.
The basic idea is depicted in Figure~\ref{fig:manifold}.

\vspace{-7pt}

\begin{figure}[h!]
\begin{center}
\includegraphics[width=0.48\textwidth]{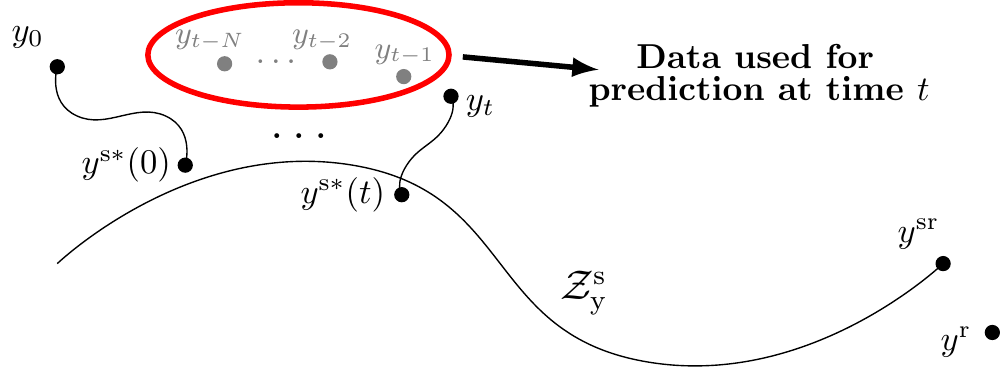}
\vspace{-12pt}
\end{center}
\caption{Graphical scheme illustrating the basic idea of our approach.
The figure displays the output equilibrium manifold $\mathcal{Z}_\rmy^\rms$, the closed-loop output and artificial equilibrium at time $0$, the closed-loop output and artificial equilibrium at time $t$, the past $N$ measurements used for prediction at time $t$, the optimal reachable equilibrium $y^{\rms\rmr}$, and the setpoint $y^\rmr$.}
\label{fig:manifold}
\end{figure}
\vspace{-5pt}

As in our companion paper~\cite{berberich2021linearpart1}, the goal is to stabilize the optimal reachable output equilibrium $y^{\rms\rmr}$ corresponding to a given setpoint $y^\rmr$, which may not lie on the output equilibrium manifold $\mathcal{Z}_\rmy^\rms$.
To this end, we employ a tracking MPC formulation with an artificial equilibrium $y^{\rms}(t)$ which is optimized online, similar to~\cite{limon2018nonlinear,koehler2020nonlinear}.
In our companion paper~\cite{berberich2021linearpart1}, we show that, if the current linearization is used for prediction, then, under suitable assumptions on the design parameters and for initial conditions close to $\mathcal{Z}_\rmy^\rms$, $y^\rms(t)$ and thus the closed-loop output $y_t$ slide along $\mathcal{Z}_\rmy^\rms$ towards $y^{\rms\rmr}$, see~\cite{berberich2021linearpart1} for details.
In the present paper, the key difference to~\cite{berberich2021linearpart1} is that no model of the nonlinear system or its linearization is available and we use past $N$ input-output measurements to predict future trajectories based on the Fundamental Lemma~\cite{willems2005note}.
Since these measurements originate\newpage \noindent from the nonlinear system, they do not provide an exact description of the linearized dynamics which poses additional challenges if compared to the model-based MPC in~\cite{berberich2021linearpart1}.
In this paper, we show that, if the system does not evolve too rapidly during the initial data collection, then the predictions are sufficiently accurate during the closed-loop operation such that practical stability can be guaranteed.
Our MPC scheme relies on 
solving strictly convex quadratic programs (QPs) and the only prior knowledge about the nonlinear system required for the implementation is a (potentially rough) upper bound on its order.
For our theoretical analysis, we assume that the closed-loop input generated by the MPC scheme is persistently exciting, but we discuss multiple practical approaches for ensuring this property and we plan to investigate this issue in more detail in future research.

The works~\cite{lawrynczuk2007computationally,papadimitriou2020control} are related to our approach since they estimate linear time-varying models of nonlinear systems from data online, but they do not provide closed-loop stability guarantees.
Controlling nonlinear systems using linear models via Koopman operator arguments has received increasing attention in recent years~\cite{korda2018linear}, also in connection to the Fundamental Lemma~\cite{lian2021koopman,lian2021nonlinear}, but typically no closed-loop guarantees can be given.
Moreover, data-driven control methods based on machine learning, cf.~\cite{sutton1998reinforcement,hewing2020cautious}, have been successfully applied but, also, they often do not provide closed-loop guarantees.

An obvious alternative to our results is provided by sequential system identification (e.g., online LTI system identification or recursive least squares estimation~\cite{ljung1987system}) and model-based MPC.
Data-driven MPC has the advantage of being more direct, only requiring to tune and solve one optimization problem, the parameters of which can even be interpreted as an implicit system identification step~\cite{doerfler2021bridging}.
Regardless, to the best of our knowledge, there are no results on closed-loop stability based on linearization arguments under similar assumptions as we consider for either identification-based or data-driven MPC.
Alternative system identification approaches for nonlinear systems combine Lipschitz continuity-like properties with set membership estimation~\cite{milanese2004set,calliess2014conservative}, possibly leading to increasingly complex models, or they require appropriately chosen basis functions, see, e.g.,~\cite{sjoeberg1995nonlinear} or approaches from nonlinear adaptive MPC~\cite{guay2015robust}.
In contrast, our data-driven MPC is direct, simple, applicable to a broad class of nonlinear systems, and it admits stability guarantees, thereby indicating potential advantages over the classical identification-based approaches.
This is possible since we implicitly encourage the closed-loop trajectory to remain in vicinity of the steady-state manifold, where our data-driven prediction model is a good approximation of the underlying nonlinear dynamics.
Initial ideas in this direction have been discussed in~\cite{berberich2021at} along with experimental results, however, without any theoretical analysis.
Finally, the presented results are also related to offset-free MPC~\cite{morari2012nonlinear}, which deals with setpoint tracking based on an uncertain model, whereas our approach achieves asymptotic convergence to the setpoint using only input-output data.

The remainder of the paper is structured as follows.
Since the linearization generally leads to an affine dynamical system, we first extend the Fundamental Lemma to affine systems in Section~\ref{sec:willems_affine}.
Next, we describe the problem setup including the required assumptions in Section~\ref{sec:prelim}.
In Section~\ref{sec:NL_MPC}, we present our data-driven MPC scheme for nonlinear systems and we prove practical exponential stability of the closed loop.
The proposed approach is applied to a nonlinear numerical example in Section~\ref{sec:example} and the paper is concluded in Section~\ref{sec:conclusion}.

\subsubsection*{Notation}
We denote the set of nonnegative integers by $\mathbb{I}_{\geq0}$, the set of integers in the interval $[a,b]$ by $\mathbb{I}_{[a,b]}$, and the nonnegative real numbers by $\mathbb{R}_{\geq0}$.
Moreover, $\lVert \cdot\rVert_2$ denotes the $2$-norm of a vector, or the induced $2$-norm if the argument is a matrix.
For a matrix $P=P^\top$, we denote by $\lambda_{\min}(P)$ ($\lambda_{\max}(P)$) the minimum (maximum) eigenvalue of $P$, we write $P\succ0$ if $P$ is positive definite, and we define $\lVert x\rVert_P^2\coloneqq x^\top P x$ for some vector $x$.
For matrices $P_1=P_1^\top$, $P_2=P_2^\top$, we define $\lambda_{\min}(P_1,P_2)\coloneqq\min\{\lambda_{\min}(P_1),\lambda_{\min}(P_2)\}$, and similarly for $\lambda_{\max}(P_1,P_2)$.
Further, $A^\dagger$ denotes the Moore-Penrose inverse of a matrix $A$ and $\otimes$ denotes the Kronecker product.
The interior of a set $X$ is denoted by $\mathrm{int}(X)$.
We define $\mathcal{K}_{\infty}$ as the class of functions $\alpha:\mathbb{R}_{\geq0}\to\mathbb{R}_{\geq0}$ which are continuous, strictly increasing, unbounded, and satisfy $\alpha(0)=0$.
For a sequence $\{x_k\}_{k=0}^{N-1}$, we define the Hankel matrix
\begin{align*}
H_L&(x)\coloneqq\begin{bmatrix}x_0 & x_1& \dots & x_{N-L}\\
x_1 & x_2 & \dots & x_{N-L+1}\\
\vdots & \vdots & \ddots & \vdots\\
x_{L-1} & x_{L} & \dots & x_{N-1}
\end{bmatrix},
\end{align*}
we denote a stacked window by $x_{[a,b]}\coloneqq\begin{bmatrix}x_a^\top&\dots&x_b^\top\end{bmatrix}$, and we write $x\coloneqq x_{[0,N-1]}$.
Finally, we write $\mathbbm{1}_{n}$ for an $n$-dimensional column vector with all entries equal to $1$.
Throughout this paper, we use the inequalities
\begin{align}\label{eq:ab_ineq2}
\lVert a+b\rVert_P^2&\leq2\lVert a\rVert_P^2+2\lVert b\rVert_P^2,\\\label{eq:ab_ineq}
\lVert a\rVert_P^2-\lVert b\rVert_P^2&\leq\lVert a-b\rVert_P^2+2\lVert a-b\rVert_P\lVert b\rVert_P,
\end{align}
which hold for any vectors $a$, $b$ and matrix $P=P^\top\succ0$.

\section{Fundamental Lemma for affine systems}\label{sec:willems_affine}
In this section, we provide a data-driven parametrization of unknown systems with affine dynamics, i.e.,
\begin{align}\label{eq:sys_affine}
x_{k+1}&=Ax_k+Bu_k+e,\\\nonumber
y_k&=Cx_k+Du_k+r
\end{align}
with state $x_k\in\mathbb{R}^n$, input $u_k\in\mathbb{R}^m$, and output $y_k\in\mathbb{R}^p$, all at time $k\in\mathbb{I}_{\geq0}$.
We assume that the matrices $A$, $B$, $C$, $D$ and the offsets $e$, $r$ are unknown, but one input-output trajectory $\{u_k^\rmd,y_k^\rmd\}_{k=0}^{N-1}$ of~\eqref{eq:sys_affine} is available.
The Fundamental Lemma~\cite{willems2005note} shows that, if $e=0$, $r=0$ (i.e., the system is linear) and certain persistence of excitation and controllability conditions hold, then a sequence $\{u_k,y_k\}_{k=0}^{L-1}$ is a trajectory of~\eqref{eq:sys_affine} if and only if there exists a vector $\alpha\in\mathbb{R}^{N-L+1}$ such that
\begin{align}
\begin{bmatrix}H_L(u^\rmd)\\H_L(y^\rmd)\end{bmatrix}\alpha=
\begin{bmatrix}u\\y\end{bmatrix}.
\end{align}
In the following, we provide an extension of this result to the class of affine systems~\eqref{eq:sys_affine}.
We note that such an extension is not trivial since, without knowledge of the vectors $e$ and $r$, they cannot be set to zero without loss of generality.
We propose the following definition of persistence of excitation, where we write $\{x_k^\rmd\}_{k=0}^{N-1}$ for the state sequence corresponding to the available input-output data $\{u_k^\rmd,y_k^\rmd\}_{k=0}^{N-1}$.
\begin{definition}\label{def:pe}
We say that the data $\{u_k^\rmd\}_{k=0}^{N-1}$, $\{x_k^\rmd\}_{k=0}^{N-L}$ are persistently exciting of order $L$ if
\begin{align}\label{eq:def_pe}
\text{rank}\left(\begin{bmatrix}H_L(u^\rmd)\\H_1(x^\rmd_{[0,N-L]})\\\mathbbm{1}_{N-L+1}^\top\end{bmatrix}\right)
=mL+n+1.
\end{align}
\end{definition}
The persistence of excitation condition used for the Fundamental Lemma in~\cite{willems2005note} requires that $H_{L+n}(u^\rmd)$ has full rank which in turn, assuming controllability, implies that
\begin{align}\label{eq:pe_original}
\text{rank}\left(\begin{bmatrix}H_L(u^\rmd)\\H_1(x^\rmd_{[0,N-L]})\end{bmatrix}\right)=mL+n.
\end{align}
In the present paper, we require the stronger condition in Definition~\ref{def:pe} to account for the affine system dynamics~\eqref{eq:sys_affine} with $e\neq0$, $r\neq0$.
In the more recent paper~\cite{martinelli2022data}, it is shown that Condition~\eqref{eq:def_pe} can be enforced if $(A,B)$ is controllable by choosing the input data sufficiently rich in the sense that $H_{L+n+1}(u^\rmd)$ has full row rank, i.e., $\mathrm{rank}(H_{L+n+1}(u^\rmd))=m(L+n+1)$.
%
\begin{theorem}\label{thm:willems_affine}
Suppose the data $\{u_k^\rmd\}_{k=0}^{N-1}$, $\{x_k^\rmd\}_{k=0}^{N-L}$ are persistently exciting of order $L$.
Then, $\{u_k,y_k\}_{k=0}^{L-1}$ is a trajectory of~\eqref{eq:sys_affine} if and only if there exists $\alpha\in\mathbb{R}^{N-L+1}$ such that
\begin{align}\label{eq:thm_willems_affine}
\sum_{i=0}^{N-L}\alpha_i&=1,\quad\begin{bmatrix}H_L(u^\rmd)\\H_L(y^\rmd)\end{bmatrix}\alpha=\begin{bmatrix}u\\y\end{bmatrix}.
\end{align}
\end{theorem}
The proof of Theorem~\ref{thm:willems_affine} 
is provided 
in Appendix~\ref{sec:app_willems_affine_proof}.
Theorem~\ref{thm:willems_affine} extends the Fundamental Lemma~\cite{willems2005note} to systems with affine dynamics.
The key difference to the linear case is the condition $\sum_{i=0}^{N-L}\alpha_i=1$ in~\eqref{eq:thm_willems_affine}, i.e., the vector $\alpha$ sums up to one.
Intuitively, this implies that the offsets $e$ and $r$ in~\eqref{eq:sys_affine} are carried through from the data $(u^\rmd,y^\rmd)$ to the new trajectory $(u,y)$ in~\eqref{eq:thm_willems_affine}.
We note that a condition similar to $\sum_{i=0}^{N-L}\alpha_i=1$ also appears in~\cite{salvador2019data}, albeit in a different problem setting with the objective of offset-free data-driven control.	
Further, instead of considering this additional condition, it is also possible to utilize the fact that the trajectory $\{u_{k+1}-u_k,y_{k+1}-y_k\}_{k=0}^{N-2}$ corresponds to an LTI system, compare, e.g.,~\cite{bianchin2021data}.

\section{Nonlinear dynamics and linearization}\label{sec:prelim}
In Section~\ref{subsec:prelim_setting}, we present the problem setup, followed by the assumptions required for our theoretical guarantees in Section~\ref{subsec:prelim_assumptions}.
Further, Section~\ref{subsec:prelim_bound} contains a technical result bounding the influence of the nonlinearity on the data.

\subsection{Problem setup}\label{subsec:prelim_setting}

We consider unknown \emph{nonlinear} systems of the form
\begin{align}\label{eq:sys_NL}
x_{k+1}&=f(x_k,u_k)=f_0(x_k)+Bu_k,\\\nonumber
y_k&=h(x_k,u_k)=h_0(x_k)+Du_k
\end{align}
with state $x_k\in\mathbb{R}^n$, input $u_k\in\mathbb{R}^m$, and output $y_k\in\mathbb{R}^p$, all at time $k\in\mathbb{I}_{\geq0}$, and $f_0:\mathbb{R}^n\to\mathbb{R}^n$, $B\in\mathbb{R}^{n\times m}$, $h_0:\mathbb{R}^{n}\to\mathbb{R}^p$, $D\in\mathbb{R}^{p\times m}$.
Note that we assume control-affine system dynamics, which can be enforced via an input transformation, see~\cite[Section II.A]{berberich2021linearpart1} for details.
System~\eqref{eq:sys_NL} is subject to pointwise-in-time input constraints $u_t\in\mathbb{U}$ for $t\in\mathbb{I}_{\geq0}$ with some convex, compact polytope $\mathbb{U}$.
For some convex polytope $\mathbb{U}^{\rms}\subseteq\mathrm{int}\left(\mathbb{U}\right)$, which is required for a local controllability argument, we define the steady-state manifold and its projection on the output by
\begin{align*}
\mathcal{Z}^\rms&\coloneqq\{(x^\rms,u^\rms)\in\mathbb{R}^n\times\mathbb{U}^{\rms}\mid x^\rms=f(x^\rms,u^\rms)\},\\
\mathcal{Z}^\rms_{\rmy}&\coloneqq\{y^\rms\in\mathbb{R}^{p}\mid y^\rms=h(x^\rms,u^\rms), (x^\rms,u^\rms)\in\mathcal{Z}^\rms\}.
\end{align*}
Further, the projection of $\calZ^\rms$ on the state is denoted by $\calZ_\rmx^\rms$.
In this paper, we propose a data-driven MPC scheme to track a desired setpoint reference $y^\rmr\in\mathbb{R}^p$ based only on input-output data of~\eqref{eq:sys_NL}, without explicit knowledge of the vector fields $f$ and $h$.
More precisely, we assume that, at time $t$, we only have access to the last $N\in\mathbb{I}_{\geq0}$ input-output measurements $\{u_k,y_k\}_{k=t-N}^{t-1}$ to compute the next control input.
While we assume noise-free data for our main theoretical results, it is straightforward to extend these results to noisy data, compare Remark~\ref{rk:noise}.
Similar to other data-driven MPC approaches based on~\cite{willems2005note}, the scheme does not involve state measurements and therefore, the analysis relies on the \emph{extended} state vector
\begin{align}\label{eq:xi_def_new}
\xi_t\coloneqq\begin{bmatrix}u_{[t-n,t-1]}\\y_{[t-n,t-1]}\end{bmatrix}
\end{align}
for $t\geq n$.
In general, $y^\rmr\notin\mathcal{Z}^\rms_{\rmy}$ such that our scheme will guarantee stability of the optimal reachable equilibrium $y^{\rms\rmr}$, which is the minimizer of
\begin{align}\label{eq:opt_reach_equil_NL}
J_{\rmeq}^*\coloneqq\min_{y^\rms\in\mathcal{Z}^\rms_{\rmy}}\lVert y^\rms-y^\rmr\rVert_S^2
\end{align}
with some $S\succ0$.
Assumptions made later will imply that this minimizer and the corresponding input-state pair $(x^{\rms\rmr},u^{\rms\rmr})$ are unique, and we denote by $\xi^{\rms\rmr}$ the corresponding extended state.
We assume that all vector fields in~\eqref{eq:sys_NL} are continuously differentiable and we define, for a linearization point $\tilde{x}\in\mathbb{R}^n$,
\begin{align}\label{eq:linearization}
A_{\tilde{x}}&\coloneqq\frac{\partial f_0}{\partial x}\Big\rvert_{\tilde{x}},\>\>
e_{\tilde{x}}\coloneqq f_0(\tilde{x})-A_{\tilde{x}}\tilde{x},\\\nonumber
C_{\tilde{x}}&\coloneqq\frac{\partial h_0}{\partial x}\Big\rvert_{\tilde{x}},\>\>
r_{\tilde{x}}\coloneqq h_0(\tilde{x})-C_{\tilde{x}}\tilde{x}.
\end{align}
Moreover, we define the affine dynamics resulting from the linearization of~\eqref{eq:sys_NL} at $(x,u)=(\tilde{x},0)$ by $f_{\tilde{x}}(x,u)\coloneqq A_{\tilde{x}}x+Bu+e_{\tilde{x}}$ and $h_{\tilde{x}}(x,u)\coloneqq C_{\tilde{x}}x+Du+r_{\tilde{x}}$.
Let now $\mathcal{D}=\{u_k',y_k'\}_{k=0}^{N-1}$ be a trajectory of the \emph{linearized} dynamics at some $\tilde{x}\in\mathbb{R}^n$ (i.e., a trajectory of~\eqref{eq:sys_NL}, replacing $f$ and $h$ by $f_{\tilde{x}}$ and $h_{\tilde{x}}$, respectively), whose input-state component is persistently exciting of order $L+n+1$ in the sense of Definition~\ref{def:pe}.
For our theoretical analysis, we define the optimal steady-state problem for the linearized dynamics by
\begin{align}\label{eq:opt_reach_equil_linearized}
J_{\mathrm{eq},\rmlin}^*(\tilde{x})&\coloneqq\min_{u^{\rms},y^{\rms},\alpha^{\rms}}\lVert y^\rms-y^\rmr\rVert_S^2\\\nonumber
\text{s.t.}\>\>&\begin{bmatrix}H_{L+n+1}(u')\\H_{L+n+1}(y')\\\mathbbm{1}_{N-L-n}^\top\end{bmatrix}\alpha^{\rms}=\begin{bmatrix}\bbone_{L+n+1}\otimes u^{\rms}\\\bbone_{L+n+1}\otimes y^{\rms}\\1\end{bmatrix},\>u^{\rms}\in\mathbb{U}^{\rms}.
\end{align}
We write $u^{\rms\rmr}_{\rmlin}(\tilde{x})$, $y^{\rms\rmr}_{\rmlin}(\tilde{x})$ for the optimal solution and $x^{\rms\rmr}_{\rmlin}(\tilde{x})$ ($\xi^{\rms\rmr}_{\rmlin}(\tilde{x})$) for the (extended) steady-state of the linearized dynamics, which are unique due to assumptions made in the following.
Further, $\alpha^{\rms\rmr}_{\rmlin}(\mathcal{D})$ denotes the optimal solution with minimum $2$-norm.
Note that, in contrast to the optimal input, state, and output, the optimal value of $\alpha^\rms$ depends not only on $\tilde{x}$ but also on the data set $\mathcal{D}$ in~\eqref{eq:opt_reach_equil_linearized}.
Throughout this paper, we assume that $\alpha^{\rms\rmr}_{\rmlin}(\mathcal{D})$ is uniformly bounded.\footnote{This is always fulfilled in closed loop if the data generated by the proposed MPC scheme are uniformly persistently exciting (cf. Assumption~\ref{ass:closed_loop_pe}), using compactness of the steady-state manifold (Assumption~\ref{ass:NL_sys}).
In case this is not guaranteed a priori, a uniform bound on $\alpha^{\rms\rmr}_{\rmlin}(\mathcal{D})$ can be ensured by using incremental data $\Delta u_k=u_{k+1}-u_k$, $\Delta y_k=y_{k+1}-y_k$ in the proposed data-driven MPC scheme and its analysis.
In this case, $(u^{\rms\rmr}_{\rmlin}(\tilde{x}),y^{\rms\rmr}_{\rmlin}(\tilde{x}))=(0,0)$ for any $\tilde{x}\in \bbr^n$ and thus, $\alpha^{\rms\rmr}_{\rmlin}(\mathcal{D})=0$, i.e., $\alpha^{\rms\rmr}_{\rmlin}(\mathcal{D})$ is uniformly bounded.}

\subsection{Assumptions}\label{subsec:prelim_assumptions}

Since our analysis relies on similar arguments as our companion paper~\cite{berberich2021linearpart1}, we require the following assumptions.
\begin{assumption}\label{ass:NL_sys}
System~\eqref{eq:sys_NL} satisfies~\cite[Assumptions 1-5]{berberich2021linearpart1}.
\end{assumption}
Through~\cite[Assumption 1]{berberich2021linearpart1}, we assume that all vector fields in~\eqref{eq:sys_NL} are twice continuously differentiable.
Moreover,~\cite[Assumptions 2 and 3]{berberich2021linearpart1} require that the linearized dynamics are controllable 
and satisfy 
a certain 
tracking condition
at any linearization point.
In~\cite[Assumption 4]{berberich2021linearpart1}, we assume that $I-A_{\tilde{x}}$ is non-singular for any $\tilde{x}\in\mathbb{R}^n$, and~\cite[Assumption 5]{berberich2021linearpart1} requires that the union of all steady-state manifolds of the linearized dynamics are compact.
We refer to~\cite{berberich2021linearpart1} for a more detailed discussion of these assumptions.
%
\begin{assumption}\label{ass:NL_obsv}
(Observability)
There exists a locally Lipschitz continuous map $T_L:\mathbb{R}^{n(m+p)}\to\mathbb{R}^n$ such that
\begin{align}\label{eq:NL_x_xi_Lipschitz}
x_t=T_\rmL(\xi_t).
\end{align}
Further, for any $x_t\in\mathbb{R}^n$, there exists an extended state $\xi_t'=\begin{bmatrix}u_{[t-n,t-1]}'^\top &y_{[t-n,t-1]}'^\top\end{bmatrix}^\top$ corresponding to $x_t$ in the dynamics linearized at $x_t$, i.e., there exists an affine map $T_{x_t}:\mathbb{R}^{(m+p)n}\to\mathbb{R}^n$ such that
\begin{align}\label{eq:NL_trafo_xi_x_lin}
x_t&=T_{x_t}(\xi_t').
\end{align}
\end{assumption}
Assumption~\ref{ass:NL_obsv} corresponds to (final state) observability (compare, e.g.,~\cite[Definition 4.29]{rawlings2020model}) of the linearized and nonlinear dynamics.
Analogous observability assumptions are also required in \emph{linear} data-driven MPC, compare, e.g.,~\cite{coulson2021distributionally,berberich2021guarantees}.
In the proof of Theorem~\ref{thm:NL_stability}, we provide a precise definition of the particular choice of $\xi_t'$ used for our theoretical analysis.
\begin{assumption}\label{ass:NL_manifold_convex}
For any compact set $\Xi$, there exist constants $c_{\mathrm{eq},1},c_{\mathrm{eq},2}>0$ such that, for any extended state $\hat{\xi}\in\Xi$ it holds that
\begin{align}\label{eq:ass_NL_manifold_convex}
c_{\mathrm{eq},1}\lVert\hat{\xi}-\xi^{\rms\rmr}_{\rmlin}(\hat{x})\rVert_2^2\leq\lVert\hat{\xi}-\xi^{\rms\rmr}\rVert_2^2\leq c_{\mathrm{eq},2}\lVert\hat{\xi}-\xi^{\rms\rmr}_{\rmlin}(\hat{x})\rVert_2^2,
\end{align}
where $\hat{x}=T_\rmL(\hat{\xi})$, compare~\eqref{eq:NL_x_xi_Lipschitz}.
\end{assumption}
Assumption~\ref{ass:NL_manifold_convex} is analogous to~\cite[Assumption 6]{berberich2021linearpart1}, where Inequality~\eqref{eq:ass_NL_manifold_convex} is assumed for the state $\hat{x}$ and the optimal reachable steady-states $x^{\rms\rmr}$, $x^{\rms\rmr}_{\rmlin}(\hat{x})$ (see~\cite{berberich2021linearpart1} for details).
In the data-driven framework, we require Assumption~\ref{ass:NL_manifold_convex} since the proposed data-driven MPC scheme only involves input-output values and hence, our theoretical analysis relies on the extended state $\xi$.
Using similar arguments as in~\cite{berberich2021linearpart1}, it can be shown that Assumption~\ref{ass:NL_manifold_convex} holds if i)~\cite[Assumptions 1, 2, 4, and 5]{berberich2021linearpart1} hold for the extended state-space system with state\footnote{It is straightforward to verify that, if~\cite[Assumptions 1, 2, 4, and 5]{berberich2021linearpart1} hold for the state $x$, then they also hold for the state $\xi$ with different constants.}  $\xi$, ii) the setpoint $y^\rmr$ is reachable, i.e., $y^{\rms\rmr}=y^\rmr$, and iii) $m=p$.

\subsection{Bounding the influence of the nonlinearity on the data }\label{subsec:prelim_bound}
Our goal is to control the unknown nonlinear system~\eqref{eq:sys_NL} via MPC using the last $N$ input-output measurements $\{u_k,y_k\}_{k=t-N}^{t-1}$ to predict future trajectories at time $t$.
By Theorem~\ref{thm:willems_affine}, data corresponding to the (affine) linearization of~\eqref{eq:sys_NL} provide an exact parametrization of all trajectories of the linearized dynamics.
However, we only have access to data of the \emph{nonlinear} dynamics.
In the following, we bound the difference between the (artificial) data of the linearization at $x_t$ and the actually available data of the nonlinear system.

Let $\{u_k,x_k,y_k\}_{k\in\mathbb{I}_{\geq0}}$ be an arbitrary trajectory of the nonlinear system~\eqref{eq:sys_NL}.
For some $t\geq N$, we write $\{x_k'(t)\}_{k=-N}^n$
and $\{y_k'(t)\}_{k=-N}^{n-1}$ for the state 
and output corresponding to the dynamics linearized at $x_t$ resulting from an application of the input $\{u_k\}_{k=t-N}^{t+n-1}$ with the initial state of the nonlinear system $x_{t-N}$ at time $t-N$, i.e., $x_{-N}'(t)=x_{t-N}$ and
\begin{align}\label{eq:data_lin}
x_{k+1}'(t)&=f_{x_t}(x_k'(t),u_{t+k}),\>\>
y_k'(t)=h_{x_t}(x_k'(t),u_{t+k})
\end{align}
for $k\in\mathbb{I}_{[-N,n-1]}$.
We denote this ``artificial'' input-output data trajectory by $\calD_t\coloneqq\{u_{t+k},y_k'(t)\}_{k=-N}^{-1}$ and we define
\begin{align}\label{eq:Hux_NL}
H_{ux,t}\coloneqq\begin{bmatrix}H_{L+n+1}(u_{[t-N,t-1]})\\H_1(x_{[-N,-L-n-1]}'(t))\\\mathbbm{1}_{N-L-n}^\top\end{bmatrix}.
\end{align}
The following result provides a bound on the difference between the (known) output of the nonlinear system $y$ and the (unknown) output of the affine dynamics $y'$.
\begin{lemma}\label{lem:pred_error_NL}
Let Assumption~\ref{ass:NL_sys} hold.
For any compact set $X\subset\mathbb{R}^n$, there exists $c_{\Delta}>0$ such that for any $t\geq N$, 
$k\in\mathbb{I}_{[-N,n-1]}$, and $x_k,x_t\in X$, $\Delta_{t,k}\coloneqq y_{t+k}-y_k'(t)$ satisfies
\begin{align}\label{eq:lem_pred_error_NL}
\lVert\Delta_{t,k}\rVert_2\leq c_{\Delta}\sum_{j=t-N}^{t+k}\lVert x_t-x_j\rVert_2^2.
\end{align}
\end{lemma}
\begin{proof}
Using $f(x,u)=f_x(x,u)$, $h(x,u)=h_x(x,u)$, we have
\begin{align}\label{eq:lem_pred_error_NL_proof_x}
x_{k+1}&=A_{x_k}x_k+Bu_k+e_{x_k}\\\nonumber
&=A_{x_t}x_k+Bu_k+e_{x_t}+\underbrace{(A_{x_k}-A_{x_t})x_k+e_{x_k}-e_{x_t}}_{\Delta_{x,k}\coloneqq},\\
\label{eq:lem_pred_error_NL_proof_y}
y_k&=C_{x_k}x_k+Du_k+r_{x_k}\\\nonumber
&=C_{x_t}x_k+Du_k+r_{x_t}+\underbrace{(C_{x_k}-C_{x_t})x_k+r_{x_k}-r_{x_t}}_{\Delta_{y,k}\coloneqq},
\end{align}
for any $k\in\mathbb{I}_{[t-N,t+n-1]}$.
By repeatedly applying~\eqref{eq:lem_pred_error_NL_proof_x}, we obtain, for $k\in\mathbb{I}_{[t-N+1,t+n-1]}$,
$
x_k=x_{k-t}'(t)+\Delta_{x,k-1}+\sum_{j=t-N}^{k-2}A_{x_{k-1}}\cdot\dots\cdots A_{x_{j+1}}\Delta_{x,j}.
$
In combination with~\eqref{eq:lem_pred_error_NL_proof_y}, this yields
\begin{align*}
y_{k}&=y_{k-t}'(t)+\Delta_{y,k}+C_{x_t}\Delta_{x,k-1}\\
&+C_{x_{t}}\sum_{j=t-N}^{k-2}A_{x_{k-1}}\cdot\ldots\cdot A_{x_{j+1}}\Delta_{x,j}= y_{k-t}'(t)+\Delta_{t,k-t}.
\end{align*}
This implies
\begin{align}\label{eq:lem_pred_error_NL_proof_Delta}
\Delta_{t,k}&=\Delta_{y,t+k}+C_{x_{t+k}}\Delta_{x,t+k-1}\\\nonumber
&\quad+C_{x_{t+k}}\sum_{j=t-N}^{t+k-2}A_{x_{t+k-1}}\cdot\ldots\cdot A_{x_{j+1}}\Delta_{x,j}
\end{align}
for $k\in\mathbb{I}_{[-N,n-1]}$.
Using smoothness of the dynamics in~\eqref{eq:sys_NL}, we can apply~\cite[Inequalities (7) and (8)]{berberich2021linearpart1} to derive
\begin{align*}
\lVert\Delta_{x,k}\rVert_{2}&=\lVert f_{x_k}(x_k,u_k)-f_{x_t}(x_k,u_k)\rVert_{2}\leq c_X\lVert x_t-x_k\rVert_2^2,\\
\lVert\Delta_{y,k}\rVert_{2}&=\lVert h_{x_k}(x_k,u_k)-h_{x_t}(x_k,u_k)\rVert_{2}\leq c_{Xh}\lVert x_t-x_k\rVert_2^2
\end{align*}
for $k\in\mathbb{I}_{[t-N,t+n-1]}$ and with suitable $c_X,c_{Xh}>0$.
Using these inequalities to bound $\Delta_{t,k}$ in~\eqref{eq:lem_pred_error_NL_proof_Delta} and exploiting that, on the compact set $X$, the Jacobians $A_{x}$ and $C_{x}$ are uniformly bounded, there exists $c_{\Delta}>0$ such that~\eqref{eq:lem_pred_error_NL} holds.
\end{proof}

Lemma~\ref{lem:pred_error_NL} shows that the difference between $y_{t+k}$ and $y_k'(t)$, i.e., the ``output measurement noise'' affecting the ideal data of the linearized dynamics,
is bounded by the squared distance of $x_t$ to the past states.
This means that, if the state trajectory does not move too rapidly, then the data collected from the nonlinear system~\eqref{eq:sys_NL} are close to those of the linearized dynamics and can therefore be used to (approximately) parametrize trajectories of the linearized dynamics via Theorem~\ref{thm:willems_affine}.

\section{Data-driven MPC for nonlinear systems}\label{sec:NL_MPC}

In this section, we present a data-driven MPC scheme to control unknown nonlinear systems with stability guarantees.
After introducing the MPC scheme in Section~\ref{subsec:NL_scheme}, we show a useful continuity property of the solution of the underlying optimal control problem (Section~\ref{subsec:NL_continuity}).
In Section~\ref{subsec:NL_guarantees}, we then prove practical exponential stability of the closed loop.

\subsection{MPC scheme}\label{subsec:NL_scheme}

\begin{subequations}\label{eq:DD_MPC_NL}
At time $t\geq N$, given past $N$ input-output measurements $\{u_k,y_k\}_{k=t-N}^{t-1}$ of the nonlinear system~\eqref{eq:sys_NL}, we define the following open-loop optimal control problem
\begin{align}\label{eq:DD_MPC_NL_cost}
\underset{\substack{\alpha(t),\sigma(t)\\u^{\rms}(t),y^{\rms}(t)}}{\min}&\sum_{k=-n}^{L}
\lVert\bar{u}_k(t)-u^{\rms}(t)\rVert_R^2+\lVert\bar{y}_k(t)-y^{\rms}(t)\rVert_Q^2\\\nonumber
+\lVert y^{\rms}&(t)-y^\rmr\rVert_S^2+\lambda_\alpha\lVert\alpha(t)-\alpha^{\rms\rmr}_{\rmlin}(\calD_t)\rVert_2^2+\lambda_\sigma\lVert\sigma(t)\rVert_{2}^2\\
\label{eq:DD_MPC_NL_hankel} \text{s.t.}\>\> &\>\begin{bmatrix}
\bar{u}(t)\\\bar{y}(t)+\sigma(t)\\1\end{bmatrix}=\begin{bmatrix}H_{L+n+1}\left(u_{[t-N,t-1]}\right)\\H_{L+n+1}\left(y_{[t-N,t-1]}\right)\\\mathbbm{1}_{N-L-n}^\top\end{bmatrix}\alpha(t),\\\label{eq:DD_MPC_NL_init}
&\>\begin{bmatrix}\bar{u}_{[-n,-1]}(t)\\\bar{y}_{[-n,-1]}(t)\end{bmatrix}=\begin{bmatrix}u_{[t-n,t-1]}\\y_{[t-n,t-1]}\end{bmatrix},
\\\label{eq:DD_MPC_NL_TEC}
&\>\begin{bmatrix}\bar{u}_{[L-n,L]}(t)\\\bar{y}_{[L-n,L]}(t)\end{bmatrix}=
\begin{bmatrix}\bbone_{n+1}\otimes u^{\rms}(t)\\
\bbone_{n+1}\otimes y^{\rms}(t)\end{bmatrix},\\\label{eq:DD_MPC_NL_constraints}
&\>\bar{u}_k(t)\in\mathbb{U},\>k\in\mathbb{I}_{[0,L]},\>u^{\rms}(t)\in\mathbb{U}^{\rms}.
\end{align}
\end{subequations}
Here, $\bar{u}(t)\in\mathbb{R}^{m(L+n+1)}$ and $\bar{y}(t)\in\mathbb{R}^{p(L+n+1)}$ denote the input and output trajectory predicted at time $t$ with elements $\bar{u}_k(t)\in\mathbb{R}^{m}$ and $\bar{y}_k(t)\in\mathbb{R}^p$ at time $k$, respectively.
The cost~\eqref{eq:DD_MPC_NL_cost} contains a tracking term with weights $Q\succ0$, $R\succ0$ w.r.t.\ an artificial setpoint $(u^\rms(t),y^\rms(t))$ which is optimized online and whose distance to the setpoint $y^\rmr$ is penalized in the cost by the weight $S\succ0$.
The fact that the cost~\eqref{eq:DD_MPC_NL_cost} is summed over $k\in\mathbb{I}_{[-n,L]}$ (instead of $k\in\mathbb{I}_{[0,L]}$ as, e.g., in~\cite{berberich2021guarantees,berberich2020tracking}) simplifies the theoretical analysis and represents a weighting of the initial (extended) state, analogously to model-based tracking MPC~\cite{berberich2021linearpart1,limon2018nonlinear,koehler2020nonlinear}.

Similar to existing data-driven MPC schemes for linear systems, e.g.,~\cite{coulson2019deepc,coulson2021distributionally,berberich2021guarantees,bongard2021robust}, Problem~\eqref{eq:DD_MPC_NL} uses a prediction model based on Hankel matrices~\eqref{eq:DD_MPC_NL_hankel}.
In contrast to these works, the data used for prediction are updated at any time step using the last $N$ measurements of the \emph{nonlinear} system.
Hence, the prediction model of Problem~\eqref{eq:DD_MPC_NL} can be seen as an approximation of the affine dynamics resulting from the linearization at $x_t$, cf. Section~\ref{subsec:prelim_bound}, which in turn provides a local approximation of the nonlinear dynamics~\eqref{eq:sys_NL}.
In order to account for the model mismatch due to the nonlinearity, the slack variable $\sigma(t)$ is introduced and the cost~\eqref{eq:DD_MPC_NL_cost} additionally contains regularization terms of $\alpha(t)$ and $\sigma(t)$ with parameters $\lambda_{\alpha},\lambda_{\sigma}>0$.
We note that analogous ingredients are employed to handle noise in linear data-driven MPC~\cite{coulson2019deepc,coulson2021distributionally,berberich2021guarantees,bongard2021robust}.

In~\eqref{eq:DD_MPC_NL_init}, the first $n$ components of the predictions are set to the past $n$ input-output measurements in order to specify initial conditions, compare~\cite{markovsky2008data}.
We note that $n$ can be replaced by any upper bound on the system order (or, more specifically, on the system's lag), i.e., the application of the proposed MPC does \emph{not} require accurate knowledge of the system order.
Moreover,~\eqref{eq:DD_MPC_NL_TEC} represents the terminal equality constraint w.r.t.\ the artificial equilibrium $(u^\rms(t),y^\rms(t))$.
It is defined over $n+1$ steps since this ensures that $(u^\rms(t),y^\rms(t))$ is an (approximate) equilibrium of the dynamics linearized at $x_t$, compare~\cite[Definition 3]{berberich2021guarantees}.
In order to parametrize (approximate) trajectories of the \emph{affine} dynamics corresponding to the linearization at $x_t$, the last line of~\eqref{eq:DD_MPC_NL_hankel} implies that $\alpha(t)$ sums up to $1$, compare Theorem~\ref{thm:willems_affine}.
Further, the scheme contains constraints on the input equilibrium and on the input trajectory in~\eqref{eq:DD_MPC_NL_constraints}.
Note that~\eqref{eq:DD_MPC_NL} is a strictly convex QP which can be solved efficiently.

\begin{remark}\label{rk:alpha_sr}
Inspired by~\cite{elokda2021quadcopters}, the regularization of $\alpha(t)$ is not w.r.t.\ zero but depends on $\alpha^{\rms\rmr}_{\rmlin}(\calD_t)$ since we want to track the generally non-zero equilibrium $(u^{\rms\rmr},y^{\rms\rmr})$.
Note that $\alpha^{\rms\rmr}_{\rmlin}(\calD_t)$ can be (approximately) computed as the least-squares solution of~\eqref{eq:opt_reach_equil_linearized} by inserting the past $N$ input-output measurements $\{u_k,y_k\}_{k=t-N}^{t-1}$ of the nonlinear system~\eqref{eq:sys_NL}.
Since these measurements are not a trajectory of the linearized dynamics, cf. Section~\ref{subsec:prelim_bound}, it is beneficial for practical purposes to solve the following robust version of~\eqref{eq:opt_reach_equil_linearized} with parameters $\lambda_{\alpha}^\rms,\lambda_{\sigma}^\rms>0$:
\begin{align}\label{eq:opt_reach_equil_linearized_robust}
&\min_{u^{\rms}\in\mathbb{U}^{\rms},y^{\rms},\alpha^{\rms},\sigma^{\rms}}\lVert y^\rms-y^\rmr\rVert_S^2+\lambda_{\alpha}^\rms\lVert\alpha^\rms\rVert_2^2+\lambda_{\sigma}^\rms\lVert\sigma^\rms\rVert_2^2\\\nonumber
\text{s.t.}\>\>&\begin{bmatrix}H_{L+n+1}(u_{[t-N,t-1]})\\H_{L+n+1}(y_{[t-N,t-1]})\\\mathbbm{1}_{N-L-n}^\top\end{bmatrix}\alpha^{\rms}=\begin{bmatrix}\bbone_{L+n+1}\otimes u^{\rms}\\\bbone_{L+n+1}\otimes y^{\rms}+\sigma^\rms\\1\end{bmatrix}.
\end{align}
If an approximation $\alpha^\rms\text{$'$}$ of $\alpha^{\rms\rmr}_{\rmlin}(\calD_t)$ with $\lVert\alpha^\rms\text{$'$}-\alpha^{\rms\rmr}_{\rmlin}(\calD_t)\rVert_2\leq c$ for some $c>0$ is known, our theoretical results remain true, albeit with more conservative bounds which deteriorate for increasing values of $c$.
Since $\alpha^{\rms\rmr}_{\rmlin}(\calD_t)$ is uniformly bounded, this is the case for any uniformly bounded $\alpha^\rms\text{$'$}$ (e.g., $\alpha^\rms\text{$'$}=0$).
\end{remark}


We denote the optimal solution of~\eqref{eq:DD_MPC_NL} at time $t$ by $\bar{u}^*(t)$, $\bar{y}^*(t)$, $\alpha^*(t)$, $\sigma^*(t)$, $u^{\rms*}(t)$, $y^{\rms*}(t)$, and the closed-loop input, state, and output at time $t$ by $u_t$, $x_t$, and $y_t$, respectively.
Further, we write $J_L^*(\xi_t)$ for the corresponding optimal cost.
Problem~\eqref{eq:DD_MPC_NL} is applied in a multi-step fashion, see Algorithm~\ref{alg:MPC_n_step_NL}.

\begin{algorithm}
\begin{Algorithm}\label{alg:MPC_n_step_NL}
\normalfont{\textbf{Nonlinear Data-Driven MPC Scheme}}\\
\textbf{Offline:}
Choose upper bound on system order $n$, prediction horizon $L$, cost matrices $Q,R,S\succ0$, regularization parameters $\lambda_{\alpha},\lambda_{\sigma}>0$, constraint sets $\mathbb{U},\mathbb{U}^{\rms}$, setpoint $y^\rmr$, and generate data $\{u_k,y_k\}_{k=0}^{N-1}$.\\
\textbf{Online:}
\begin{enumerate}
\item[1)] At time $t\geq N$, compute $\alpha^{\rms\rmr}_{\rmlin}(\calD_t)$ by solving~\eqref{eq:opt_reach_equil_linearized} or its approximation~\eqref{eq:opt_reach_equil_linearized_robust}.
\item[2)] Solve~\eqref{eq:DD_MPC_NL} and apply the first $n$ input components $u_{t+k}=\bar{u}_k^*(t)$, $k\in\mathbb{I}_{[0,n-1]}$.
\item[3)] Set $t=t+n$ and go back to 1).
\end{enumerate}
\end{Algorithm}
\end{algorithm}
Considering a multi-step MPC scheme instead of a standard (one-step) MPC scheme simplifies the theoretical analysis with terminal equality constraints due to a local controllability argument in the proof, similar to the model-based MPC in our companion paper~\cite{berberich2021linearpart1}.

Note that Algorithm~\ref{alg:MPC_n_step_NL} allows to control unknown nonlinear systems based only on measured data and without any model knowledge except for a (potentially rough) upper bound on the system order.
Moreover, it only requires solving one (or two, if $\alpha^{\rms\rmr}_{\rmlin}(\calD_t)$ is computed online, see Step 1)) strictly convex QPs.
In the remainder of this section, we prove that, under suitable assumptions, the closed loop under Algorithm~\ref{alg:MPC_n_step_NL} is practically exponentially stable.
Our analysis builds on the MPC approach based on linearized models in our companion paper~\cite{berberich2021linearpart1}.
Similar to~\cite{berberich2021linearpart1}, the key idea is that, if $\lambda_{\max}(S)$ is sufficiently small and the initial state is sufficiently close to $\calZ_\rmx^\rms$, then the data-driven prediction model in~\eqref{eq:DD_MPC_NL_hankel} provides a good approximation of the nonlinear dynamics in closed loop, thus allowing to prove closed-loop stability.

\subsection{Continuity of the optimal input and cost}\label{subsec:NL_continuity}
In this section, we present a key technical result for our closed-loop analysis bounding the distance of the optimal input/cost of Problem~\eqref{eq:DD_MPC_NL} to the corresponding optimal input/cost with ``ideal'' data of the linearized dynamics.
\begin{subequations}\label{eq:DD_MPC_NL_nominal}
Given a vector $\tilde{\sigma}=\begin{bmatrix}\tilde{\sigma}_{\mathrm{init}}\\\tilde{\sigma}_{\mathrm{dyn}}\end{bmatrix}\in\mathbb{R}^{p(L+2n+1)+mn}$, we define the optimization problem
\begin{align}\label{eq:DD_MPC_NL_nominal_cost}
\underset{\substack{\alpha(t)\\u^{\rms}(t),y^{\rms}(t)}}{\min}&\sum_{k=-n}^{L}
\lVert\bar{u}_k(t)-u^{\rms}(t)\rVert_R^2+\lVert\bar{y}_k(t)-y^{\rms}(t)\rVert_Q^2\\\nonumber
&+\lVert y^\rms(t)-y^\rmr\rVert_S^2+\lambda_{\alpha}\lVert\alpha(t)-\alpha^{\rms\rmr}_{\rmlin}(\calD_t)\rVert_2^2\\
\label{eq:DD_MPC_NL_nominal_hankel} \text{s.t.}\>\> &\>\begin{bmatrix}
\bar{u}(t)\\\bar{y}(t)+\tilde{\sigma}_{\mathrm{dyn}}\\1\end{bmatrix}=\begin{bmatrix}H_{L+n+1}\left(u_{[t-N,t-1]}\right)\\H_{L+n+1}\left(y'(t)\right)\\\mathbbm{1}_{N-L-n}^\top\end{bmatrix}\alpha(t),\\\label{eq:DD_MPC_NL_nominal_init}
&\>\begin{bmatrix}\bar{u}_{[-n,-1]}(t)\\\bar{y}_{[-n,-1]}(t)\end{bmatrix}=
\begin{bmatrix}u_{[t-n,t-1]}'\\y_{[t-n,t-1]}'\end{bmatrix}+\tilde{\sigma}_{\mathrm{init}},
\\\label{eq:DD_MPC_NL_nominal_TEC}
&\>\begin{bmatrix}\bar{u}_{[L-n,L]}(t)\\\bar{y}_{[L-n,L]}(t)\end{bmatrix}=
\begin{bmatrix}\bbone_{n+1}\otimes u^{\rms}(t)\\
\bbone_{n+1}\otimes y^{\rms}(t)\end{bmatrix},\\\label{eq:DD_MPC_NL_nominal_constraints}
&\>\bar{u}_k(t)\in\mathbb{U},\>k\in\mathbb{I}_{[0,L]},\>u^{\rms}(t)\in\mathbb{U}^{\rms}.
\end{align}
\end{subequations}
We denote the optimal solution of~\eqref{eq:DD_MPC_NL_nominal} with\footnote{We require the flexibility of choosing $\tilde{\sigma}\neq0$ for a technical argument in the proof of Proposition~\ref{prop:continuity} in Appendix~\ref{sec:app_prop_continuity_proof}.} $\tilde{\sigma}=0$ by $\check{\alpha}^*(t)$, $\check{u}^{\rms*}(t)$, $\check{y}^{\rms*}(t)$, $\check{u}^*(t)$, $\check{y}^*(t)$.
Moreover, we write $\check{J}_L^*(\xi_t',\mathcal{D}_t)$ for the optimal cost to emphasize its dependence on the initial condition $\xi_t'$ as well as the data set $\mathcal{D}_t=\{u_{t+k},y_k'(t)\}_{k=-N}^{-1}$.
For $\tilde{\sigma}=0$, the constraints of~\eqref{eq:DD_MPC_NL_nominal} correspond to those of Problem~\eqref{eq:DD_MPC_NL} with $\sigma(t)=0$ and replacing the data $\{u_k,y_k\}_{k=t-N}^{t-1}$ and initial condition $\xi_t$ of the nonlinear system by that of the dynamics linearized at $x_t$, i.e., $\{u_{t+k},y_k'(t)\}_{k=-N}^{-1}$ and $\xi_t'$.
Thus, according to Theorem~\ref{thm:willems_affine}, any $\bar{u}(t)$, $\bar{y}(t)$ satisfying the constraints of Problem~\eqref{eq:DD_MPC_NL_nominal} with $\tilde{\sigma}=0$ are a trajectory of the dynamics linearized at $x_t$.
Note that the initial conditions $\{u_{t+k}',y_{t+k}'\}_{k=-n}^{-1}$ in~\eqref{eq:DD_MPC_NL_nominal_init} are in general different from $\{u_{t+k},y'_k(t)\}_{k=-n}^{-1}$ in~\eqref{eq:DD_MPC_NL_nominal_hankel}, although both correspond to the dynamics linearized at $x_t$ and both are close to the input-output trajectory $\{u_{t+k},y_{t+k}\}_{k=-n}^{-1}$ of the nonlinear system.
\begin{assumption}\label{ass:LICQ}
(LICQ)
For any $t\geq N$, Problem~\eqref{eq:DD_MPC_NL_nominal} satisfies a linear independence constraint qualification (LICQ), i.e., the row entries of the equality and active inequality constraints are linearly independent.
\end{assumption}
Assumption~\ref{ass:LICQ} is required for a technical step in the following result.
Such an LICQ assumption is common in linear MPC, compare~\cite{bemporad2002explicit}, and we conjecture that it may be possible to relax it at the price of a more involved analysis.
We now derive a bound on the difference between the \emph{nominal} optimal cost $\check{J}_L^*(\xi_t',\mathcal{D}_t)$ and input $\check{u}^*(t)$ corresponding to Problem~\eqref{eq:DD_MPC_NL_nominal} with $\tilde{\sigma}=0$ and the \emph{perturbed} optimal cost $J_L^*(\xi_t)$ and input $\bar{u}^*(t)$ corresponding to Problem~\eqref{eq:DD_MPC_NL}.
\begin{proposition}\label{prop:continuity}
Let Assumptions~\ref{ass:NL_sys} and~\ref{ass:LICQ} hold and suppose the data $\{u_{t+k}\}_{k=-N}^{-1}$, $\{x_k'(t)\}_{k=-N}^{-L}$ (cf.~\eqref{eq:data_lin}) are persistently exciting of order $L+n+1$, compare Definition~\ref{def:pe}.
Moreover, for some $\bar{\varepsilon}>0$, consider 
\begin{align}\label{eq:prop_continuity_noise_bound}
&\lVert\xi_t-\xi_t'\rVert_{2}\leq\bar{\varepsilon},\>\>\lVert y_{t+k}-y_k'(t)\rVert_{2}\leq\bar{\varepsilon}\>\>\>\forall k\in\mathbb{I}_{[-N,-1]},\\\label{eq:prop_continuity_lambda_def}
&\lambda_{\alpha}=\bar{\lambda}_{\alpha}\bar{\varepsilon}^{\beta_{\alpha}},\>\>\lambda_{\sigma}=\frac{\bar{\lambda}_{\sigma}}{\bar{\varepsilon}^{\beta_{\sigma}}}
\end{align}
for some $\bar{\lambda}_{\alpha},\bar{\lambda}_{\sigma},\beta_{\alpha},\beta_{\sigma}>0$ with $\beta_{\alpha}+2\beta_{\sigma}<2$, where $\xi_t'$ satisfies~\eqref{eq:NL_trafo_xi_x_lin} and $y'(t)$ is the output of the dynamics linearized at $x_t$ with input $u_{[t-N,t-1]}$ and initial condition $x_{t-N}$, compare~\eqref{eq:data_lin}.
\begin{itemize}
\item[(i)] There exist $\bar{\varepsilon}_{\max},c_{\rmJ,\rma},c_{\rmJ,\rmb}>0$ such that, if $\bar{\varepsilon}\leq\bar{\varepsilon}_{\max}$, \eqref{eq:prop_continuity_noise_bound}, and~\eqref{eq:prop_continuity_lambda_def} hold, then
\begin{align}\label{eq:prop_continuity_cost}
J_L^*(\xi_t)&\leq\left(1+c_{\rmJ,\rma}\bar{\varepsilon}^{\beta_{\sigma}}\right)\check{J}_L^*(\xi_t',\mathcal{D}_t)\\\nonumber
&\quad+c_{\rmJ,\rmb}\bar{\varepsilon}^{2-\beta_\sigma}\left(1+\lVert H_{ux,t}^\dagger\rVert_2^2\bar{\varepsilon}^{\beta_{\alpha}}\right)
\end{align}
with $H_{ux,t}$ as in~\eqref{eq:Hux_NL}.
\item[(ii)] For any $\bar{J}>0$, there exist $\bar{\varepsilon}_{\max}>0$, $\beta_{\rmu}\in\mathcal{K}_{\infty}$ such that, if $\check{J}_L^*(\xi_t',\mathcal{D}_t)\leq\bar{J}$, $\bar{\varepsilon}\leq\bar{\varepsilon}_{\max}$,~\eqref{eq:prop_continuity_noise_bound}, and~\eqref{eq:prop_continuity_lambda_def} hold, then
\begin{align}\label{eq:prop_continuity}
\lVert \bar{u}^*(t)-\check{u}^*(t)\rVert_2\leq\beta_{\rmu}(\bar{\varepsilon}).
\end{align}
\end{itemize}
\end{proposition}

The proof of Proposition~\ref{prop:continuity} is provided in Appendix~\ref{sec:app_prop_continuity_proof}.
Proposition~\ref{prop:continuity} bounds the difference between the optimal input/cost for the robust MPC problem~\eqref{eq:DD_MPC_NL} with data of the nonlinear system and the optimal input/cost for the nominal MPC problem~\eqref{eq:DD_MPC_NL_nominal} with data of the linearized dynamics, i.e., with $\tilde{\sigma}=0$.
In Part (i) of the proof, we bound the optimal cost $J_L^*(\xi_t)$ in terms of $\check{J}_L^*(\xi_t',\calD_t)$ using a simple candidate solution, i.e., we show that the robust MPC problem~\eqref{eq:DD_MPC_NL} is feasible whenever the nominal problem~\eqref{eq:DD_MPC_NL_nominal} is feasible.
Part (ii), i.e., the proof of~\eqref{eq:prop_continuity}, relies on a combination of strong convexity of Problem~\eqref{eq:DD_MPC_NL_nominal}, another candidate solution for Problem~\eqref{eq:DD_MPC_NL}, and properties of multi-parametric QPs.

For this technical result, we assume that the difference between the data and initial condition of the nonlinear system and that of the linearized dynamics is bounded, compare~\eqref{eq:prop_continuity_noise_bound}.
This condition can be interpreted as a bound on the ``noise'' affecting the data of the linearized dynamics and it is only required for the open-loop analysis of Problem~\eqref{eq:DD_MPC_NL} in Proposition~\ref{prop:continuity}.
In Section~\ref{subsec:NL_guarantees}, we prove that a bound as in~\eqref{eq:prop_continuity_noise_bound} is implicitly enforced by our MPC scheme in closed loop.
Since the bounds~\eqref{eq:prop_continuity_cost} and~\eqref{eq:prop_continuity} depend on the linearization-induced error $\bar{\varepsilon}$, the result can be interpreted as a continuity property of the optimal solution of Problem~\eqref{eq:DD_MPC_NL} w.r.t.\ additive and multiplicative perturbations.
Note that, in contrast to existing closed-loop results on data-driven MPC for linear systems~\cite{berberich2021guarantees,bongard2021robust}, Proposition~\ref{prop:continuity} not only allows for noisy output data in the Hankel matrices but also for additional disturbances affecting the initial extended state.

\begin{remark}\label{rk:affine_motivation}
Proposition~\ref{prop:continuity} plays a crucial role in the theoretical analysis of the proposed data-driven MPC scheme for nonlinear systems.
Beyond its application in Section~\ref{subsec:NL_guarantees}, the result also enables a novel separation-type proof of stability and robustness of data-driven MPC for affine systems with noisy data.
More precisely, according to Proposition~\ref{prop:continuity}, bounded output measurement noise (i.e., $\lVert y_{t+k}-y_k'(t)\rVert_{2}\leq\bar{\varepsilon}$) and initial condition perturbations (i.e., $\lVert\xi_t-\xi_t'\rVert_{2}\leq\bar{\varepsilon}$) in data-driven MPC translate into an input disturbance for the closed loop.
As we show in Appendix~\ref{sec:app_affine_MPC}, the nominal (i.e., noise-free) MPC scheme~\eqref{eq:DD_MPC_NL_nominal} without updating the data in~\eqref{eq:DD_MPC_NL_nominal_hankel} online stabilizes affine systems in the presence of input disturbances.
Combining these two results, we can directly conclude stability of data-driven tracking MPC for affine systems with noisy data and perturbed initial conditions.
This result has significant advantages over existing closed-loop results for linear data-driven MPC with noisy data~\cite{berberich2021guarantees,bongard2021robust}.
In particular, due to the online optimization of the artificial equilibrium $(u^{\rms}(t),y^{\rms}(t))$, i) it is not required to know a priori whether a given input-output setpoint is a feasible equilibrium, ii) the resulting scheme is recursively feasible in case of online setpoint changes, and iii) it possesses a significantly improved robustness and larger region of attraction.
Finally, Proposition~\ref{prop:continuity} can also be used to simplify the existing stability proofs of data-driven MPC in~\cite{berberich2021guarantees,bongard2021robust}.
In Appendix~\ref{sec:app_affine_MPC}, we provide a detailed discussion and analysis of data-driven tracking MPC for affine systems based on Proposition~\ref{prop:continuity}.
\end{remark}

\subsection{Closed-loop guarantees}\label{subsec:NL_guarantees}
Before presenting our main theoretical result, we make an additional assumption on closed-loop persistence of excitation.
\begin{assumption}\label{ass:closed_loop_pe}
(Closed-loop persistence of excitation)
There exists $c_{\rmH}>0$ such that, for all $t\geq N$, the matrix $H_{ux,t}$ from~\eqref{eq:Hux_NL} has full row rank and $\lVert H_{ux,t}^\dagger\rVert_2\leq c_{\rmH}$.
\end{assumption}
Assumption~\ref{ass:closed_loop_pe} ensures that the persistence of excitation condition in 
Definition~\ref{def:pe} holds uniformly in closed loop.
To be precise, in addition to $H_{ux,t}$ having full row rank, we also require a uniform upper bound on $\lVert H_{ux,t}^\dagger\rVert_2$ which holds if the singular values of $H_{ux,t}$ are uniformly lower and upper bounded.
Assumption~\ref{ass:closed_loop_pe} is crucial for our theoretical results since, if the data used for prediction are updated online, they may in general not be persistently exciting and can thus lead to inaccurate predictions.
We note that Assumption~\ref{ass:closed_loop_pe} can be restrictive and may not be satisfied in practice, in particular upon convergence of the closed loop.
Nevertheless, there are various pragmatic approaches to ensure the availability of persistently exciting data for our purposes such as (i) stopping the data updates as soon as a neighborhood of the setpoint is reached, (ii) adding suitable excitation signals to the closed-loop input applied to the system, or (iii) incentivizing persistently exciting inputs similar to adaptive MPC approaches~\cite{lu2020robust}.
Ensuring closed-loop persistence of excitation within our data-driven MPC framework for nonlinear systems is beyond the scope of this paper and we plan to investigate this issue in future research.

In the following, we present our main theoretical result on closed-loop stability.
Our analysis relies on the Lyapunov function candidate $V(\xi_t',\mathcal{D}_t)\coloneqq \check{J}_L^*(\xi_t',\mathcal{D}_t)-J_{\mathrm{eq},\rmlin}^*(x_t)$.


\begin{theorem}\label{thm:NL_stability}
Suppose $L\geq2n$ and Assumptions~\ref{ass:NL_sys}--\ref{ass:LICQ} hold.
There exists $\bar{\theta}>0$ such that for any $\theta\in(0,\bar{\theta}]$, if Assumption~\ref{ass:closed_loop_pe} holds with $c_{\rmH}=\frac{\bar{c}_{\rmH}}{\theta}$ for some $\bar{c}_{\rmH}>0$, then there exist $V_{\mathrm{ROA}},\bar{S},C>0$, regularization parameters $\lambda_{\alpha},\lambda_{\sigma}>0$, and $0<c_{\rmV}<1$, $\beta_{\theta}\in\mathcal{K}_{\infty}$, such that, if
\begin{align}
V(\xi_N',\mathcal{D}_N)&\leq V_{\mathrm{ROA}},\>\lambda_{\max}(S)\leq\bar{S}\\\label{eq:thm_NL_stability_xi_N}
\lVert \xi_N-\xi_i\rVert_2&\leq\theta,\>i\in\mathbb{I}_{[0,N-1]},
\end{align}
then, for any $t=N+ni$, $i\in\mathbb{I}_{\geq0}$, Problem~\eqref{eq:DD_MPC_NL} is feasible and the closed loop under Algorithm~\ref{alg:MPC_n_step_NL} satisfies
\begin{align}\label{eq:thm_NL_stability_decay}
\lVert \xi_{t}-\xi^{\rms\rmr}\rVert_2^2\leq c_{\rmV}^iC\lVert\xi_N-\xi^{\rms\rmr}\rVert_2^2+\beta_{\theta}(\theta).
\end{align}
\end{theorem}

The proof of Theorem~\ref{thm:NL_stability} is provided in Appendix~\ref{sec:app_thm_stab_proof}.
The result shows that the closed-loop state trajectory converges close to the optimal reachable equilibrium (cf.~\eqref{eq:thm_NL_stability_decay}) if the cost matrix $S$ is sufficiently small, the regularization parameters are chosen suitably (roughly speaking, $\lambda_{\alpha}$ is small and $\lambda_{\sigma}$ is large), and if the initial data satisfy $\lVert\xi_N-\xi_i\rVert_2\leq\theta$ for $i\in\mathbb{I}_{[0,N-1]}$ with a sufficiently small $\theta$.
Inequality~\eqref{eq:thm_NL_stability_decay} implies that the optimal reachable setpoint $\xi^{\rms\rmr}$ is practically exponentially stable, i.e., the closed-loop trajectory exponentially converges to a region around $\xi^{\rms\rmr}$ whose size increases with $\theta$.
In particular, under suitable assumptions, the proposed MPC scheme steers the closed loop arbitrarily close to $\xi^{\rms\rmr}$ if the initial state evolution, i.e., the parameter $\theta$, is sufficiently small.   
The proof relies on a separation argument, combining the continuity of data-driven MPC in Proposition~\ref{prop:continuity} with the model-based MPC analysis from~\cite{berberich2021linearpart1}.

In the following, we provide an interpretation for this stability result.
As shown in Section~\ref{subsec:prelim_bound}, the prediction accuracy of the implicit prediction model in~\eqref{eq:DD_MPC_NL} depends on the distance of past state values to the current state.
The parameter $\theta$ provides a bound on this distance at initial time $t=N$, which is shown to hold recursively for all $t\geq N$ in the proof.
Thus, if $\theta$ is sufficiently small at $t=N$, then the closed loop does not move too rapidly such that the prediction model remains accurate and stability can be shown.
The conditions $\lambda_{\max}(S)\leq\bar{S}$ and $V(\xi_0'(N),\mathcal{D}_N)\leq V_{\mathrm{ROA}}$ for sufficiently small $\bar{S}$ and $V_{\mathrm{ROA}}$ are similar to the results for model-based MPC in~\cite{berberich2021linearpart1}, and they guarantee that the closed-loop trajectory remains close to the steady-state manifold and that the bound $\lVert \xi_t-\xi_i\rVert_2\leq c_{\theta,1}\theta$, $i\in\mathbb{I}_{[t-N,t-1]}$, holds recursively.
Hence, in comparison to~\cite{berberich2021linearpart1}, $\bar{S}$ as well as $V_{\mathrm{ROA}}$ need to be potentially smaller such that the bound $\lVert \xi_t-\xi_i\rVert_2\leq c_{\theta,1}\theta$, $i\in\mathbb{I}_{[t-N,t-1]}$, holds and, therefore, closed-loop stability can be guaranteed.
Further, the proof of Theorem~\ref{thm:NL_stability} shows that the choice of the regularization parameters $\lambda_{\alpha}$ and $\lambda_{\sigma}$ leading to closed-loop practical stability depends on $\theta$.
In particular, we have $\lambda_{\alpha}=\bar{\lambda}_{\alpha}\theta^{2\beta_{\alpha}}$ and $\lambda_{\sigma}=\frac{\bar{\lambda}_{\sigma}}{\theta^{2\beta_{\sigma}}}$ for some $\bar{\lambda}_{\alpha},\bar{\lambda}_{\sigma},\beta_{\alpha},\beta_{\sigma}>0$, i.e., the regularization of $\alpha(t)$ / $\sigma(t)$ increases / decreases with the parameter $\theta$ quantifying the prediction error.

The theoretical stability result (Theorem~\ref{thm:NL_stability}) uses Assumptions~\ref{ass:NL_sys}--\ref{ass:closed_loop_pe}.
If the nonlinear dynamics~\eqref{eq:sys_NL} are known, then they can be verified analogously to~\cite{berberich2021linearpart1}.
If the dynamics are unknown but only input-output data are available, then performing this verification is an interesting issue for future research with preliminary results for linear systems in~\cite{berberich2020constraints}.

Note that Assumption~\ref{ass:closed_loop_pe} requires a minimum amount of persistence of excitation, i.e., a uniform bound $\lVert H_{ux,t}^\dagger\rVert_2\leq c_{\rmH}$.
On the other hand, Theorem~\ref{thm:NL_stability} also requires an (initial) upper bound on $\lVert\xi_t-\xi_{t-N+i}\rVert\leq\theta$, $i\in\mathbb{I}_{[0,N-1]}$, which in turn limits the amount of persistence of excitation.
By allowing $c_{\rmH}$ to depend on $\theta$ as $c_{\rmH}=\frac{\bar{c}_{\rmH}}{\theta}$, our analysis takes into account the fact that $\theta$ cannot be arbitrarily small for a fixed $c_{\rmH}$, i.e., for a fixed lower bound on the amount of persistence of excitation.
Assumption~\ref{ass:closed_loop_pe} is the main reason why Theorem~\ref{thm:NL_stability} only provides a \emph{practical} stability result, i.e., asymptotic stability can in general not be proven if we assume uniform closed-loop persistence of excitation.
This is in contrast to the results in our companion paper~\cite{berberich2021linearpart1} which showed closed-loop \emph{exponential} stability assuming that an exact model of the linearization is available, i.e., without any requirements on persistence of excitation for the closed loop.

For the numerical example in Section~\ref{sec:example}, we observe that the closed loop indeed converges very closely to the optimal reachable equilibrium $\xi^{\rms\rmr}$ and, moreover, closed-loop persistence of excitation can be ensured by stopping the data updates.
Finally, the condition $\lVert \xi_t-\xi_i\rVert_2\leq\theta$, $i\in\mathbb{I}_{[t-N,t-1]}$, is generally easier to satisfy for smaller values of the data length $N$ and the prediction model in~\eqref{eq:DD_MPC_NL_hankel} is less accurate for too large values of $N$.
This is different to linear data-driven MPC, where larger values of $N$ usually improve the closed-loop performance in case of noisy data.
For nonlinear systems, on the other hand, ``older'' data points correspond to an ``older'' approximate linear model of the system which typically yields a worse approximation at the current state.

\begin{remark}\label{rk:noise}
While we assume that the data $\{u_k,y_k\}_{k=t-N}^{t-1}$ used in Problem~\eqref{eq:DD_MPC_NL} are noise-free, our results can be extended to output measurement noise.
To be more precise, suppose output measurements $\tilde{y}_k\coloneqq y_k+\epsilon_k$ with $\lVert\epsilon_k\rVert_{2}\leq\bar{\epsilon}$ are available for some sufficiently small $\bar{\epsilon}>0$.
Then, under the conditions in Theorem~\ref{thm:NL_stability}, it can be shown analogously to the proof of Theorem~\ref{thm:NL_stability} that there exist $\beta_{\epsilon}\in\mathcal{K}_{\infty}$ as well as (possibly different) $0<c_{\rmV}<1$, $C>0$, $\beta_{\theta}\in\mathcal{K}_{\infty}$, such that
\begin{align}\label{eq:rk_noise_ineq}
\lVert\xi_t-\xi^{\rms\rmr}\rVert_2^2\leq c_{\rmV}^iC\lVert\xi_N-\xi^{\rms\rmr}\rVert_2^2
+\beta_{\theta}(\theta)+\beta_{\epsilon}(\bar{\epsilon})
\end{align}
for any $t=N+ni$, $i\in\mathbb{I}_{\geq0}$.
Inequality~\eqref{eq:rk_noise_ineq} means that the closed loop converges to a region around the optimal reachable steady-state whose size increases with the parameter $\theta$ and the noise bound $\bar{\epsilon}$.
This is possible since the perturbation of Problem~\eqref{eq:DD_MPC_NL} caused by the linearization error is in fact handled as output measurement noise affecting the input-output data, compare Lemma~\ref{lem:pred_error_NL}.
Thus, small levels of noise will not affect the qualitative theoretical guarantees, but may only deteriorate the quantitative properties, e.g., increase the tracking error and decrease the region of attraction.
\end{remark}


\section{Numerical example}\label{sec:example}
We apply the presented data-driven MPC approach (see Algorithm~\ref{alg:MPC_n_step_NL}) to the nonlinear continuous stirred tank reactor from~\cite{mayne2011tube}.
We consider the same discretized system as in our companion paper~\cite{berberich2021linearpart1} (see~\cite{berberich2021linearpart1} for a detailed description of the system dynamics, parameters, and satisfaction of Assumption~\ref{ass:NL_sys} as well as the counterpart of Assumption~\ref{ass:NL_manifold_convex}).
In contrast to~\cite{berberich2021linearpart1}, the system model is unknown and we only have access to input-output data.
To implement the MPC, we choose the parameters
\begin{align*}
Q=1, R=5\cdot 10^{-2}, S=10, \lambda_{\alpha}=3\cdot 10^{-6},\lambda_{\sigma}=10^7,
\end{align*}
and the prediction horizon $L=40$.
Further, the constraints are $\mathbb{U}=[0.1,2]$, $\mathbb{U}^{\rms}=[0.11,1.99]$.
In comparison to~\cite{berberich2021linearpart1}, where we chose $S=100$, we use the smaller choice $S=10$ for the data-driven MPC scheme considered in this paper to ensure
that the closed loop does not change too rapidly, i.e., the bound $\lVert\xi_t-\xi_i\rVert_2\leq c_{\xi,0}\theta$, $i\in\mathbb{I}_{[t-N,t-1]}$, holds in closed loop for a sufficiently small $\theta$.
This is required since, in contrast to~\cite{berberich2021linearpart1}, Problem~\eqref{eq:DD_MPC_NL} only contains an \emph{approximation} of the linearized dynamics whose accuracy deteriorates for larger values of $\theta$, compare Lemma~\ref{lem:pred_error_NL}.

In order to ensure that the dynamics are control-affine as in~\eqref{eq:sys_NL}, we introduce an incremental input $\Delta u_k=u_{k+1}-u_k$ playing the role of the input $u_k$ in the previous sections, similar to our companion paper~\cite{berberich2021linearpart1}.
To guarantee satisfaction of the input constraints $u_t\in\mathbb{U}$, we consider an extended output vector $\hat{y}_t=\begin{bmatrix}y_t\\u_t\end{bmatrix}$ and add output constraints of the form $\hat{y}_k(t)\in\mathbb{R}\times\mathbb{U}$, $k\in\mathbb{I}_{[0,L]}$, as well as $\hat{y}^\rms(t)\in\mathbb{R}\times\mathbb{U}^\rms$ in Problem~\eqref{eq:DD_MPC_NL}.
With these modifications, the theoretical guarantees of our MPC scheme remain true and the output constraints are satisfied in closed loop since the available data provide an exact ``prediction model'' of $\Delta u\mapsto u$.

Further, we replace $\alpha^{\rms\rmr}_{\rmlin}(x_t)$ in Problem~\eqref{eq:DD_MPC_NL_cost} by the ``approximation'' $\alpha^*(t-n)$ to regularize w.r.t.\ the previously optimal solution and thus to encourage stationary behavior.
The data length is chosen as $N=120$ and we generate initial data samples for $t\in\mathbb{I}_{[0,N-1]}$ by sampling the input uniformly from $u_t\in[0.1,1]$.
Finally, to ensure that the data used for prediction are persistently exciting, we stop updating the data in~\eqref{eq:DD_MPC_NL_hankel} 
as soon as the tracking cost is sufficiently small, i.e.,
\begin{align*}
\sum_{k=-n}^L&\lVert\Delta \bar{u}_k^*(t)-\overbrace{\Delta u^{\rms*}(t)}^{=0}\rVert_2^2
+\lVert\bar{u}_k^*(t)-u^{\rms*}(t)\rVert_R^2\\
&+\lVert\bar{y}_k^*(t)-y^{\rms*}(t)\rVert_{Q}^2\leq10^{-5}.
\end{align*}

We now apply the multi-step MPC scheme in Algorithm~\ref{alg:MPC_n_step_NL} with the above parameters and modifications. 
First, we note that updating the data in Problem~\eqref{eq:DD_MPC_NL} online is a crucial ingredient of our approach.
In particular, a data-driven MPC using the data $\{u_k,y_k\}_{k=0}^{N-1}$ for prediction at all times leads to a huge tracking error, i.e., the output converges to $1.14\neq y^\rmr=0.6519$.
The closed-loop input-output trajectory under our MPC approach is depicted in Figure~\ref{fig:example}, along with the trajectory resulting from the multi-step MPC based on a linearized model in our companion paper~\cite{berberich2021linearpart1} (using the same parameters as above, except for $S=100$ instead of $S=10$).
In the data-driven MPC, the input is more unsteady which can be explained by the combination of the less accurate prediction model and terminal equality constraints.
Moreover, since the matrix $S$ is chosen smaller in comparison to the model-based MPC ($S=10$ instead of $S=100$), the convergence towards the setpoint $y^\rmr$ is slower.
Solving Problem~\eqref{eq:DD_MPC_NL} online takes on average $13$ milliseconds, i.e., slightly longer than the model-based MPC algorithm from~\cite{berberich2021linearpart1} which takes $7.4$ milliseconds.

Finally, for comparison, we apply a model-based MPC scheme based on online least-squares identification\footnote{For the identification, we implemented an incremental parameter update based on a regularization w.r.t.\ the previous parameter estimate.
In particular, a model-based MPC using a simple least-squares parameter estimation does not suffice to successfully steer the system to the setpoint.}
of an affine input-output model 
\begin{align*}
y_{k+2}=a_1y_{k+1}+a_0y_k+b_2u_{k+2}+b_1u_{k+1}+b_0u_k+c,
\end{align*}
using the last $N$ data points $\{u_k,y_k\}_{k=t-N}^{t-1}$ for identification at time $t$.
The MPC implementation is analogous to that in~\cite{berberich2021linearpart1} with an extended state-space model, compare, e.g.,~\cite[Lemma 2]{koch2021provably}, and the parameters $Q$, $R$, horizon length $L$, and data length $N$ are as above, including the incremental input penalty $\lVert \Delta u_k\rVert_2^2$.
The matrix $S$ is chosen as $S=30$, i.e., larger than for the data-driven MPC, since the slack variable in Problem~\eqref{eq:DD_MPC_NL} relaxes the terminal equality constraint and thereby speeds up the closed-loop convergence.
While this identification-based MPC scheme exhibits comparable performance to the proposed data-driven MPC scheme (cf. Figure~\ref{fig:example}), the existing literature does not contain any theoretical results on closed-loop stability under similar assumptions as in the present paper.

\begin{figure}
		\begin{center}
		\subfigure
		{\includegraphics[width=0.48\textwidth]{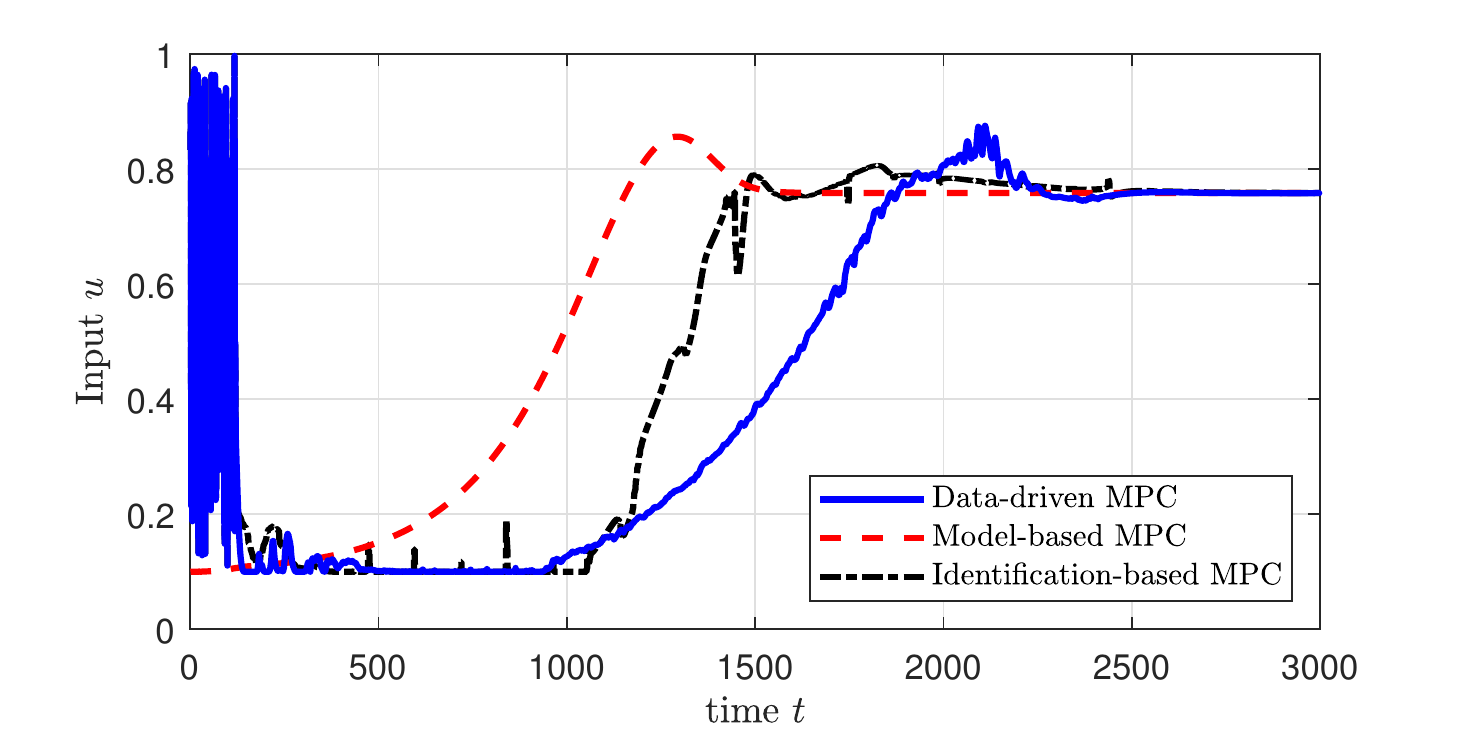}}
		\subfigure
		{\includegraphics[width=0.48\textwidth]{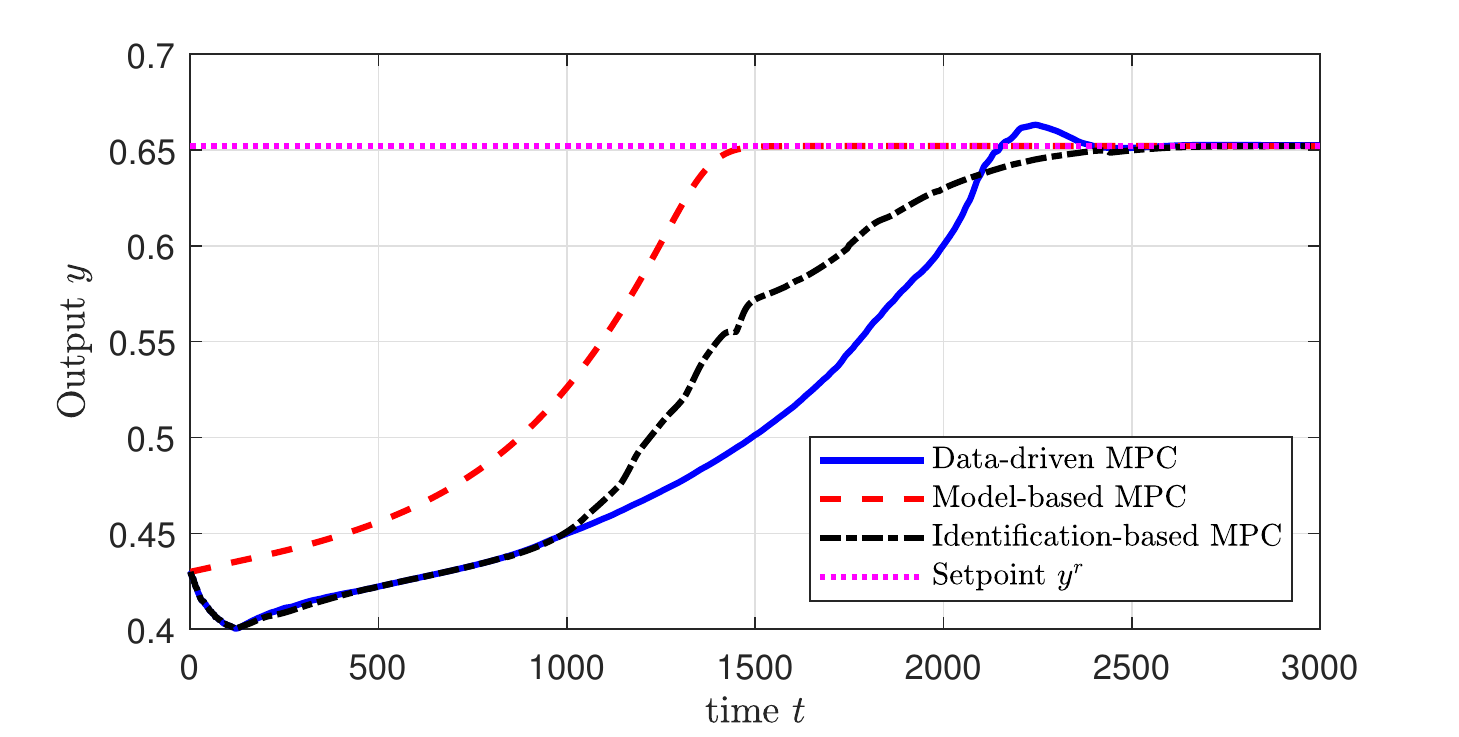}}
		\end{center}
		\caption{Closed-loop input $u$ and output $y$ of the example in Section~\ref{sec:example}.
		}	\label{fig:example}
\end{figure}

Regarding the condition~\eqref{eq:thm_NL_stability_xi_N} in Theorem~\ref{thm:NL_stability}, we note that the quantity $\bar{\theta}(t)\coloneqq\max_{i\in\mathbb{I}_{[t-N,t-1]}}\lVert\xi_t-\xi_i\rVert_2$ takes its maximum at $t=141$, i.e., directly after the initial excitation phase.
This yields a uniform bound on the distance between closed-loop state trajectories, i.e., $\lVert\xi_t-\xi_i\rVert_2\leq c_{\theta,1}\theta$ for all $t\in\mathbb{I}_{[N,3000]}$, $i\in\mathbb{I}_{[t-N,t-1]}$, where $\theta\coloneqq\max_t\bar{\theta}(t)$, compare~\eqref{eq:thm_NL_stability_proof_xi_bound} in the proof of Theorem~\ref{thm:NL_stability}.
Let us now investigate whether this uniform upper bound conflicts with the requirements on closed-loop persistence of excitation in Assumption~\ref{ass:closed_loop_pe}.
For the present example, we observe that the norm of $H_{ux,t}^\dagger$ increases in closed loop under the data-driven MPC when the trajectory approaches the setpoint.
However, the product $\lVert H_{ux,t}^\dagger\rVert_2\bar{\theta}(t)$ is uniformly bounded, even when $x_t$ is close to $x^{\rms\rmr}$, and hence, there exists a constant $\bar{c}_{\rmH}>0$ such that $\lVert H_{ux,t}^\dagger\rVert_2\leq\frac{\bar{c}_{\rmH}}{\bar{\theta}(t)}$ holds, compare the conditions in Theorem~\ref{thm:NL_stability}.
Thus, while $\lVert H_{ux,t}^\dagger\rVert_2$ increases for smaller bounds $\bar{\theta}(t)$, the order of increase is not larger than $\frac{1}{\bar{\theta}(t)}$ such that the persistence of excitation assumptions required for Theorem~\ref{thm:NL_stability} are fulfilled.

In Figure~\ref{fig:cost_plot}, we show the closed-loop cost $\sum_{t=0}^{3000}\lVert y_t-y^{\rmr}\rVert_2^2$ when varying the parameters $S$ and $\lambda_{\alpha}$, keeping all other parameters as above.
The displayed cost is normalized, i.e., it is multiplied by $\frac{1}{J_{\mathrm{mdl}}}$, where $J_{\mathrm{mdl}}$ denotes the closed-loop cost $\sum_{t=0}^{3000}\lVert y_t-y^{\rmr}\rVert_2^2$ of the model-based MPC from~\cite{berberich2021linearpart1}.
First, the closed-loop cost of the data-driven MPC is generally worse than that of the model-based MPC due to the initial excitation phase as well as the smaller choice of $S$ and the resulting slower convergence.
Further, recall that smaller values of $S$ decrease the bound $\bar{\theta}(t)$ above and thus improve the prediction accuracy.
If $S$ is chosen too large, the closed-loop trajectory moves too rapidly such that the predictions are inaccurate and the performance is more sensitive w.r.t.\ variations of $\lambda_{\alpha}$.
On the other hand, too small values of $S$ lead to a large cost since the convergence is slower and the term $\lambda_{\alpha}\lVert\alpha(t)-\alpha^{\rms\rmr}_{\rmlin}(\calD_t)\rVert_2^2$ dominates in the optimization of~\eqref{eq:DD_MPC_NL}, thus increasing the stationary tracking error (recall that we consider $\alpha^{\rms\rmr}_{\rmlin}(\calD_t)\approx\alpha^*(t-n)$).
However, there exists a corridor $S\in[7,30]$ for which the closed-loop performance is acceptable over a wide range of $\lambda_{\alpha}$.

\begin{figure}
\begin{center}
\includegraphics[width=0.48\textwidth]{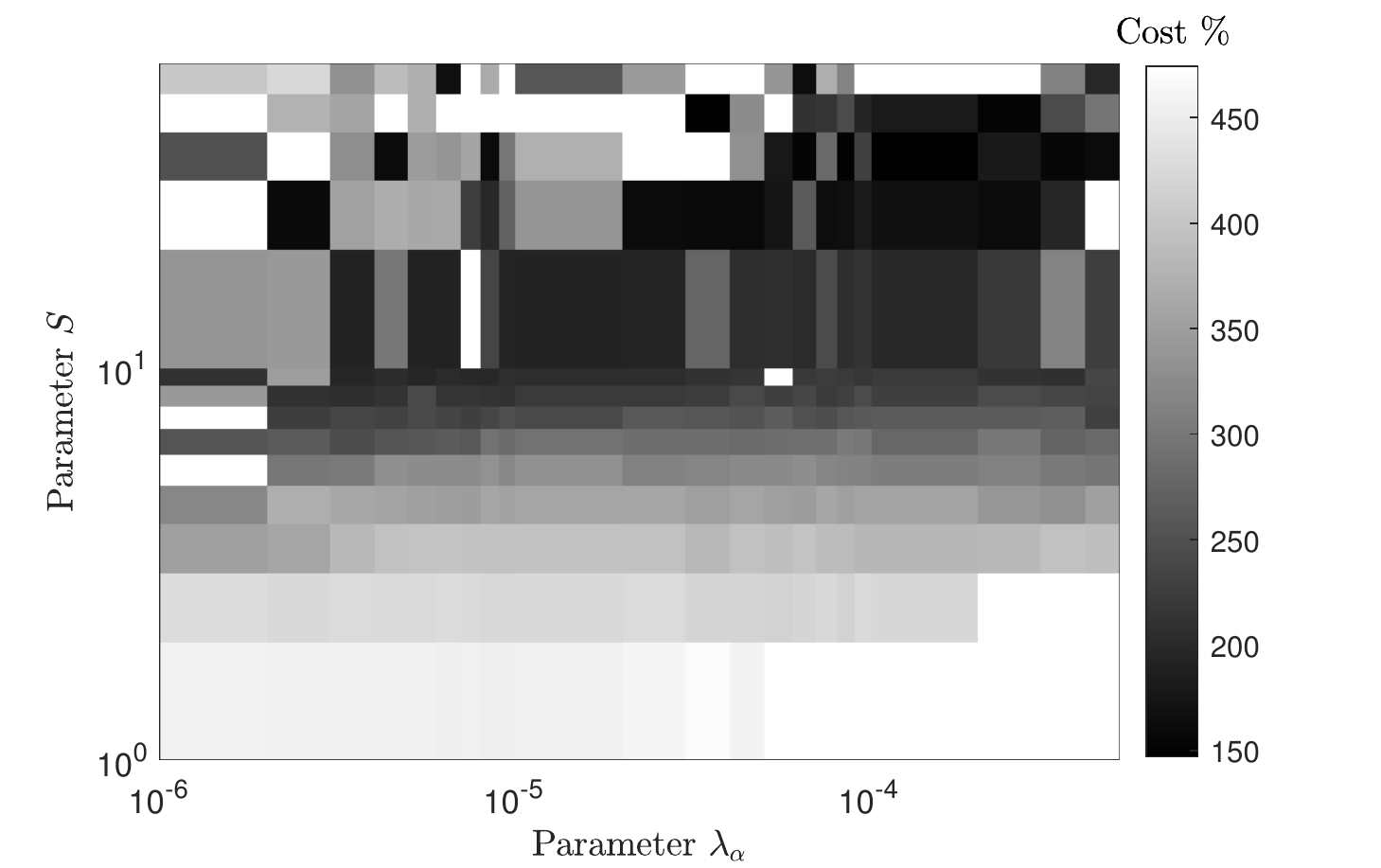}
\end{center}
\caption{Normalized closed-loop cost depending on the parameters $S$ and $\lambda_{\alpha}$ for the numerical example in Section~\ref{sec:example}.}
\label{fig:cost_plot}
\end{figure}

The other parameters in Problem~\eqref{eq:DD_MPC_NL} influence the closed-loop performance as well and performing a more detailed case study for different applications is a very interesting issue for future research.
To this end,~\cite{berberich2021at} analyzes the influence of different parameter choices on the closed-loop operation when applying the proposed MPC scheme to a nonlinear four-tank system in simulation and in a real-world experiment.

Finally, we note that the proof of Theorem~\ref{thm:NL_stability} relies on multiple conservative estimates and should therefore be interpreted as a qualitative stability result, compare the discussion below Theorem~\ref{thm:NL_stability}.
Correspondingly, the magnitude of the (initial) excitation bound $\theta$ and the choices of the tuning variables $S$, $\lambda_{\alpha}$, $\lambda_{\sigma}$ considered in the example are not necessarily chosen small (for $\theta$, $S$, $\lambda_{\alpha}$) or large (for $\lambda_{\sigma}$) enough to satisfy the corresponding bounds in the proof.

\section{Conclusion}\label{sec:conclusion}
We presented a novel data-driven MPC approach to control unknown nonlinear systems based on input-output data with closed-loop stability guarantees.
Our approach relied on the data-driven system parametrization of affine systems stated in Section~\ref{sec:willems_affine}, which can (approximately) describe nonlinear systems by updating the data online.
We proved that the proposed MPC scheme practically stabilizes the closed loop under the assumptions that the design parameters are chosen suitably, the initial condition is close to the steady-state manifold, and the data generated during initialization are not too far apart.
As intermediate results of independent interest, we extended the Fundamental Lemma of~\cite{willems2005note} to affine systems and we derived novel robustness bounds of data-driven MPC which are directly applicable to other data-driven MPC formulations in the literature.
Further, as we show in Appendix~\ref{sec:app_affine_MPC}, they allow for designing a robust data-driven tracking MPC scheme for affine systems which admits stronger theoretical guarantees than existing data-driven MPC approaches, cf. Remark~\ref{rk:affine_motivation}.


We did not address the issue of how to ensure closed-loop persistence of excitation (Assumption~\ref{ass:closed_loop_pe}) and, in particular, we only showed \emph{practical} stability of the closed loop.
Tackling these issues by developing appropriate modifications of our MPC algorithm is a highly interesting and relevant future research direction, potentially leading to both superior theoretical guarantees and more reliable practical performance.

\bibliographystyle{IEEEtran}   
\bibliography{Literature}  

\section*{Appendix}

\renewcommand\thesection{\Alph{section}}
\setcounter{section}{0}

\section{Proof of Proposition~\ref{prop:continuity}}\label{sec:app_prop_continuity_proof}
\begin{proof}
\textbf{(i) Proof of~\eqref{eq:prop_continuity_cost}}\\
We construct a candidate solution for Problem~\eqref{eq:DD_MPC_NL} based on the optimal solution of Problem~\eqref{eq:DD_MPC_NL_nominal} for $\tilde{\sigma}=0$.
We choose 
\begin{align}\label{eq:app_prop_proof2}
\bar{u}(t)=\begin{bmatrix}u_{[t-n,t-1]}\\\check{u}_{[0,L]}^*(t)\end{bmatrix},\>\>
\bar{y}(t)=\begin{bmatrix}y_{[t-n,t-1]}\\\check{y}_{[0,L]}^*(t)\end{bmatrix},
\end{align}
and $u^\rms(t)=\check{u}^{\rms*}(t)$, $y^\rms(t)=\check{y}^{\rms*}(t)$.
Further, we define $\alpha(t)=H_{ux,t}^\dagger\begin{bmatrix}\bar{u}(t)\\x_{t-n}\\1\end{bmatrix}$ with $H_{ux,t}$ as in~\eqref{eq:Hux_NL}.
Moreover, we let 
\begin{align}\label{eq:app_prop_proof3}
\sigma(t)=H_{L+n+1}(y_{[t-N,t-1]})\alpha(t)-\bar{y}(t).
\end{align}
This candidate fulfills all constraints of Problem~\eqref{eq:DD_MPC_NL}.
Note that $\check{\alpha}^*(t)-\alpha^{\rms\rmr}_{\rmlin}(\calD_t)$ satisfies
\begin{align}\label{eq:app_Hankel_equation}
&\begin{bmatrix}H_{L+n+1}(u_{[t-N,t-1]})\\H_{L+n+1}(y'(t))\\\mathbbm{1}_{N-L-n}^\top\end{bmatrix}(\check{\alpha}^*(t)-\alpha^{\rms\rmr}_{\rmlin}(\calD_t))\\\nonumber
=&\begin{bmatrix}\check{u}^*(t)-\bbone_{L+n+1}\otimes u^{\rms\rmr}_{\rmlin}(x_t)\\\check{y}^*(t)-\bbone_{L+n+1}\otimes y^{\rms\rmr}_{\rmlin}(x_t)\\0\end{bmatrix}.
\end{align}
Due to the minimization in~\eqref{eq:DD_MPC_NL_nominal_cost}, $\check{\alpha}^*(t)-\alpha^{\rms\rmr}_{\rmlin}(\calD_t)$ is the vector satisfying~\eqref{eq:app_Hankel_equation} with minimum norm, which implies $\check{\alpha}^*(t)=H_{ux,t}^\dagger\begin{bmatrix}\check{u}^*(t)\\x_{t-n}\\1\end{bmatrix}$.
Thus, it holds that
\begin{align}\label{eq:app_alpha_diff_aux}
\lVert\alpha(t)-\check{\alpha}^*(t)\rVert_2^2\stackrel{\eqref{eq:app_prop_proof2}}{\leq}&\lVert H_{ux,t}^\dagger\rVert_2^2\lVert u_{[t-n,t-1]}-u_{[t-n,t-1]}'\rVert_2^2\\\nonumber
\stackrel{\eqref{eq:prop_continuity_noise_bound}}{\leq}&\lVert H_{ux,t}^\dagger\rVert_2^2\bar{\varepsilon}^2.
\end{align}
Using $\lVert a+b\rVert_2^2\leq(1+\bar{\varepsilon}^{\beta_{\sigma}})\lVert a\rVert_2^2+(1+\frac{1}{\bar{\varepsilon}^{\beta_{\sigma}}})\lVert b\rVert_2^2$ for arbitrary $a$, $b$, this implies
\begin{align}\label{eq:app_alpha_diff}
&\lVert\alpha(t)-\alpha^{\rms\rmr}_{\rmlin}(\calD_t)\rVert_2^2-\lVert\check{\alpha}^*(t)-\alpha^{\rms\rmr}_{\rmlin}(\calD_t)\rVert_2^2\\\nonumber
\leq&\bar{\varepsilon}^{\beta_{\sigma}}\lVert\check{\alpha}^*(t)-\alpha^{\rms\rmr}_{\rmlin}(\calD_t)\rVert_2^2+\lVert H_{ux,t}^\dagger\rVert_2^2(\bar{\varepsilon}^2+\bar{\varepsilon}^{2-\beta_{\sigma}}).
\end{align}
Note that the output trajectories $H_{L+n+1}(y'(t))\alpha(t)$ and $\check{y}^*(t)$ differ only due to the difference in the first $n$ components of the corresponding input trajectory, which is in turn bounded by $\bar{\varepsilon}$, cf.~\eqref{eq:prop_continuity_noise_bound} and~\eqref{eq:app_prop_proof2}.
Thus, there exists $c_1^{\rmu}>0$ such that
\begin{align}\label{eq:app_prop_proof1}
\lVert H_{L+n+1}(y'(t))\alpha(t)-\check{y}^*(t)\rVert_2^2\stackrel{\eqref{eq:prop_continuity_noise_bound}}{\leq} c_1^{\rmu}\bar{\varepsilon}^2.
\end{align}
Hence, there exist $c_2^{\rmu},c_3^{\rmu}>0$ such that
\begin{align}\label{eq:app_sigma_diff}
&\quad\lVert\sigma(t)\rVert_2^2\\\nonumber
&\stackrel{\eqref{eq:ab_ineq2},\eqref{eq:app_prop_proof3}}{\leq} 2\lVert \left(H_{L+n+1}(y_{[t-N,t-1]})-H_{L+n+1}(y'(t))\right)\alpha(t)\rVert_2^2\\\nonumber
&\quad+4\lVert H_{L+n+1}(y'(t))\alpha(t)-\check{y}^*(t)\rVert_2^2+4\lVert\check{y}^*(t)-\bar{y}(t)\rVert_2^2\\\nonumber
&\stackrel{\eqref{eq:prop_continuity_noise_bound},\eqref{eq:app_prop_proof2},\eqref{eq:app_prop_proof1}}{\leq}
c_2^{\rmu}\bar{\varepsilon}^2+c_3^{\rmu}\bar{\varepsilon}^2\lVert\alpha(t)\rVert_2^2.
\end{align}
Applying $\lVert a+b\rVert_R^2\leq(1+\bar{\varepsilon}^{\beta_\sigma})\lVert a\rVert_R^2+(1+\frac{1}{\bar{\varepsilon}^{\beta_{\sigma}}})\lVert b\rVert_R^2$, which holds for arbitrary vectors $a$, $b$, we have
\begin{align}\nonumber
\sum_{k=-n}^{-1}\lVert u_{t+k}-\check{u}^{\rms*}(t)&\rVert_R^2
\leq\sum_{k=-n}^{-1}\left(1+\frac{1}{\bar{\varepsilon}^{\beta_{\sigma}}}\right)\lVert u_{t+k}-u_{t+k}'\rVert_R^2\\\label{eq:app_input_traj_diff}
&+(1+\bar{\varepsilon}^{\beta_{\sigma}})\lVert u_{t+k}'-\check{u}^{\rms*}(t)\rVert_R^2,
\end{align}
and analogously for the output.
Using additionally
\begin{align*}
&\sum_{k=-n}^{-1}\lVert u_{t+k}-u_{t+k}'\rVert_R^2+\lVert y_{t+k}-y_{t+k}'\rVert_Q^2
\stackrel{\eqref{eq:prop_continuity_noise_bound}}{\leq} \lambda_{\max}(Q,R)\bar{\varepsilon}^2,\\
&\sum_{k=-n}^{-1}\lVert u_{t+k}'-\check{u}^{\rms*}(t)\rVert_R^2
+\lVert y_{t+k}'-\check{y}^{\rms*}(t)\rVert_Q^2\leq\check{J}_L^*(\xi_t',\mathcal{D}_t),
\end{align*}
we obtain
\begin{align}\label{eq:app_prop_proof4}
&\sum_{k=-n}^{-1}\lVert u_{t+k}-\check{u}^{\rms*}(t)\rVert_R^2-\lVert u_{t+k}'-\check{u}^{\rms*}(t)\rVert_R^2\\\nonumber
&+\sum_{k=-n}^{-1}\lVert y_{t+k}-\check{y}^{\rms*}(t)\rVert_Q^2-\lVert y_{t+k}'-\check{y}^{\rms*}(t)\rVert_Q^2\\\nonumber
\leq&\bar{\varepsilon}^{\beta_{\sigma}}\check{J}_L^*(\xi_t',\mathcal{D}_t)
+c_4^{\rmu}(\bar{\varepsilon}^2+\bar{\varepsilon}^{2-\beta_{\sigma}})
\end{align}
for a suitably defined $c_4^{\rmu}>0$.
Using the constructed candidate solution, we have
\begin{align}\label{eq:app_J}
&J_L^*(\xi_t)-\check{J}_L^*(\xi_t',\calD_t)\\\nonumber
\stackrel{\eqref{eq:app_prop_proof4}}{\leq} &c_4^{\rmu}(\bar{\varepsilon}^2+\bar{\varepsilon}^{2-\beta_{\sigma}})+\bar{\varepsilon}^{\beta_{\sigma}}\check{J}_L^*(\xi_t',\calD_t)+\frac{\bar{\lambda}_{\sigma}}{\bar{\varepsilon}^{\beta_{\sigma}}}\lVert\sigma(t)\rVert_2^2\\\nonumber
&+\bar{\lambda}_{\alpha}\bar{\varepsilon}^{\beta_{\alpha}}(\lVert\alpha(t)-\alpha^{\rms\rmr}_{\rmlin}(x_t)\rVert_2^2-\lVert\check{\alpha}^*(t)-\alpha^{\rms\rmr}_{\rmlin}(\calD_t)\rVert_2^2)\\\nonumber
\stackrel{\eqref{eq:app_alpha_diff},\eqref{eq:app_sigma_diff}}{\leq}&c_4^{\rmu}(\bar{\varepsilon}^2+\bar{\varepsilon}^{2-\beta_{\sigma}})+2\bar{\varepsilon}^{\beta_{\sigma}}\check{J}_L^*(\xi_t',\calD_t)+\bar{\lambda}_{\sigma}c_2^{\rmu}\bar{\varepsilon}^{2-\beta_{\sigma}}\\\nonumber
&+\bar{\lambda}_{\alpha}\bar{\varepsilon}^{\beta_{\alpha}}\lVert H_{ux,t}^\dagger\rVert_2^2(\bar{\varepsilon}^{2}+\bar{\varepsilon}^{2-\beta_{\sigma}})+\bar{\lambda}_{\sigma}c_3^{\rmu}\bar{\varepsilon}^{2-\beta_{\sigma}}\lVert\alpha(t)\rVert_2^2,
\end{align}
where we also use $\bar{\lambda}_{\alpha}\bar{\varepsilon}^{\beta_{\alpha}}\lVert\check{\alpha}^*(t)-\alpha^{\rms\rmr}_{\rmlin}(\calD_t)\rVert_2^2\leq\check{J}_L^*(\xi_t',\calD_t)$ for the second inequality.
The term $\lVert\alpha(t)\rVert_2^2$ is bounded as
\begin{align*}
&\quad\lVert\alpha(t)\rVert_2^2\\
&\stackrel{\eqref{eq:ab_ineq2}}{\leq}2\lVert\alpha(t)-\check{\alpha}^*(t)\rVert_2^2
+4\lVert\check{\alpha}^*(t)-\alpha^{\rms\rmr}_{\rmlin}(\calD_t)\rVert_2^2
+4\lVert\alpha^{\rms\rmr}_{\rmlin}(\calD_t)\rVert_2^2\\
&\stackrel{\eqref{eq:app_alpha_diff_aux}}{\leq} 2\lVert H_{ux,t}^\dagger\rVert_2^2\bar{\varepsilon}^2+4\frac{1}{\bar{\lambda}_{\alpha}\bar{\varepsilon}^{\beta_{\alpha}}}\check{J}_L^*(\xi_t',\calD_t)+4\lVert\alpha^{\rms\rmr}_{\rmlin}(\calD_t)\rVert_2^2,
\end{align*}
where $\lVert\alpha_{\rmlin}^{\rms\rmr}(\calD_t)\rVert_2^2$ is uniformly bounded by assumption.
Plugging this into~\eqref{eq:app_J} and using $\beta_{\alpha}+2\beta_{\sigma}<2$, the term with the smallest exponent of $\bar{\varepsilon}$ is proportional to $\bar{\varepsilon}^{2-\beta_{\sigma}}\left(1+\lVert H_{ux,t}^\dagger\rVert_2^2\bar{\varepsilon}^{\beta_{\alpha}}\right)$.
Thus, letting $\bar{\varepsilon}_{\max}<1$, we obtain~\eqref{eq:prop_continuity_cost} for appropriately defined constants $c_{\rmJ,\rma},c_{\rmJ,\rmb}>0$.\\
\textbf{(ii) Proof of~\eqref{eq:prop_continuity}}\\
We note that 
\begin{align}\label{eq:prop_continuity_proof_first_n_steps}
&\lVert\bar{u}_{[-n,-1]}^*(t)-\check{u}_{[-n,-1]}^*(t)\rVert_2\\\nonumber
=&\lVert u_{[t-n,t-1]}-u'_{[t-n,t-1]}\rVert_2\stackrel{\eqref{eq:prop_continuity_noise_bound}}{\leq}\bar{\varepsilon},
\end{align}
which implies a bound of the form~\eqref{eq:prop_continuity} for the first $n$ elements.
Therefore, we focus on $k\in\mathbb{I}_{[0,L]}$ in the following.\\
\textbf{(ii).a Strong convexity of Problem~\eqref{eq:DD_MPC_NL_nominal}}\\
Exploiting the initial condition~\eqref{eq:DD_MPC_NL_nominal_init} and the terminal constraint~\eqref{eq:DD_MPC_NL_nominal_TEC}, the terms in~\eqref{eq:DD_MPC_NL_nominal_cost} involving $\bar{u}(t)$ and $u^\rms(t)$ are
\begin{align*}
\sum_{k=-n}^{-1}\lVert u_{t+k}'+F_k\tilde{\sigma}_{\mathrm{init}}-u^\rms(t)\rVert_R^2+\sum_{k=0}^{L-n-1}\lVert\bar{u}_k(t)-u^\rms(t)\rVert_R^2,
\end{align*}
where the matrix $F_k$ picks the $k$-th input component of $\tilde{\sigma}_{\mathrm{init}}$.
The Hessian w.r.t.\ $\begin{bmatrix}\bar{u}_{[0,L-n-1]}(t)\\u^\rms(t)\end{bmatrix}$ is equal to
\begin{align*}
2\begin{bmatrix}I_{L-n}\otimes R&-\mathbbm{1}_{L-n}\otimes R\\
-\mathbbm{1}_{L-n}^\top\otimes R&L R
\end{bmatrix}.
\end{align*}
This matrix is positive definite since its Schur complement w.r.t.\ the left upper block is
\begin{align*}
&2L R-2(\mathbbm{1}_{L-n}^\top\otimes R)(I_{L-n}\otimes R)^{-1}(\mathbbm{1}_{L-n}\otimes R)
=2nR\succ0.
\end{align*}
Thus, Problem~\eqref{eq:DD_MPC_NL_nominal} is strongly convex in $\begin{bmatrix}\bar{u}_{[0,L-n-1]}(t)\\u^\rms(t)\end{bmatrix}$ and hence, using $\bar{u}_k(t)=u^\rms(t)$ for $k\in\mathbb{I}_{[L-n,L]}$ due to~\eqref{eq:DD_MPC_NL_nominal_TEC}, it is strongly convex in $\bar{u}_{[0,L]}(t)$.\\
\textbf{(ii).b Bound on $\lVert\bar{u}^*(t)-\tilde{u}(t)\rVert_2^2$}\\
We denote the optimal solution of~\eqref{eq:DD_MPC_NL_nominal} with
\begin{align}\label{eq:prop_continuity_proof_sigma_tilde_def}
\tilde{\sigma}=\tilde{\sigma}_{\varepsilon}\coloneqq\begin{bmatrix}u_{[t-n,t-1]}-u_{[t-n,t-1]}'\\
y_{[t-n,t-1]}-y_{[t-n,t-1]}'\\
\sigma^*(t)+H_{L+n+1}\left(y'(t)-y_{[t-N,t-1]}\right)\alpha^*(t)\end{bmatrix}
\end{align}
by $\tilde{\alpha}(t)$, $\tilde{u}^s(t)$, $\tilde{y}^s(t)$, $\tilde{u}(t)$, $\tilde{y}(t)$, and the corresponding cost by $\tilde{J}_L$.
Since $\check{J}_L^*(\xi_t',\calD_t)\leq\bar{J}$ by assumption and using~\eqref{eq:prop_continuity_cost} by Part (i), we infer $J_L^*(\xi_t)<\infty$, i.e., Problem~\eqref{eq:DD_MPC_NL} is feasible.
The optimal solution of Problem~\eqref{eq:DD_MPC_NL} is a feasible (but not necessarily optimal) solution of~\eqref{eq:DD_MPC_NL_nominal} with $\tilde{\sigma}=\tilde{\sigma}_{\varepsilon}$, i.e., 
\begin{align}\label{eq:prop_continuity_proof_J_tilde_bound}
\tilde{J}_L\leq J_L^*(\xi_t)-\frac{\bar{\lambda}_{\sigma}}{\bar{\varepsilon}^{\beta_{\sigma}}}\lVert\sigma^*(t)\rVert_2^2.
\end{align}
Here, the term $\frac{\bar{\lambda}_{\sigma}}{\bar{\varepsilon}^{\beta_{\sigma}}}\lVert\sigma^*(t)\rVert_2^2$ is due to the fact that, in contrast to Problem~\eqref{eq:DD_MPC_NL}, Problem~\eqref{eq:DD_MPC_NL_nominal} does not include the slack variable $\sigma(t)$ in the cost.
We now construct a candidate solution for Problem~\eqref{eq:DD_MPC_NL} based on the optimal solution of~\eqref{eq:DD_MPC_NL_nominal} with $\tilde{\sigma}=\tilde{\sigma}_{\varepsilon}$.
To this end, we choose $\bar{u}(t)=\tilde{u}(t)$, $\bar{y}(t)=\tilde{y}(t)$, $\alpha(t)=\tilde{\alpha}(t)$, $u^\rms(t)=\tilde{u}^s(t)$, $y^\rms(t)=\tilde{y}^s(t)$, and
\begin{align}\label{eq:prop_continuity_proof_sigma_check}
\sigma(t)=\sigma^*(t)+H_{L+n+1}(\varepsilon^\rmd)(\tilde{\alpha}(t)-\alpha^*(t)).
\end{align}
Note that $(\alpha(t),u^\rms(t),y^\rms(t),\bar{u}(t),\bar{y}(t))$ is a feasible solution of~\eqref{eq:DD_MPC_NL}, whose cost we denote by $\bar{J}_L'$, which satisfies $\bar{J}_L'\geq J_L^*(\xi_t)$ by optimality.
By definition,
\begin{align}\label{eq:prop_continuity_proof_cost_diff}
\bar{J}_L'-\tilde{J}_L=\frac{\bar{\lambda}_{\sigma}}{\bar{\varepsilon}^{\beta_{\sigma}}}\lVert\sigma(t)\rVert_2^2.
\end{align}
Moreover, it follows from the definition of $J_L^*(\xi_t)$ that
\begin{align}\label{eq:prop_continuity_alpha_opt_bound}
\lVert\alpha^*(t)-\alpha^{\rms\rmr}_{\rmlin}(\calD_t)\rVert_2^2\leq&\frac{J_L^*(\xi_t)}{\bar{\lambda}_{\alpha}\bar{\varepsilon}^{\beta_{\alpha}}},\\
\label{eq:prop_continuity_sigma_opt_bound}
\lVert\sigma^*(t)\rVert_2^2\leq&\bar{\varepsilon}^{\beta_{\sigma}}\frac{J_L^*(\xi_t)}{\bar{\lambda}_{\sigma}}.
\end{align}
Similarly, we have
\begin{align}\label{eq:prop_continuity_alpha_tilde_bound}
\lVert\tilde{\alpha}(t)-\alpha^{\rms\rmr}_{\rmlin}(\calD_t)\rVert_2^2&\leq\frac{\tilde{J}_L}{\bar{\lambda}_{\alpha}\bar{\varepsilon}^{\beta_{\alpha}}}
\stackrel{\eqref{eq:prop_continuity_proof_J_tilde_bound}}{\leq}\frac{J_L^*(\xi_t)}{\bar{\lambda}_{\alpha}\bar{\varepsilon}^{\beta_{\alpha}}}.
\end{align}
With the definition of $\sigma(t)$ in~\eqref{eq:prop_continuity_proof_sigma_check}, we obtain
\begin{align}\label{eq:prop_continuity_sigma_diff_bound}
&\quad\lVert\sigma(t)\rVert_2^2-\lVert\sigma^*(t)\rVert_2^2\\\nonumber
&\stackrel{\eqref{eq:ab_ineq},\eqref{eq:prop_continuity_proof_sigma_check}}{\leq}
c_{\sigma,1}^2\bar{\varepsilon}^2\lVert\tilde{\alpha}(t)-\alpha^*(t)\rVert_2^2+2c_{\sigma,1}\bar{\varepsilon}\lVert\tilde{\alpha}(t)-\alpha^*(t)\rVert_2\lVert\sigma^*(t)\rVert_2\\\nonumber
&\stackrel{\eqref{eq:ab_ineq2},\text{\eqref{eq:prop_continuity_alpha_opt_bound}--\eqref{eq:prop_continuity_alpha_tilde_bound}}}{\leq}
c_{\sigma,2}\bar{\varepsilon}^{2-\beta_{\alpha}}J_L^*(\xi_t)+c_{\sigma,3}\bar{\varepsilon}^{1+\frac{\beta_{\sigma}-\beta_{\alpha}}{2}}J_L^*(\xi_t)
\end{align}
for suitably defined $c_{\sigma,i}>0$, $i\in\mathbb{I}_{[1,3]}$.
Recall that
\begin{align}\label{eq:prop_continuity_sandwich_bound}
\tilde{J}_L+\frac{\bar{\lambda}_{\sigma}}{\bar{\varepsilon}^{\beta_{\sigma}}}\lVert\sigma^*(t)\rVert_2^2\stackrel{\eqref{eq:prop_continuity_proof_J_tilde_bound}}{\leq} J_L^*(\xi_t)\leq\bar{J}_L'.
\end{align}
Using that the solution with cost $\bar{J}_L'$ is feasible for Problem~\eqref{eq:DD_MPC_NL}, strong convexity as in Part (ii.a) implies
\begin{align}\label{eq:prop_continuity_proof_strong_convex_u}
\lVert\bar{u}(t)-\bar{u}^*(t)\rVert_2^2\leq c_{\rmu,\rmc}(\bar{J}_L'-J_L^*(\xi_t))
\end{align}
for some $c_{\rmu,\rmc}>0$, compare~\cite[Inequality (11)]{koehler2020nonlinear}.
Together with $\bar{u}(t)=\tilde{u}(t)$, this implies
\begin{align}\label{eq:prop_continuity_proof_strong_convex}
&\lVert\tilde{u}(t)-\bar{u}^*(t)\rVert_2^2\stackrel{\eqref{eq:prop_continuity_proof_strong_convex_u}}{\leq}c_{\rmu,\rmc}(\bar{J}_L'-J_L^*(\xi_t))\\\nonumber
\stackrel{\eqref{eq:prop_continuity_sandwich_bound}}{\leq}&c_{\rmu,\rmc}\left(\bar{J}_L'-\tilde{J}_L-\frac{\bar{\lambda}_{\sigma}}{\bar{\varepsilon}^{\beta_{\sigma}}}\lVert\sigma^*(t)\rVert_2^2\right)\\\nonumber
\stackrel{\eqref{eq:prop_continuity_proof_cost_diff},\eqref{eq:prop_continuity_sigma_diff_bound}}{\leq}
&c_{\rmu,\rmc}\bar{\lambda}_{\sigma}\left(c_{\sigma,2}\bar{\varepsilon}^{2-\beta_{\alpha}-\beta_{\sigma}}+c_{\sigma,3}\bar{\varepsilon}^{\frac{2-\beta_{\alpha}-\beta_{\sigma}}{2}}\right)J_L^*(\xi_t).
\end{align}
\textbf{(ii).c Bound on $\lVert\tilde{u}(t)-\check{u}^*(t)\rVert_2^2$}\\
Using strong convexity (compare Part (ii).a of the proof) and the LICQ (Assumption~\ref{ass:LICQ}), the optimal solution of Problem~\eqref{eq:DD_MPC_NL_nominal} is continuous and piecewise affine in $(\xi_t',\tilde{\sigma})$, compare~\cite{bemporad2002explicit}.
In particular, using that the partition of the optimal solution admits finitely many regions, it is uniformly Lipschitz continuous in $\tilde{\sigma}$.
Moreover, Problem~\eqref{eq:DD_MPC_NL_nominal} is feasible for $\tilde{\sigma}=0$ by assumption, i.e., by $\check{J}_L^*(\xi_t',\calD_t)\leq\bar{J}$, and for $\tilde{\sigma}=\tilde{\sigma}_{\varepsilon}$ since Problem~\eqref{eq:DD_MPC_NL} is feasible, cf. Part (i).
Thus, there exists $c_{\sigma,4}>0$ such that
\begin{align}\label{eq:prop_continuity_proof_c_sigma_tilde}
\lVert \tilde{u}(t)-\check{u}^*(t)\rVert_2&\leq c_{\sigma,4}\lVert \tilde{\sigma}_{\varepsilon}\rVert_2.
\end{align}
The definition of $\tilde{\sigma}_{\varepsilon}$ in~\eqref{eq:prop_continuity_proof_sigma_tilde_def} together with~\eqref{eq:prop_continuity_noise_bound},~\eqref{eq:prop_continuity_alpha_opt_bound},~\eqref{eq:prop_continuity_sigma_opt_bound}, and the triangle inequality implies
\begin{align}\label{eq:prop_continuity_proof_sigma_tilde_bound}
&\lVert\tilde{\sigma}_{\varepsilon}\rVert_2
\leq c_{\sigma,4}\bar{\varepsilon}+\bar{\varepsilon}^{\frac{\beta_{\sigma}}{2}}\sqrt{\frac{J_L^*(\xi_t)}{\bar{\lambda}_\sigma}}
+c_{\sigma,5}\bar{\varepsilon}^{1-\frac{\beta_{\alpha}}{2}}\sqrt{J_L^*(\xi_t)}
\end{align}
for some $c_{\sigma,4},c_{\sigma,5}>0$.
Combining~\eqref{eq:prop_continuity_cost},~\eqref{eq:prop_continuity_proof_c_sigma_tilde},~\eqref{eq:prop_continuity_proof_sigma_tilde_bound}, and $\check{J}_L^*(\xi_t',\calD_t)\leq\bar{J}$, the term $\lVert\tilde{u}(t)-\check{u}^*(t)\rVert_2$ is bounded by $\beta_{\rmu}'(\bar{\varepsilon})$ with some function $\beta_{\rmu}'$.
It is simple to verify that $\beta_{\rmu}'\in\mathcal{K}_{\infty}$ if $\beta_{\alpha}+2\beta_{\sigma}<2$.
Together with~\eqref{eq:prop_continuity_cost},~\eqref{eq:prop_continuity_proof_first_n_steps},~\eqref{eq:prop_continuity_proof_strong_convex}, and $\check{J}_L^*(\xi_t',\calD_t)\leq\bar{J}$, this implies the existence of $\beta_{\rmu}\in\mathcal{K}_{\infty}$ such that~\eqref{eq:prop_continuity} holds.
\end{proof}

\section{Proof of Theorem~\ref{thm:NL_stability}}\label{sec:app_thm_stab_proof}

\begin{proof}
After stating some preliminaries in Part (i), we show in Part (ii) that, if $\Vert\xi_{k+n}-\xi_{k+j}\rVert_2$, $j\in\mathbb{I}_{[0,n-1]}$, is bounded by some constant for $k=t-N,t-N+n,\dots,t-n$, then it is bounded by the same constant for $k=t$.
In combination with Lemma~\ref{lem:pred_error_NL}, this serves as a bound on the ``perturbation'' of the data in Problem~\eqref{eq:DD_MPC_NL}, which we combine with the robustness of data-driven MPC in Section~\ref{subsec:NL_continuity} and the results in our companion paper~\cite{berberich2021linearpart1} to prove practical stability in Part (iii).

For simplicity, we assume w.l.o.g.\ that $\frac{N}{n}\in\mathbb{I}_{\geq0}$, i.e., $N$ is divisible by $n$.
It is straightforward to show that the Lyapunov function candidate $V(\xi_t',\calD_t)$ admits quadratic lower and upper bounds of the form
\begin{align}\label{eq:thm_NL_stab_proof_lower_upper}
c_\rml\lVert\xi_t'-\xi^{\rms\rmr}_{\rmlin}(x_t)\rVert_2^2\leq V(\xi_t',\calD_t)
\leq c_{\rmu}\lVert\xi_t'-\xi^{\rms\rmr}_{\rmlin}(x_t)\rVert_2^2
\end{align}
for some $c_{\rml},c_{\rmu}>0$.
The proof of this fact is provided for completeness in Appendix~\ref{sec:app_affine_MPC} (Theorem~\ref{thm:stab_affine_nominal}).
Note that $\xi_N'$ lies in the set $\{\xi'\mid V(\xi',\mathcal{D}_N)\leq V_{\mathrm{ROA}}\}$, which is compact due to the lower bound~\eqref{eq:thm_NL_stab_proof_lower_upper} and Assumption~\ref{ass:NL_sys}, i.e., compactness of the (linearized) steady-state manifold~\cite[Assumption 5]{berberich2021linearpart1}.
Similar to the proof of~\cite[Proposition 2]{berberich2021linearpart1}, Lipschitz continuity of the dynamics~\eqref{eq:sys_NL} and compactness of $\mathbb{U}$ imply that the union of the $N$-step reachable sets of the linearized and the nonlinear dynamics (compare~\cite{schuermann2018reachset}) starting in $\{\xi'\mid V(\xi',\mathcal{D}_N)\leq V_{\mathrm{ROA}}\}$, which we denote by $X$, is compact.
In the proof, we show that $\{\xi'\mid V(\xi',\mathcal{D}_N)\leq V_{\mathrm{ROA}}\}$ is positively invariant and thus, the bound in Lemma~\ref{lem:pred_error_NL} as well as Lipschitz continuity of the map~\eqref{eq:NL_x_xi_Lipschitz} hold uniformly throughout the proof.
Similarly, whenever we apply Proposition~\ref{prop:continuity}, we use that the bounds hold uniformly for all linearized dynamics considered in the proof (after potential modification of some constants).\\
\textbf{(i) Preliminaries}\\
Let~\eqref{eq:DD_MPC_NL} be feasible at time $t=ni\geq N$ and suppose
\begin{align}\label{eq:thm_NL_stability_recursive_bound}
\lVert\xi_{k+n}-\xi_{k+j}\rVert_2\leq c_{\xi,0}\theta,\quad j\in\mathbb{I}_{[0,n-1]},
\end{align}
for $k=t-N,t-N+n,\dots,t-n$ with some $c_{\xi,0}>0$.
Note that this holds at initial time $t=N$ with $c_{\xi,0}=2$ due to $\lVert\xi_N-\xi_k\rVert_2\leq\theta$ for $k\in\mathbb{I}_{[0,N-1]}$.
Clearly,~\eqref{eq:thm_NL_stability_recursive_bound} implies
\begin{align}\label{eq:thm_NL_stability_proof_xi_bound}
\lVert\xi_t-\xi_k\rVert_2\leq c_{\theta,1}\theta,\quad k\in\mathbb{I}_{[t-N,t-1]},
\end{align}
for some $c_{\theta,1}>0$.
Using Lemma~\ref{lem:pred_error_NL}, the difference between $y_{t+k}$ and $y_k'(t)$ is bounded for $k\in\mathbb{I}_{[-N,-1]}$ by
\begin{align}\label{eq:thm_NL_stab_Delta_bound}
&\lVert y_{t+k}-y_k'(t)\rVert_2=\lVert\Delta_{t,k}\rVert_2
\stackrel{\eqref{eq:lem_pred_error_NL}}{\leq}c_{\Delta}\sum_{i=t-N}^{t-1}\lVert x_t-x_i\rVert_2^2\\\nonumber
&\stackrel{\eqref{eq:NL_x_xi_Lipschitz}}{\leq} c_\rmL c_{\Delta}\sum_{i=t-N}^{t-1}\lVert\xi_t-\xi_i\rVert_2^2
\stackrel{\eqref{eq:thm_NL_stability_proof_xi_bound}}{\leq} c_\rmL c_{\Delta}Nc_{\theta,1}^2\theta^2\eqqcolon c_{\theta,2}\theta^2
\end{align}
for some $c_\rmL>0$, using that $T_\rmL$ in~\eqref{eq:NL_x_xi_Lipschitz} is Lipschitz continuous.

Next, we provide an extended state $\xi_t'=\begin{bmatrix}u_{[t-n,t-1]}'\\y_{[t-n,t-1]}'\end{bmatrix}$ satisfying~\eqref{eq:NL_trafo_xi_x_lin} with a suitable bound on $\lVert\xi_t-\xi_t'\rVert_2$.
We write $\hat{x}_{[t-n,t]}$ for the state trajectory resulting from an application of the closed-loop input $u_{[t-n,t-1]}$ to the initial condition $x_{t-n}$ for the dynamics linearized at $x_t$.
We define the input component $u_{[t-n,t-1]}'$ of $\xi_t'$ as
\begin{align}\label{eq:thm_NL_stab_proof_bar_xi_t_u_bound1}
u_{[t-n,t-1]}'=u_{[t-n,t-1]}+\Gamma_\rmc(x_t)^\dagger(x_t-\hat{x}_t),
\end{align}
where $\Gamma_\rmc(x_t)$ is the controllability matrix of the dynamics linearized at $x_t$.
According to Assumption~\ref{ass:NL_sys}, i.e.,~\cite[Assumption 3]{berberich2021linearpart1}, the Moore-Penrose inverse $\Gamma_\rmc(x_t)^\dagger$ is uniformly bounded.
Further, we define $y_{[t-n,t-1]}'$ as the output trajectory for the dynamics linearized at $x_t$ with input $u_{[t-n,t-1]}'$ and initial condition $x_{t-n}$.
Similar to Lemma~\ref{lem:pred_error_NL}, we can show that
\begin{align}\label{eq:thm_NL_stab_proof_bar_xi_t_x_bound}
\lVert x_t-\hat{x}_t\rVert_2\leq c_{\theta,3}\sum_{j=t-n}^{t-1}\lVert x_t-x_j\rVert_2^2
\stackrel{\eqref{eq:NL_x_xi_Lipschitz},\eqref{eq:thm_NL_stability_recursive_bound}}{\leq}
c_{\theta,4}\theta^2
\end{align}
for some $c_{\theta,3},c_{\theta,4}>0$.
Combining this with~\eqref{eq:thm_NL_stab_proof_bar_xi_t_u_bound1}, we obtain
\begin{align}\label{eq:thm_NL_stab_proof_bar_xi_t_u_bound2}
\lVert u_{[t-n,t-1]}-u_{[t-n,t-1]}'\rVert_2\leq c_{\theta,5}\theta^2
\end{align}
with some $c_{\theta,5}>0$.
We write $\underline{x}_{[t-n,t-1]}$ and $\underline{y}_{[t-n,t-1]}$ for the state and output resulting from an application of $u_{[t-n,t-1]}'$ to the nonlinear system~\eqref{eq:sys_NL} with initial state $x_{t-n}$.
Using Lipschitz continuity of~\eqref{eq:sys_NL} by Assumption~\ref{ass:NL_sys} (i.e.,~\cite[Assumption 1]{berberich2021linearpart1}), there exists $c_{\theta,6}>0$ such that
\begin{align}\label{eq:thm_NL_stab_proof_bar_xi_t_y_bound1}
\lVert y_{[t-n,t-1]}-\underline{y}_{[t-n,t-1]}\rVert_2\leq c_{\theta,6}\lVert u_{[t-n,t-1]}-u_{[t-n,t-1]}'\rVert_2.
\end{align}
Using again similar arguments as in Lemma~\ref{lem:pred_error_NL}, we can show
\begin{align}\nonumber
&\lVert\underline{y}_{[t-n,t-1]}-y_{[t-n,t-1]}'\rVert_2\leq c_{\theta,7}\sum_{j=t-n}^{t-1}\lVert x_t-\underline{x}_j\rVert_2^2\\\nonumber
\stackrel{\eqref{eq:ab_ineq2}}{\leq}
&2c_{\theta,7}\sum_{j=t-n}^{t-1}(\lVert x_t-x_j\rVert_2^2+\lVert x_j-\underline{x}_j\rVert_2^2)\\\label{eq:thm_NL_stab_proof_bar_xi_t_y_bound2}
\stackrel{\eqref{eq:NL_x_xi_Lipschitz},\eqref{eq:thm_NL_stability_recursive_bound}}{\leq}
&c_{\theta,8}\theta^2+c_{\theta,9}\lVert u_{[t-n,t-1]}-u_{[t-n,t-1]}'\rVert_2^2,
\end{align}
for suitably defined $c_{\theta,i}>0$, $i\in\mathbb{I}_{[7,9]}$, where the last inequality also uses Lipschitz continuity of the nonlinear dynamics~\eqref{eq:sys_NL}.
Combining~\eqref{eq:thm_NL_stab_proof_bar_xi_t_u_bound2}--\eqref{eq:thm_NL_stab_proof_bar_xi_t_y_bound2} and letting $\bar{\theta}<1$ such that $\theta^4<\theta^2$, there exists $c_{\theta,10}>0$ such that
\begin{align}\label{eq:thm_NL_stab_proof_xi_bar_bound}
\lVert\xi_t-\xi_t'\rVert_2\leq c_{\theta,10}\theta^2.
\end{align}
\textbf{(ii) Bound on $\Vert\xi_{t+n}-\xi_{t+j}\rVert_2$, $i\in\mathbb{I}_{[0,n-1]}$}\\
In this part of the proof, we show that, under suitable conditions,~\eqref{eq:thm_NL_stability_recursive_bound} holds for $k=t$ if it holds for $k=t-N,t-N+n,\dots,t-n$.
We denote the extended state~\eqref{eq:xi_def_new} corresponding to $(u^{\rms*}(t),y^{\rms*}(t))$ by $\xi^{\rms*}(t)$.
For $j\in\mathbb{I}_{[0,n-1]}$, it holds that
\begin{align}\label{eq:thm_NL_stab_proof_bound_xi1}
\lVert\xi_{t+n}-\xi_{t+j}\rVert_2
\leq&\lVert\xi_{t+n}-\xi^{\rms*}(t)\rVert_2+\lVert\xi_{t+j}-\xi^{\rms*}(t)\rVert_2\\\nonumber
\leq&c_{\xi,1}(\lVert\xi_{t+n}-\xi^{\rms*}(t)\rVert_2+\lVert\xi_t-\xi^{\rms*}(t)\rVert_2)
\end{align}
for some $c_{\xi,1}>0$, where we use that $\xi_{t+n}$ and $\xi_t$ contain all elements of $\xi_{t+j}$.
Using optimality in Problem~\eqref{eq:DD_MPC_NL}, we have
\begin{align}\label{eq:thm_NL_stab_proof_bound_xi2}
\lVert\xi_t-\xi^{\rms*}(t)\rVert_2^2\leq\frac{1}{\lambda_{\min}(Q,R)}J_L^*(\xi_t).
\end{align}
Moreover, using $u_{t+k}=\bar{u}_k^*(t)$ for $k\in\mathbb{I}_{[0,n-1]}$ due to the $n$-step scheme, cf. Algorithm~\ref{alg:MPC_n_step_NL}, we can derive
\begin{align}\nonumber
\lVert\xi_{t+n}-\xi^{\rms*}(t)\rVert_2
=&\sum_{k=0}^{n-1}\lVert u_{t+k}-u^{\rms*}(t)\rVert_2+\lVert y_{t+k}-y^{\rms*}(t)\rVert_2\\\nonumber
\leq&\sum_{k=0}^{n-1}\lVert\bar{u}_k^*(t)-u^{\rms*}(t)\rVert_2+\lVert y_{t+k}-\bar{y}_k^*(t)\rVert_2\\\nonumber
&\quad+\lVert\bar{y}_k^*(t)-y^{\rms*}(t)\rVert_2\\\nonumber
\leq&\sum_{k=0}^{n-1}\lVert y_{t+k}-y_{k}'(t)\rVert_2+\lVert y_k'(t)-\bar{y}_k^*(t)\rVert_2\\\label{eq:thm_NL_stab_proof_bound_xi3}
&+\frac{1}{\sqrt{\lambda_{\min}(Q,R)}}\sqrt{J_L^*(\xi_t)}.
\end{align}
Analogously to~\eqref{eq:thm_NL_stab_Delta_bound}, there exists $c_{\xi,2}>0$ such that
\begin{align}\label{eq:thm_NL_stab_proof_bound_xi4}
\sum_{k=0}^{n-1}\lVert y_{t+k}-y_k'(t)\rVert_2&\leq c_{\xi,2}\sum_{k=-N}^{n-1}\lVert \xi_{t+k}-\xi_t\rVert_2^2\\\nonumber
&\stackrel{\eqref{eq:thm_NL_stability_proof_xi_bound}}{\leq} c_{\xi,2}\sum_{k=0}^{n-1}\lVert \xi_{t+k}-\xi_t\rVert_2^2+c_{\xi,2}Nc_{\theta,1}^2\theta^2.
\end{align}
We now use a bound on the deviation between the predicted output $\bar{y}^*(t)$ in Problem~\eqref{eq:DD_MPC_NL} and the output of the linearized dynamics $y'(t)$.
This bound is shown for data-driven MPC of LTI systems in~\cite[Lemma 2]{berberich2021guarantees} but the result remains true for affine systems since only differences between trajectories are considered.
With the ``noise bound'' $\bar{\varepsilon}=c_{\theta,2}\theta^2$ (cf.~\eqref{eq:thm_NL_stab_Delta_bound}),~\cite[Lemma 2]{berberich2021guarantees} implies
\begin{align}\label{eq:thm_NL_stab_proof_bound_xi5}
\lVert y_k'(t)-\bar{y}_{k}^*(t)\rVert_2
\leq &c_{\xi,3}\theta^2\lVert\alpha^*(t)\rVert_2+c_{\xi,4}\theta^2+c_{\xi,5}\lVert\sigma^*(t)\rVert_2
\end{align}
for $k\in\mathbb{I}_{[0,n-1]}$ with some $c_{\xi,i}>0$, $i\in\mathbb{I}_{[4,6]}$.
Finally, using the definition of $J_L^*(\xi_t)$, we have
\begin{align}\label{eq:thm_NL_stab_proof_alpha_bound}
\lVert\alpha^*(t)\rVert_2&\leq\sqrt{\frac{1}{\lambda_{\alpha}}J_L^*(\xi_t)}+\lVert\alpha^{\rms\rmr}_{\rmlin}(\calD_t)\rVert_2,\\\label{eq:thm_NL_stab_proof_sigma_bound}
\lVert\sigma^*(t)\rVert_2&\leq\sqrt{\frac{1}{\lambda_{\sigma}}J_L^*(\xi_t)},
\end{align}
where $\lVert\alpha^{\rms\rmr}_{\rmlin}(\calD_t)\rVert_2$ is uniformly bounded by assumption.

Let now $\lambda_{\alpha}=\bar{\lambda}_{\alpha}\theta^{2\beta_{\alpha}}$ and $\lambda_{\sigma}=\frac{\bar{\lambda}_{\sigma}}{\theta^{2\beta_{\sigma}}}$ for some $\bar{\lambda}_{\alpha},\bar{\lambda}_{\sigma},\beta_{\alpha},\beta_{\sigma}>0$ with $\beta_{\alpha}>1$, $\beta_{\alpha}+2\beta_{\sigma}<2$, compare Proposition~\ref{prop:continuity}.
There exist $c_{\xi,6},c_{\xi,7}>0$ such that, for $k\in\mathbb{I}_{[0,n-1]}$,
\begin{align}\label{eq:thm_NL_stab_proof_bound_xi6}
\lVert y_{k}'(t)-\bar{y}_k^*(t)\rVert_2\leq &c_{\xi,6}\theta^2+c_{\xi,7}(\theta^{2-\beta_{\alpha}}+\theta^{\beta_{\sigma}}) \sqrt{J_L^*(\xi_t)}.
\end{align}
Combining~\eqref{eq:thm_NL_stab_proof_bound_xi1}--\eqref{eq:thm_NL_stab_proof_bound_xi4} and~\eqref{eq:thm_NL_stab_proof_bound_xi6}, we infer for $j\in\mathbb{I}_{[0,n-1]}$
\begin{align}\label{eq:thm_NL_stab_proof_bound_xi7}
&\lVert\xi_{t+n}-\xi_{t+j}\rVert_2\\\nonumber
\leq&c_{\xi,1}\Bigg(\frac{2}{\sqrt{\lambda_{\min}(Q,R)}}\sqrt{J_L^*(\xi_t)}
+c_{\xi,2}\sum_{k=0}^{n-1}\lVert \xi_{t+k}-\xi_t\rVert_2^2\\\nonumber
&+c_{\xi,2}Nc_{\theta,1}^2\theta^2+n\bigg(c_{\xi,6}\theta^2+c_{\xi,7}(\theta^{2-\beta_{\alpha}}+\theta^{\beta_{\sigma}})\sqrt{J_L^*(\xi_t)}\bigg)\Bigg).
\end{align}
Let now 
\begin{align}\label{eq:thm_NL_stab_proof_S_V_ROA_def}
\bar{S}=\theta^4,\>\>V_{\mathrm{ROA}}=\theta^{2+\eta}
\end{align}
for some $0<\eta<2(\beta_{\alpha}-1)$ and suppose $\bar{\theta}<1$.
This implies $\theta^{\eta_1}<\theta^{\eta_2}$ if $\eta_1>\eta_2$, which will be used throughout the proof.
Using Assumption~\ref{ass:NL_sys}, i.e.,~\cite[Assumptions 4 and 5]{berberich2021linearpart1}, $J_{\mathrm{eq},\rmlin}^*(x_t)$ is uniformly bounded by $\bar{S}$, i.e.,
\begin{align}\label{eq:thm_NL_stab_proof_J_eq_bound}
J_{\mathrm{eq},\rmlin}^*(x_t)\leq c_{\rmJ,S}\bar{S}=c_{\rmJ,S}\theta^4
\end{align}
with $c_{\rmJ,S}>0$.
Hence, $\check{J}_L^*(\bar{\xi}_t,\mathcal{D}_t)\leq V_{\mathrm{ROA}}+J_{\mathrm{eq},\rmlin}^*(x_t)$ implies
\begin{align}\label{eq:thm_NL_stab_proof_J_nom_theta_bound}
\check{J}_L^*(\xi_t',\mathcal{D}_t)\leq \check{c}_{\rmJ,\theta}\theta^{2+\eta}
\end{align}
for some $\check{c}_{\rmJ,\theta}>0$.
We now apply Proposition~\ref{prop:continuity}, where the noise / disturbance is bounded as in~\eqref{eq:prop_continuity_noise_bound} by $\bar{\varepsilon}=c_{\theta}\theta^2$ for some $c_{\theta}>0$, compare~\eqref{eq:thm_NL_stab_Delta_bound} for the ``noise'' entering the output Hankel matrix and~\eqref{eq:thm_NL_stab_proof_xi_bar_bound} for the ``disturbance'' in the initial conditions.
Inequalities~\eqref{eq:thm_NL_stab_proof_J_eq_bound} and~\eqref{eq:thm_NL_stab_proof_J_nom_theta_bound} together with Proposition~\ref{prop:continuity} imply
\begin{align}\label{eq:thm_NL_stab_proof_J_theta_bound}
J_L^*(\xi_t)&\leq \left(1+c_{\rmJ,\rma}(c_{\theta}\theta^2)^{\beta_{\sigma}}\right)
\check{c}_{\rmJ,\theta}\theta^{2+\eta}\\\nonumber
&\quad+c_{\rmJ,\rmb}(c_{\theta}\theta^2)^{2-\beta_{\sigma}}\left(1+\lVert H_{ux,t}^\dagger\rVert_2^2(c_{\theta}\theta^2)^{\beta_{\alpha}}\right)
\\\nonumber
&\leq c_{\rmJ,\theta}\theta^{2+\eta}
\end{align}
for some $c_{\rmJ,\theta}>0$, where we use $\lVert H_{ux,t+n}^\dagger\rVert_2^2\leq \frac{\bar{c}_{\rmH}^2}{\theta^2}$ by assumption, $\bar{\theta}<1$ as well as
\begin{align*}
0<\eta<2(\beta_{\alpha}-1)<2(1-2\beta_{\sigma})<2(1-\beta_{\sigma}).
\end{align*}
Plugging this into~\eqref{eq:thm_NL_stab_proof_bound_xi7}, we obtain
\begin{align}\label{eq:thm_NL_stab_proof_bound_xi8}
\lVert\xi_{t+n}-\xi_{t+j}\rVert_2\leq\beta_1(\theta)+c_{\xi,1}c_{\xi,2}\sum_{k=0}^{n-1}\lVert \xi_{t+k}-\xi_t\rVert_2^2
\end{align}
for a suitably defined $\beta_1\in\mathcal{K}_{\infty}$ containing only terms with order strictly larger than $1$.
Each of the terms $\lVert\xi_{t+k}-\xi_t\rVert_2$, $k\in\mathbb{I}_{[0,n-1]}$ can be bounded analogously to $\lVert\xi_{t+n}-\xi_{t+j}\rVert_2$ in~\eqref{eq:thm_NL_stab_proof_bound_xi8}, using the same bounds as above leading to~\eqref{eq:thm_NL_stab_proof_bound_xi8}.
To be precise, following the same steps as above, there exist $\hat{\beta}\in\mathcal{K}_{\infty}$, $\hat{c}_{\xi}>0$ such that, for all $k\in\mathbb{I}_{[0,n-1]}$,
\begin{align*}
\lVert\xi_{t+k}-\xi_t\rVert_2^2\leq\hat{\beta}(\theta)+\hat{c}_{\xi}\sum_{s=0}^{k-1}\lVert\xi_{t+s}-\xi_t\rVert_2^2,
\end{align*}
where $\hat{\beta}$ contains only terms with order strictly larger than $1$.
Applying this bound recursively $n-1$ times and plugging the result into~\eqref{eq:thm_NL_stab_proof_bound_xi8}, we obtain
\begin{align}\label{eq:thm_NL_stab_proof_bound_xin_K}
\lVert\xi_{t+n}-\xi_{t+j}\rVert_2
\leq \beta_2(\theta)
\end{align}
for $j\in\mathbb{I}_{[0,n-1]}$ with $\beta_2\in\mathcal{K}_{\infty}$ containing only terms with order strictly larger than $1$.
Hence, if $\bar{\theta}$ is sufficiently small such that $\beta_2(\theta)\leq c_{\xi,0}\theta$, then~\eqref{eq:thm_NL_stability_recursive_bound} holds for $k=t$.
To conclude, we have shown that~\eqref{eq:thm_NL_stability_recursive_bound} holds recursively, assuming that $V(\xi_t',\mathcal{D}_t)\leq V_{\mathrm{ROA}}$.
Thus, using~\eqref{eq:thm_NL_stability_recursive_bound} for $k=t-N+n,t-N+2n,\dots,t$, we infer (cf.~\eqref{eq:thm_NL_stability_proof_xi_bound}, \eqref{eq:thm_NL_stab_proof_xi_bar_bound})
\begin{align}\label{eq:thm_NL_stab_proof_data_bound}
\lVert \xi_{t+n}-\xi_k\rVert_2&\leq c_{\theta,1}\theta,\quad k\in\mathbb{I}_{[t+n-N,t+n-1]},\\
\label{eq:thm_NL_stab_proof_disturbance_bound}
\lVert\xi_{t+n}'-\xi_{t+n}\rVert_2&\leq c_{\theta,10}\theta^2.
\end{align}
With $c_{\theta,2}$ as in~\eqref{eq:thm_NL_stab_Delta_bound}, we have for $k\in\mathbb{I}_{[-N,-1]}$
\begin{align}\label{eq:thm_NL_stab_Delta_bound_tn}
\lVert y_{t+n+k}-y_k'(t+n)\rVert_2&\stackrel{\eqref{eq:lem_pred_error_NL}}{\leq} c_{\Delta}\sum_{i=t+n-N}^{t+n-1}\lVert x_{t+n}-x_i\rVert_2^2\leq c_{\theta,2}\theta^2,
\end{align}
i.e.,~\eqref{eq:thm_NL_stab_Delta_bound} also holds recursively with $t$ replaced by $t+n$.\\
\textbf{(iii) Invariance and Lyapunov function decay}\\
\textbf{(iii).a Application of model-based results from~\cite{berberich2021linearpart1}}\\
Recall that the prediction error due to the inexact model in Problem~\eqref{eq:DD_MPC_NL} at time $t$ (compared to the nominal MPC problem with cost $\check{J}_L^*(\xi_t',\mathcal{D}_t)$) is induced by output measurement noise and perturbed initial conditions with bound $c_{\theta}\theta^2$ for some $c_{\theta}>0$ (compare~\eqref{eq:thm_NL_stab_Delta_bound} and~\eqref{eq:thm_NL_stab_proof_xi_bar_bound}).
According to Proposition~\ref{prop:continuity}, this noise translates into an input disturbance for the resulting closed loop with bound $\beta_{\rmu}(c_{\theta}\theta^2)$, i.e.,
\begin{align}\label{eq:thm_NL_stab_part_iii_u_nom_bound}
\lVert\check{u}_{[0,n-1]}^*(t)-u_{[t,t+n-1]}\rVert_2&\leq \beta_{\rmu}(c_{\theta}\theta^2),
\end{align}
where $\check{u}^*(t)$ is the nominal optimal input corresponding to the optimal cost $\check{J}_L^*(\xi_t',\mathcal{D}_t)$.
Since the nominal MPC problem~\eqref{eq:DD_MPC_NL_nominal} with $\tilde{\sigma}=0$ contains an exact model of the linearization, we can apply the main technical result\footnote{The result can be applied despite the fact that the cost of Problem~\eqref{eq:DD_MPC_NL_nominal} depends on the output, i.e., is in general only positive semidefinite in the state, since we perform an $n$-step analysis and the cost is positive definite over $n$ steps.}
in our companion paper~\cite[Proposition 2]{berberich2021linearpart1} to arrive at
\begin{align}\label{eq:thm_NL_stab_proof_Lyapunov_decay1}
&V(\xi_{t+n}',\mathcal{D}_{t+n})\leq c_{\rmV,1}V(\xi_t',\mathcal{D}_t)\\\nonumber
&+\bar{\lambda}_{\alpha}(c_{\theta}\theta^2)^{\beta_{\alpha}}\lVert\tilde{\alpha}(t+n)-\alpha^{\rms\rmr}_{\rmlin}(\calD_{t+n})\rVert_2^2+c_{\rmd,1}\beta_{\rmu}(c_{\theta}\theta^2)^2
\end{align}
for some $0<c_{\rmV,1}<1$, $c_{\rmd,1}>0$.
Here, $\tilde{\alpha}(t+n)$ is a later specified candidate solution corresponding to the optimal control problem with optimal cost $\check{J}_L^*(\xi_{t+n}',\mathcal{D}_{t+n})$.
The additional term depending on $\tilde{\alpha}(t+n)$ is due to the regularization of $\alpha$ in the cost, which is not present in the model-based MPC scheme in~\cite{berberich2021linearpart1}.
Further, the term $c_{\rmd,1}\beta_{\rmu}(c_{\theta}\theta^2)^2$ is due to the fact that $x_{t+n}$ results from applying the optimal input $u_{[t,t+n-1]}=\bar{u}_{[0,n-1]}^*(t)$ of Problem~\eqref{eq:DD_MPC_NL} to the state $x_t$, whereas~\cite[Proposition 2]{berberich2021linearpart1} considers the closed loop under the nominal MPC input $\check{u}_{[0,n-1]}^*(t)$.
To be more precise, with minor adaptations of the proof of~\cite[Proposition 2]{berberich2021linearpart1}, it can be shown that, if the closed-loop state at time $t+n$ results from applying an input which differs from $\check{u}^*(t)$ by no more than $\beta_{\rmu}(c_{\theta}\theta^2)$ (compare~\eqref{eq:thm_NL_stab_part_iii_u_nom_bound}), then the Lyapunov function decay shown in~\cite{berberich2021linearpart1} remains true if the squared disturbance bound $c_{\rmd,1}\beta_{\rmu}(c_{\theta}\theta^2)^2$ is added on the right-hand side (i.e.,~\eqref{eq:thm_NL_stab_proof_Lyapunov_decay1} holds).
This is possible since the main proof idea of~\cite[Proposition 2]{berberich2021linearpart1} does not rely on the exact state value at time $t+n$, but rather on the fact that it remains close to $x_t$.\\
\textbf{(iii).b Bound on $\lVert\tilde{\alpha}(t+n)-\alpha^{\rms\rmr}_{\rmlin}(\calD_{t+n})\rVert_2^2$}\\
We bound this term using the candidate solution of the model-based MPC in~\cite{berberich2021linearpart1}.
We write $\tilde{u}(t+n)$ for a candidate input used in the proof of~\cite[Proposition 2]{berberich2021linearpart1}, which corresponds to a feasible candidate solution to Problem~\eqref{eq:DD_MPC_NL_nominal} with $\tilde{\sigma}=0$ for  $\tilde{\alpha}(t+n)=H_{ux,t+n}^\dagger\begin{bmatrix}\tilde{u}(t+n)^\top&x_t^\top&1\end{bmatrix}^\top$.
Thus, we infer
\begin{align}\label{eq:thm_NL_stab_part_iii_alpha_bound1}
&\lVert\tilde{\alpha}(t+n)-\alpha^{\rms\rmr}_{\rmlin}(\calD_{t+n})\rVert_2^2\\\nonumber
\leq&\lVert H_{ux,t+n}^\dagger\rVert_2^2\Big(\lVert\tilde{u}(t+n)-\bbone_{L+n+1}\otimes u^{\rms\rmr}_{\rmlin}(x_{t+n})\rVert_2^2\\\nonumber
&+\lVert x_t-x^{\rms\rmr}_{\rmlin}(x_{t+n})\rVert_2^2\Big).
\end{align}
The analysis in~\cite[Proposition 2]{berberich2021linearpart1} implies that both candidate solutions used in~\cite{berberich2021linearpart1} satisfy
\begin{align}\label{eq:thm_NL_stab_part_iii_alpha_bound2}
\lVert\tilde{u}(t+n)-\bbone_{L+n+1}\otimes u^{\rms\rmr}_{\rmlin}(x_{t+n})\rVert_2^2
\leq c_{\rmV,2}V(\xi_t',\mathcal{D}_t)
\end{align}
for some $c_{\rmV,2}>0$.
Further, there exist $c_{\rmV,3},c_{\rmV,4}>0$ such that
\begin{align}\label{eq:thm_NL_stab_part_iii_alpha_bound3}
&\lVert x_t-x^{\rms\rmr}_{\rmlin}(x_{t+n})\rVert_2^2\\\nonumber
\stackrel{\eqref{eq:ab_ineq2}}{\leq}&2\lVert x_t-x_{t+n}\rVert_2^2+2\lVert x_{t+n}-x^{\rms\rmr}_{\rmlin}(x_{t+n})\rVert_2^2\\\nonumber
\stackrel{\eqref{eq:NL_x_xi_Lipschitz},\eqref{eq:thm_NL_stab_proof_data_bound},\eqref{eq:thm_NL_stab_proof_lower_upper}}{\leq}
&c_{\rmV,3}\theta^2+c_{\rmV,4}V(\xi_{t+n}',\mathcal{D}_{t+n}).
\end{align}
Assumption~\ref{ass:closed_loop_pe} and $c_{\rmH}=\frac{\bar{c}_{\rmH}}{\theta}$ imply $\lVert H_{ux,t+n}^\dagger\rVert_2^2\leq \frac{\bar{c}_{\rmH}^2}{\theta^2}$.
Plugging~\eqref{eq:thm_NL_stab_part_iii_alpha_bound1}--\eqref{eq:thm_NL_stab_part_iii_alpha_bound3} into~\eqref{eq:thm_NL_stab_proof_Lyapunov_decay1}, we thus infer
\begin{align}\label{eq:thm_NL_stab_proof_Lyapunov_decayx}
&V(\xi_{t+n}',\mathcal{D}_{t+n})\\\nonumber
\leq&(c_{\rmV,1}+c_{\rmV,5}\theta^{2\beta_{\alpha}-2})V(\xi_t',\mathcal{D}_t)+c_{\rmd,1}\beta_{\rmu}(c_{\theta}\theta^2)^2\\\nonumber
&+c_{\rmV,6}\theta^{2\beta_{\alpha}}+c_{\rmV,7}\theta^{2\beta_{\alpha}-2}V(\xi_{t+n}',\mathcal{D}_{t+n})
\end{align}
for suitably defined $c_{V,i}>0$, $i\in\mathbb{I}_{[5,7]}$.
Given that $\beta_{\alpha}>1$, $c_{\rmV,1}<1$, we can choose $\bar{\theta}$ sufficiently small such that
\begin{align*}
c_{\rmV,7}\theta^{2\beta_{\alpha}-2}<1,\>\>
\frac{c_{\rmV,1}+c_{\rmV,5}\theta^{2\beta_{\alpha}-2}}{1-c_{\rmV,7}\theta^{2\beta_{\alpha}-2}}<1.
\end{align*}
Multiplication of both sides of~\eqref{eq:thm_NL_stab_proof_Lyapunov_decayx} by $\frac{1}{1-c_{\rmV,7}\theta^{2\beta_{\alpha}-2}}$ yields
\begin{align}\label{eq:thm_NL_stab_proof_Lyapunov_decay2}
V(\xi_{t+n}',\mathcal{D}_{t+n})&\leq c_{\rmV,8}V(\xi_t',\mathcal{D}_t)\\\nonumber
&\quad+c_{\rmV,9}(\beta_{\rmu}(c_{\theta}\theta^2)^2+\theta^{2\beta_{\alpha}})
\end{align}
for suitably defined $0<c_{\rmV,8}<1$, $c_{\rmV,9}>0$.\\
\textbf{(iii).c Practical stability}\\
Plugging~\eqref{eq:thm_NL_stab_proof_J_theta_bound} and $\bar{\varepsilon}=c_{\theta}\theta^2$ into~\eqref{eq:prop_continuity_proof_strong_convex}--\eqref{eq:prop_continuity_proof_sigma_tilde_bound}, straightforward algebraic calculations reveal that~\eqref{eq:prop_continuity} also holds with 
\begin{align*}
\beta_{\rmu}(c_{\theta}\theta^2)\leq c_{\rmV,10}\left(\theta^{2-\frac{\beta_{\alpha}+\beta_{\sigma}}{2}+\frac{\eta}{2}}+\theta^2+\theta^{1+\beta_{\sigma}+\frac{\eta}{2}}\right)\eqqcolon\beta_{\rmV,1}(\theta)
\end{align*}
for some $c_{\rmV,10}>0$.
Using $\beta_{\alpha}+2\beta_{\sigma}<2$ and $\eta<2(\beta_{\alpha}-1)$, it is straightforward to verify that the smallest exponent appearing in $(\beta_{\rmV,1}(\theta))^2+\theta^{2\beta_{\alpha}}$ is strictly larger than $2+\eta$.
This implies that, if $\bar{\theta}$ is sufficiently small, then~\eqref{eq:thm_NL_stab_proof_Lyapunov_decay2} yields $V(\xi_{t+n}',\mathcal{D}_{t+n})\leq \theta^{2+\eta}=V_{\mathrm{ROA}}$.
Hence, all bounds in this proof hold recursively for $t=N+ni$, $i\in\mathbb{I}_{\geq0}$.
In particular, the nominal MPC problem with cost $\check{J}_L^*(\xi_t,\mathcal{D}_t)$ and, using Proposition~\ref{prop:continuity}, Problem~\eqref{eq:DD_MPC_NL} are recursively feasible.
Finally, a recursive application of~\eqref{eq:thm_NL_stab_proof_Lyapunov_decay2} yields 
\begin{align}
V(\xi_t',\mathcal{D}_{t})\leq c_{\rmV,8}^iV(\xi_N',\mathcal{D}_N)+\beta_{\rmV,3}(\theta)
\end{align}
for some $\beta_{\rmV,3}\in\mathcal{K}_{\infty}$ and any $t=N+ni$, $i\in\mathbb{I}_{\geq0}$.
Exploiting the lower and upper bounds in~\eqref{eq:thm_NL_stab_proof_lower_upper}, we obtain
\begin{align*}
\lVert\xi_t'-\xi^{\rms\rmr}_{\rmlin}(x_{t})\rVert_2^2\leq\frac{c_\rmu}{c_\rml}c_{\rmV,8}^i\lVert\xi_N'-\xi^{\rms\rmr}_{\rmlin}(x_N)\rVert_2^2+\frac{1}{c_\rml}\beta_{\rmV,3}(\theta).
\end{align*}
Using Assumption~\ref{ass:NL_manifold_convex}, i.e.,~\eqref{eq:ass_NL_manifold_convex}, as well as~\eqref{eq:thm_NL_stab_proof_xi_bar_bound}, we obtain~\eqref{eq:thm_NL_stability_decay} with $c_{\rmV}=c_{\rmV,8}$ and some $C>0$, $\beta_{\theta}\in\mathcal{K}_{\infty}$.
\end{proof}

\begin{IEEEbiography}[{\includegraphics[width=1in,height=1.25in,clip,keepaspectratio]{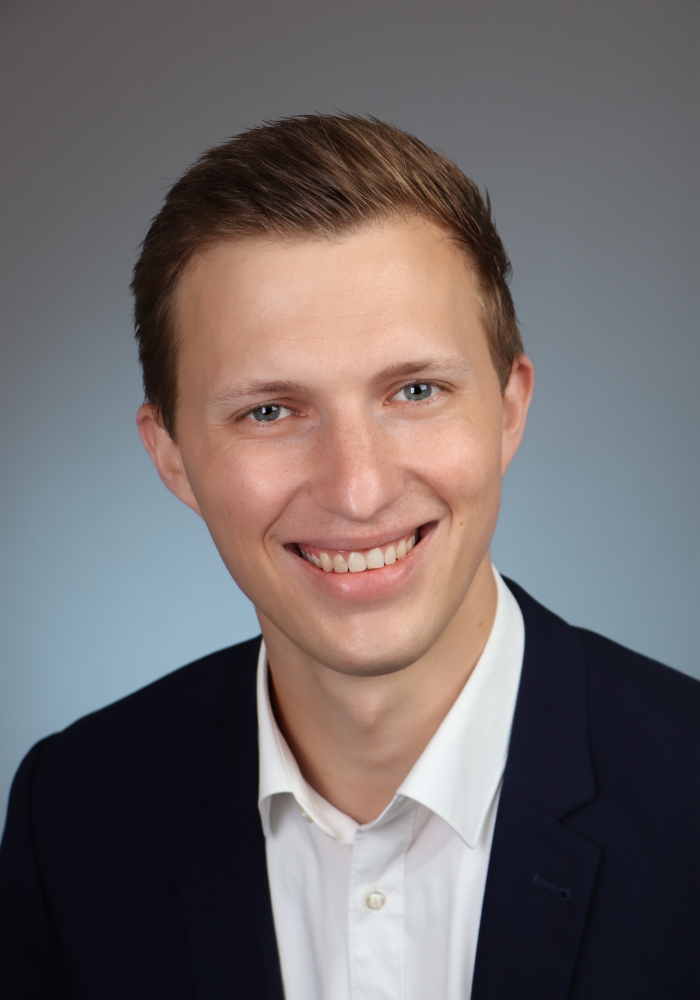}}]{Julian Berberich}
received the Master’s degree in Engineering Cybernetics from the University of Stuttgart, Germany, in 2018. Since 2018, he has been a Ph.D. student at the Institute for Systems Theory and Automatic Control under supervision of Prof. Frank Allg\"ower and a member of the International Max-Planck Research School (IMPRS) at the University of Stuttgart. He has received the Outstanding Student Paper Award at the 59th Conference on Decision and Control in 2020. His research interests are in the area of data-driven analysis and control.
\end{IEEEbiography}

\begin{IEEEbiography}[{\includegraphics[width=1in,height=1.25in,clip,keepaspectratio]{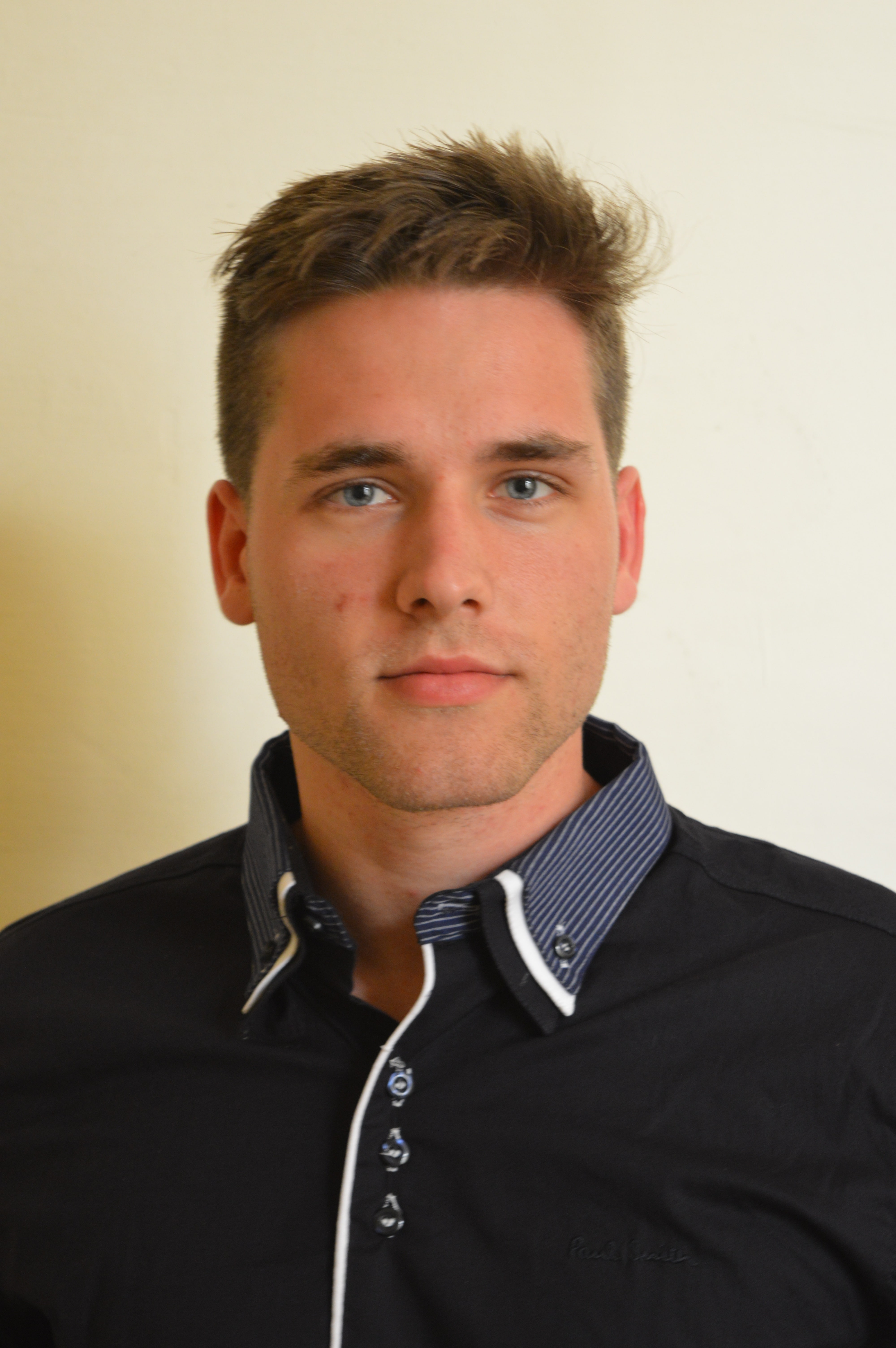}}]{Johannes K\"ohler}
received his Master degree in Engineering Cybernetics from the University of Stuttgart, Germany, in 2017.
In 2021, he obtained a Ph.D. in Mechanical Engineering, also from the University of Stuttgart, Germany. 
Since then, he is a postdoctoral researcher at the \textit{Institute for Dynamic Systems and Control} at ETH Zürich. 
His research interests are in the area of model predictive control and the control of nonlinear uncertain systems.
 \end{IEEEbiography}

\begin{IEEEbiography}[{\includegraphics[width=1in,height=1.25in,clip,keepaspectratio]{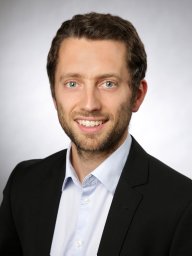}}]{Matthias A. M\"uller}
received a Diploma degree in Engineering Cybernetics from the University of Stuttgart, Germany, and an M.S. in Electrical and Computer Engineering from the University of Illinois at Urbana-Champaign, US, both in 2009.
In 2014, he obtained a Ph.D. in Mechanical Engineering, also from the University of Stuttgart, Germany, for which he received the 2015 European Ph.D. award on control for complex and heterogeneous systems. Since 2019, he is director of the Institute of Automatic Control and full professor at the Leibniz University Hannover, Germany. 
He obtained an ERC Starting Grant in 2020 and is recipient of the inaugural Brockett-Willems Outstanding Paper Award for the best paper published in Systems \& Control Letters in the period 2014-2018. His research interests include nonlinear control and estimation, model predictive control, and data-/learning-based control, with application in different fields including biomedical engineering.
\end{IEEEbiography}

\begin{IEEEbiography}[{\includegraphics[width=1in,height=1.25in,clip,keepaspectratio]{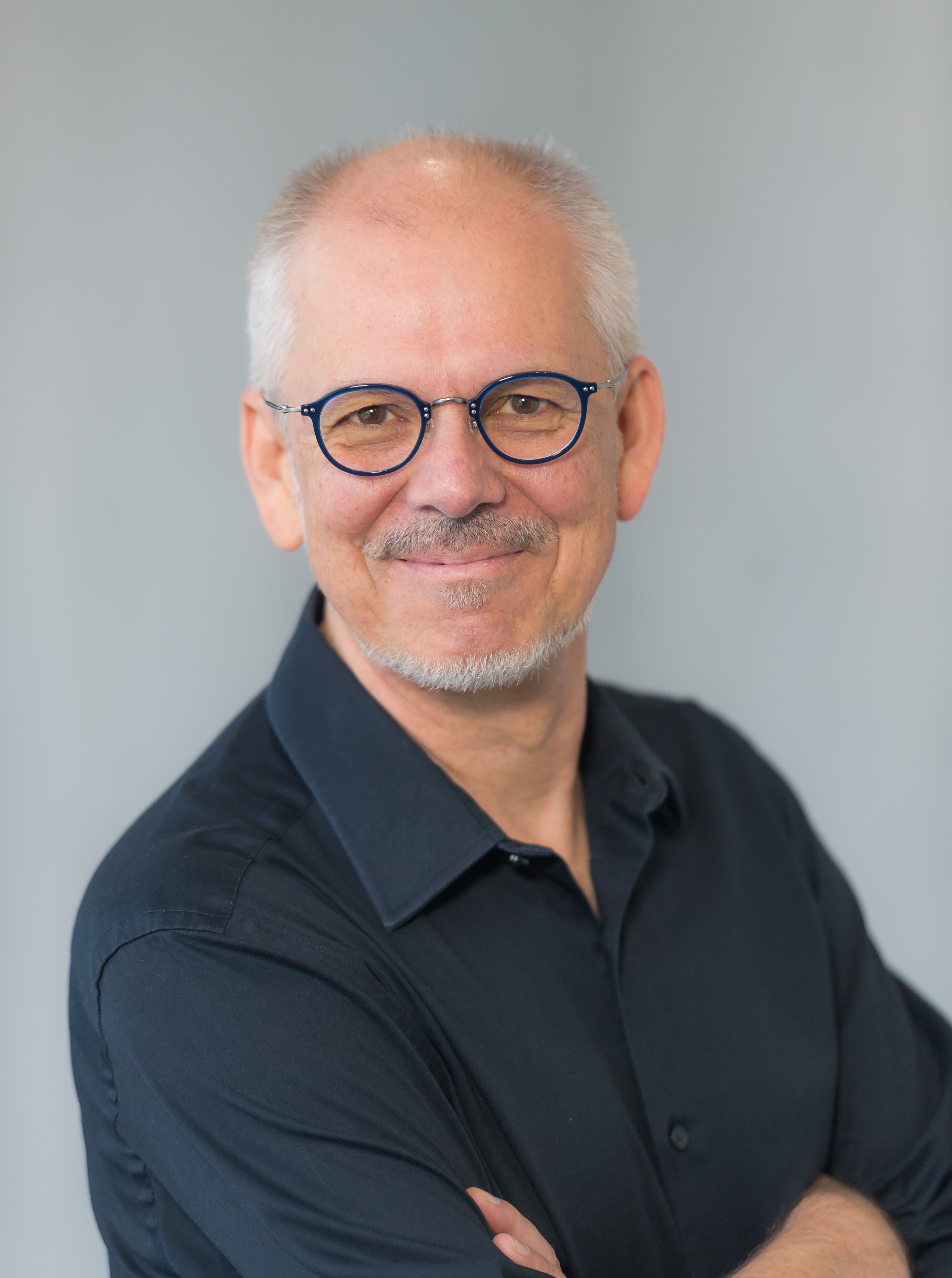}}]{Frank Allg\"ower}
is professor of mechanical engineering at the University of Stuttgart, Germany, and Director of the Institute for Systems Theory and Automatic Control (IST) there.\\ 
Frank is active in serving the community in several roles: Among others he has been President of the International Federation of Automatic Control (IFAC) for the years 2017-2020, Vice-president for Technical Activities of the IEEE Control Systems Society for 2013/14, and Editor of the journal Automatica from 2001 until 2015. From 2012 until 2020 Frank served in addition as Vice-president for the German Research Foundation (DFG), which is Germany’s most important research funding organization. \\
His research interests include predictive control, data-based control, networked control, cooperative control, and nonlinear control with application to a wide range of fields including systems biology.
\end{IEEEbiography}

\section{Proof of Theorem~\ref{thm:willems_affine}}\label{sec:app_willems_affine_proof}
\begin{proof}
\textbf{Proof of ``if'':}\\
It clearly holds that
\begin{align}\label{eq:thm_hankel_proof_evolution}
y_{[i,i+L-1]}^\rmd=\Phi_L x_i^\rmd+\Gamma_{u,L} u_{[i,i+L-1]}^\rmd+\Gamma_{e,L} e_L+r_L
\end{align}
for suitably defined matrices $\Phi_L$, $\Gamma_{e,L}$, and $\Gamma_{u,L}$ depending on $A$, $B$, $C$, $D$, and for $e_L\coloneqq\mathbbm{1}_{L}\otimes e$, $r_L\coloneqq\mathbbm{1}_{L}\otimes r$.
Hence,
\begin{align}\label{eq:thm_hankel_proof1}
&y_{[0,L-1]}\stackrel{\eqref{eq:thm_willems_affine}}{=}\sum_{i=0}^{N-L}y^\rmd_{[i,i+L-1]}\alpha_i\\\nonumber
&\stackrel{\eqref{eq:thm_hankel_proof_evolution}}{=}\sum_{i=0}^{N-L}\alpha_i\left(\Phi_L x_i^\rmd+\Gamma_{u,L}u_{[i,i+L-1]}^\rmd+\Gamma_{e,L} e_L+r_L\right)\\\nonumber
&\stackrel{\eqref{eq:thm_willems_affine}}{=}\Phi_L\sum_{i=0}^{N-L}\alpha_ix_i^\rmd+\Gamma_{u,L}u_{[0,L-1]}+\Gamma_{e,L} e_L+r_L.
\end{align}
This implies that $\{u_k,y_k\}_{k=0}^{L-1}$ is a trajectory of~\eqref{eq:sys_affine} with initial condition $x_0=\sum_{i=0}^{N-L}\alpha_ix_i^d$.\\
\textbf{Proof of ``only if'':}\\
Let $\{u_k,y_k\}_{k=0}^{L-1}$ be a trajectory of~\eqref{eq:sys_affine} with initial condition $x_0$.
Using~\eqref{eq:def_pe}, there exists $\alpha\in\mathbb{R}^{N-L+1}$ such that
\begin{align}\label{eq:thm_hankel_proof2}
\begin{bmatrix}H_L(u^\rmd)\\H_1(x^\rmd_{[0,N-L]})\\\mathbbm{1}_{N-L+1}^\top\end{bmatrix}
\alpha=\begin{bmatrix}u_{[0,L-1]}\\x_0\\1\end{bmatrix}.
\end{align}
Note that the last row implies $\sum_{i=0}^{N-L}\alpha_i=1$.
Moreover,
\begin{align*}
&y_{[0,L-1]}=\Phi_L x_0+\Gamma_{u,L} u_{[0,L-1]}+\Gamma_{e,L}e_L+r_L\\
&\stackrel{\eqref{eq:thm_hankel_proof2}}{=}\sum_{i=0}^{N-L}\alpha_i\left(\Phi_L x_i^\rmd+\Gamma_{u,L}u_{[i,i+L-1]}^\rmd+\Gamma_{e,L} e_L+r_L\right)\\
&\stackrel{\eqref{eq:thm_hankel_proof_evolution}}{=}\sum_{i=0}^{N-L}\alpha_i y^\rmd_{[i,i+L-1]}=H_L(y^\rmd)\alpha.\qedhere
\end{align*}
\end{proof}

\section{Robust data-driven tracking MPC for affine systems}\label{sec:app_affine_MPC}
In this section, we present a novel data-driven tracking MPC scheme to control unknown affine systems using noisy input-output data.
The presented results combine and extend the linear data-driven MPC results in~\cite{berberich2021guarantees} and~\cite{berberich2020tracking} regarding robustness to noise and nominal setpoint tracking, respectively.
After stating the problem setting in Section~\ref{subsec:affine_MPC_setting}, we present the MPC scheme in Section~\ref{subsec:affine_MPC_scheme}.
Section~\ref{subsec:affine_stability} contains our closed-loop stability results and the proof is provided in Section~\ref{sec:app_B}.

\subsection{Problem setting}\label{subsec:affine_MPC_setting}
Let us consider the affine system~\eqref{eq:sys_affine}.
Similar to other works on data-driven MPC with closed-loop guarantees (see, e.g.,~\cite{berberich2021guarantees}) we assume that~\eqref{eq:sys_affine} is controllable and observable.
\begin{assumption}\label{ass:affine_ctrb}
(Controllability and observability)
The pair $(A,B)$ is controllable and the pair $(A,C)$ is observable.
\end{assumption}
%
Since we only have access to input-output data of~\eqref{eq:sys_affine}, we define an input-output equilibrium as follows (cf.~\cite{berberich2021guarantees,berberich2020tracking}).
\begin{definition}\label{def:IO_equilibrium}
We say that an input-output pair $(u^{\rms},y^{\rms})\in\mathbb{R}^{m+p}$ is an equilibrium of~\eqref{eq:sys_affine}, if the sequence $\{\bar{u}_k,\bar{y}_k\}_{k=0}^{n}$ with $(\bar{u}_k,\bar{y}_k)=(u^{\rms},y^{\rms})$ for $k\in\mathbb{I}_{[0,n]}$ is a trajectory of~\eqref{eq:sys_affine}.
\end{definition}
In this section, the goal is to track an output setpoint\footnote{Input setpoints can be included via an augmented output $y'=\begin{bmatrix}y\\u\end{bmatrix}$.} $y^\rmr\in\mathbb{R}^{p}$, which need not be an equilibrium of~\eqref{eq:sys_affine}, while satisfying pointwise-in-time input constraints $u_t\in\mathbb{U}\subseteq\mathbb{R}^m$, $t\in\mathbb{I}_{\geq0}$, with a compact, convex polytope $\mathbb{U}$.
We do not consider output constraints to avoid the additional challenges in the required robust constraint tightening, compare~\cite{berberich2020constraints}.
The proposed MPC scheme includes an artificial setpoint in the online optimization, thereby guaranteeing closed-loop stability of the optimal reachable equilibrium.
Due to a local controllability argument required for our theoretical results, we consider only equilibria whose input component lies in the interior of the constraints, i.e., in some convex polytope $\mathbb{U}^{\rms}\subseteq\text{int}(\mathbb{U})$.
Given a matrix $S\succ0$ and a data trajectory $\{u_k^\rmd,y_k^\rmd\}_{k=0}^{N-1}$ of~\eqref{eq:sys_affine} with persistently exciting input and state according to Definition~\ref{def:pe}, we define the optimal reachable equilibrium $(u^{\rms\rmr},y^{\rms\rmr})$ as the minimizer of
\begin{align}\label{eq:opt_reach_equil_affine}
J_{\mathrm{eq}}^*\coloneqq&\min_{u^{\rms},y^{\rms},\alpha^{\rms}}\lVert y^\rms-y^\rmr\rVert_S^2\\\nonumber
\text{s.t.}\>\>&\begin{bmatrix}H_{L+n+1}(u^\rmd)\\H_{L+n+1}(y^\rmd)\\\mathbbm{1}_{N-L-n}^\top\end{bmatrix}\alpha^{\rms}=\begin{bmatrix}\bbone_{L+n+1}\otimes u^{\rms}\\ \bbone_{L+n+1}\otimes y^{\rms}\\1\end{bmatrix},\>u^{\rms}\in\mathbb{U}^{\rms}
\end{align}
with some $L\in\mathbb{I}_{\geq0}$.
We note that $y^{\rms\rmr}$ is unique due to $S\succ0$ and $u^{\rms\rmr}$ is unique due to Assumption~\ref{ass:affine_unique_steady_state} below.
On the other hand, the corresponding solution $\alpha^\rms$ is in general not unique and we denote the solution with minimum $2$-norm by $\alpha^{\rms\rmr}$.
\begin{assumption}\label{ass:affine_unique_steady_state}
(Unique steady-state)
The matrix $\begin{bmatrix}A-I&B\\C&D\end{bmatrix}$ has full column rank.
\end{assumption}
Assumption~\ref{ass:affine_unique_steady_state} is a standard condition in tracking MPC (cf.~\cite[Lemma 1.8]{rawlings2020model},~\cite[Remark 1]{limon2018nonlinear}) and it implies the existence of an affine (hence Lipschitz continuous) map $\hat{g}$ 
which uniquely maps output equilibria $y^\rms$ to their corresponding input-state components $(x^\rms,u^\rms)$, i.e., $\hat{g}(y^\rms)=(x^\rms,u^\rms)$ and
\begin{align}\label{eq:ass_affine_unique_steady_state}
\left\lVert \begin{bmatrix}x_1^\rms\\u_1^\rms\end{bmatrix}
-\begin{bmatrix}x_2^\rms\\u_2^\rms\end{bmatrix}\right\rVert_2^2\leq c_g\lVert y_1^\rms-y_2^\rms\rVert_2^2
\end{align}
with some $c_g>0$ for any two input-output equilibria $(u_1^\rms,y_1^\rms)$, $(u_2^\rms,y_2^\rms)$ with corresponding steady-states $x_1^\rms$, $x_2^\rms$.
Since $S\succ0$, the cost of~\eqref{eq:opt_reach_equil_affine} is strongly convex in $y^\rms$ and for any $y^\rms$ satisfying the constraints of~\eqref{eq:opt_reach_equil_affine} for some $u^\rms$, $\alpha^\rms$, we have
\begin{align}\label{eq:strong_convexity_affine}
&\lVert y^\rms-y^\rmr\rVert_S^2-J_{\mathrm{eq}}^*\geq\lVert y^{\rms}-y^{\rms\rmr}\rVert_S^2,
\end{align}
compare~\cite[Inequality (11)]{koehler2020nonlinear}.
Throughout this section and in contrast to Section~\ref{sec:willems_affine}, we assume that a \emph{noisy} input-output trajectory $\{u_k^\rmd,\tilde{y}_k^\rmd\}_{k=0}^{N-1}$ of~\eqref{eq:sys_affine} is available, where the output $\tilde{y}_k^\rmd=y_k^\rmd+\varepsilon_k^\rmd$, $k\in\mathbb{I}_{[0,N-1]}$, is perturbed by additive measurement noise $\{\varepsilon_k^\rmd\}_{k=0}^{N-1}$.
Additionally, the input-output measurements obtained online which are used to include initial conditions are also affected by noise $\{\varepsilon_k^\rmu,\varepsilon_k\}_{k=0}^{\infty}$ as $\tilde{u}_k=u_k+\varepsilon^\rmu_k$, $\tilde{y}_k=y_k+\varepsilon_k$, $k\in\mathbb{I}_{\geq0}$.
\begin{assumption}\label{ass:affine_noise_bound}
(Noise bound)
It holds that $\lVert\varepsilon_k^\rmd\rVert_{2}\leq\bar{\varepsilon}$, $k\in\mathbb{I}_{[0,N-1]}$, and $\lVert\varepsilon_k^\rmu\rVert_{2}\leq\bar{\varepsilon}$, $\lVert\varepsilon_k\rVert_{2}\leq\bar{\varepsilon}$, $k\in\mathbb{I}_{\geq0}$, with $\bar{\varepsilon}>0$ known.
\end{assumption}
Further, we require that the data are persistently exciting.
\begin{assumption}\label{ass:affine_pe}
(Persistence of excitation)
The data $\{u_k^\rmd\}_{k=0}^{N-1}$, $\{x_k^\rmd\}_{k=0}^{N-L}$ are persistently exciting of order $L+n+1$ in the sense of Definition~\ref{def:pe}.
\end{assumption}

\subsection{Proposed MPC scheme}\label{subsec:affine_MPC_scheme}

\begin{subequations}\label{eq:DD_MPC_affine}
Given data $\{u_k^\rmd,\tilde{y}_k^\rmd\}_{k=0}^{N-1}$ as well as initial conditions $\{\tilde{u}_k,\tilde{y}_k\}_{k=t-n}^{t-1}$, the following open-loop optimal control problem is the basis for our MPC scheme:
\begin{align}\label{eq:DD_MPC_affine_cost}
\underset{\substack{\alpha(t),\sigma(t)\\u^{\rms}(t),y^{\rms}(t)}}{\min}&\sum_{k=-n}^{L}
\lVert\bar{u}_k(t)-u^{\rms}(t)\rVert_R^2+\lVert\bar{y}_k(t)-y^{\rms}(t)\rVert_Q^2\\\nonumber
+\lVert y^\rms&(t)-y^\rmr\rVert_S^2+\lambda_\alpha\bar{\varepsilon}^{\beta_{\alpha}}\lVert\alpha(t)-\alpha^{\rms\rmr}\rVert_2^2+\frac{\lambda_\sigma}{\bar{\varepsilon}^{\beta_{\sigma}}}\lVert\sigma(t)\rVert_2^2\\
\label{eq:DD_MPC_affine_hankel} \text{s.t.}\>\> &\>\begin{bmatrix}
\bar{u}(t)\\\bar{y}(t)+\sigma(t)\\1\end{bmatrix}=\begin{bmatrix}H_{L+n+1}\left(u^\rmd\right)\\H_{L+n+1}\left(\tilde{y}^\rmd\right)\\\mathbbm{1}_{N-L-n}^\top\end{bmatrix}\alpha(t),\\\label{eq:DD_MPC_affine_init}
&\>\begin{bmatrix}\bar{u}_{[-n,-1]}(t)\\\bar{y}_{[-n,-1]}(t)\end{bmatrix}=\begin{bmatrix}\tilde{u}_{[t-n,t-1]}\\\tilde{y}_{[t-n,t-1]}\end{bmatrix},
\\\label{eq:DD_MPC_affine_TEC}
&\>\begin{bmatrix}\bar{u}_{[L-n,L]}(t)\\\bar{y}_{[L-n,L]}(t)\end{bmatrix}=\begin{bmatrix}
\bbone_{n+1}\otimes u^{\rms}(t)\\
\bbone_{n+1}\otimes y^{\rms}(t)\end{bmatrix},\\\label{eq:DD_MPC_affine_constraints}
&\>\bar{u}_k(t)\in\mathbb{U},\>k\in\mathbb{I}_{[0,L]},\>u^{\rms}(t)\in\mathbb{U}^{\rms}.
\end{align}
\end{subequations}
Problem~\eqref{eq:DD_MPC_affine} is analogous to Problem~\eqref{eq:DD_MPC_NL} with the main differences that i) the data used for prediction via Hankel matrices in~\eqref{eq:DD_MPC_affine_hankel} are only collected once offline, i.e., they are \emph{not} updated online, and ii) the  measurements $\{u_k^{\rmd},\tilde{y}_k^{\rmd}\}_{k=0}^{N-1}$ and $\{\tilde{u}_k,\tilde{y}_k\}_{k=t-n}^{t-1}$ originate from the \emph{affine} dynamics~\eqref{eq:sys_affine} with additional output measurement noise.
We refer to Section~\ref{subsec:NL_scheme} for a detailed discussion of the ingredients of Problem~\eqref{eq:DD_MPC_NL}, which applies analogously for Problem~\eqref{eq:DD_MPC_affine}.

Note that the regularization of $\alpha(t)$ and $\sigma(t)$ depends on the noise bound $\bar{\varepsilon}$ such that, in the limit $\bar{\varepsilon}\to0$, a nominal MPC scheme with guaranteed exponential stability is recovered.
Similar to Proposition~\ref{prop:continuity}, we include parameters $\beta_{\alpha},\beta_{\sigma}>0$ satisfying $\beta_{\alpha}+2\beta_{\sigma}<2$.
In the literature, different choices for $\beta_{\alpha}$, $\beta_{\sigma}$ have been considered such as $\beta_{\alpha}=1$, $\beta_{\sigma}=0$ (see~\cite{coulson2021distributionally,berberich2021guarantees}) or $\beta_{\alpha}=1$, $\beta_{\sigma}=1$ (see~\cite{bongard2021robust}).
In this paper, the additional flexibility provided by the parameters $\beta_{\alpha}$ and $\beta_{\sigma}$ is mainly required for an argument in the nonlinear stability proof in Section~\ref{sec:NL_MPC}.
All theoretical results in the present section remain valid as long as $\beta_{\alpha}+2\beta_{\sigma}<2$.

Throughout this section, the optimal solution of~\eqref{eq:DD_MPC_affine} at time $t$ is denoted by $\bar{u}^*(t)$, $\bar{y}^*(t)$, $\alpha^*(t)$, $\sigma^*(t)$, $u^{\rms*}(t)$, $y^{\rms*}(t)$, and the closed-loop input, state, and output at time $t$ are denoted by $u_t$, $x_t$, and $y_t$, respectively.
Algorithm~\ref{alg:MPC_n_step} summarizes the proposed MPC scheme, which takes a multi-step form.

\begin{algorithm}
\begin{Algorithm}\label{alg:MPC_n_step}
\normalfont{\textbf{Robust Data-Driven MPC Scheme}}\\
\textbf{Offline:} Choose upper bound on system order $n$, prediction horizon $L$, cost matrices $Q,R,S\succ0$, regularization parameters $\lambda_{\alpha},\lambda_{\sigma},\beta_{\alpha},\beta_{\sigma}>0$, constraint sets $\mathbb{U},\mathbb{U}^{\rms}$, noise bound $\bar{\varepsilon}$, setpoint $y^\rmr$, and generate data $\{u_k^\rmd,\tilde{y}_k^\rmd\}_{k=0}^{N-1}$.
Compute an approximation of $\alpha^{\rms\rmr}$ by solving~\eqref{eq:opt_reach_equil_affine}, compare Remark~\ref{rk:alpha_sr}.\\
\textbf{Online:}
\begin{enumerate}
\item At time $t$, take the past $n$ measurements $\{\tilde{u}_k,\tilde{y}_k\}_{k=t-n}^{t-1}$ and solve~\eqref{eq:DD_MPC_affine}.
\item Apply the input sequence $u_{[t,t+n-1]}=\bar{u}_{[0,n-1]}^*(t)$ over the next $n$ time steps.
\item Set $t=t+n$ and go back to 1).
\end{enumerate}
\end{Algorithm}
\end{algorithm}

Similar to Algorithm~\ref{alg:MPC_n_step_NL}, considering a multi-step MPC scheme instead of a standard (one-step) MPC scheme simplifies the theoretical analysis with terminal equality constraints due to a local controllability argument in the proof.
In~\cite{berberich2020tracking}, guarantees for a \emph{one-step} tracking MPC scheme for linear systems are provided for noise-free data.
We conjecture that the results in this section hold locally close to the steady-state $x^{\rms\rmr}$ if Algorithm~\ref{alg:MPC_n_step} is executed in a one-step fashion, see also~\cite[Remark 4]{berberich2021guarantees} for a similar argument in data-driven MPC without online optimization of $u^\rms(t)$, $y^\rms(t)$.

A key difficulty in analyzing closed-loop properties of data-driven MPC based on Theorem~\ref{thm:willems_affine} (and analogously for MPC based on~\cite{willems2005note}) is that no state measurements are available and the cost only penalizes the input and output.
Therefore, similar to the results in Section~\ref{sec:NL_MPC}, our analysis uses the extended state $\xi_t$ and its noisy version $\tilde{\xi}_t$
\begin{align}\label{eq:xi_def}
\xi_t\coloneqq\begin{bmatrix}u_{[t-n,t-1]}\\y_{[t-n,t-1]}\end{bmatrix},\quad
\tilde{\xi}_t\coloneqq\begin{bmatrix}\tilde{u}_{[t-n,t-1]}\\\tilde{y}_{[t-n,t-1]}\end{bmatrix},
\end{align}
and we write $J_L^*(\tilde{\xi}_t)$ for the optimal cost of~\eqref{eq:DD_MPC_affine}.
We denote the optimal reachable extended state by $\xi^{\rms\rmr}\coloneqq\begin{bmatrix}\bbone_n\otimes u^{\rms\rmr}\\\bbone_n\otimes y^{\rms\rmr}\end{bmatrix}$.
Further, we note that, by observability, there exist $T_{\rmx}$, $T_\rme$, $T_\rmr$, such that
\begin{align}\label{eq:trafo_T_extended_state2}
x_t=T_{\rmx}\xi_t+T_\rme e+T_\rmr r.
\end{align}
Thus, for any two pairs of state and extended state vectors $(x_{t}^\rma,\xi_{t}^\rma)$ and $(x_{t}^\rmb,\xi_{t}^\rmb)$, it holds that
\begin{align}\label{eq:trafo_T_extended_state}
x_{t}^\rma-x_{t}^\rmb=T_{\rmx}(\xi_{t}^\rma-\xi_{t}^\rmb).
\end{align}

\subsection{Closed-loop guarantees}\label{subsec:affine_stability}
In this section, we prove that Algorithm~\ref{alg:MPC_n_step} practically exponentially stabilizes the closed loop.
To this end, we employ Proposition~\ref{prop:continuity}, where the bound on the linearization error is replaced by the noise bound in Assumption~\ref{ass:affine_noise_bound}.
Then,
Proposition~\ref{prop:continuity} implies that the noisy measurements in Problem~\eqref{eq:DD_MPC_affine} can be translated into an additive input disturbance for data-driven MPC with noise-free data.
In the following, we prove that the latter scheme is robust w.r.t.\ input disturbances (Theorem~\ref{thm:stab_affine_nominal}) which we then combine with Proposition~\ref{prop:continuity} to conclude practical exponential stability of the closed loop under the robust data-driven MPC scheme in Algorithm~\ref{alg:MPC_n_step} (Corollary~\ref{cor:stab_affine_robust}).

\begin{subequations}\label{eq:DD_MPC_affine_nominal}
Let us define the underlying nominal MPC scheme:
\begin{align}\label{eq:DD_MPC_affine_nominal_cost}
\underset{\substack{\alpha(t)\\u^{\rms}(t),y^{\rms}(t)}}{\min}&\sum_{k=-n}^{L}
\lVert\bar{u}_k(t)-u^{\rms}(t)\rVert_R^2+\lVert\bar{y}_k(t)-y^{\rms}(t)\rVert_Q^2\\\nonumber
&+\lVert y^\rms(t)-y^\rmr\rVert_S^2+\lambda_{\alpha}\bar{\varepsilon}^{\beta_{\alpha}}\lVert\alpha(t)-\alpha^{\rms\rmr}\rVert_2^2\\
\label{eq:DD_MPC_affine_nominal_hankel} \text{s.t.}\>\> &\>\begin{bmatrix}
\bar{u}(t)\\\bar{y}(t)\\1\end{bmatrix}=\begin{bmatrix}H_{L+n+1}\left(u^{\rmd})\right)\\H_{L+n+1}\left(y^{\rmd}\right)\\\mathbbm{1}_{N-L-n}^\top\end{bmatrix}\alpha(t),\\\label{eq:DD_MPC_affine_nominal_init}
&\>\begin{bmatrix}\bar{u}_{[-n,-1]}(t)\\\bar{y}_{[-n,-1]}(t)\end{bmatrix}=
\begin{bmatrix}u_{[t-n,t-1]}\\y_{[t-n,t-1]}\end{bmatrix},
\\\label{eq:DD_MPC_affine_nominal_TEC}
&\>\begin{bmatrix}\bar{u}_{[L-n,L]}(t)\\\bar{y}_{[L-n,L]}(t)\end{bmatrix}=
\begin{bmatrix}\bbone_{n+1}\otimes u^{\rms}(t)\\
\bbone_{n+1}\otimes y^{\rms}(t)\end{bmatrix},\\\label{eq:DD_MPC_affine_nominal_constraints}
&\>\bar{u}_k(t)\in\mathbb{U},\>k\in\mathbb{I}_{[0,L]},\>u^{\rms}(t)\in\mathbb{U}^{\rms}.
\end{align}
\end{subequations}

Similar to Section~\ref{subsec:NL_continuity}, we denote the optimal solution of Problem~\eqref{eq:DD_MPC_affine_nominal} by $\check{\alpha}^*(t)$, $\check{u}^{\rms*}(t)$, $\check{y}^{\rms*}(t)$, $\check{u}^*(t)$, $\check{y}^*(t)$, and the corresponding optimal cost by $\check{J}_L^*(\xi_t)$.
For the stability analysis, we consider the Lyapunov function candidate $V(\xi_t)\coloneqq \check{J}_L^*(\xi_t)-J_{\mathrm{eq}}^*$.

\begin{theorem}\label{thm:stab_affine_nominal}
Suppose $L\geq2n$, Assumptions~\ref{ass:affine_ctrb},~\ref{ass:affine_unique_steady_state}, and~\ref{ass:affine_pe} hold, and consider System~\eqref{eq:sys_affine} under control with an $n$-step MPC scheme (cf. Algorithm~\ref{alg:MPC_n_step}) based on Problem~\eqref{eq:DD_MPC_affine_nominal}, where the input applied to~\eqref{eq:sys_affine} is perturbed as
\begin{align}\label{eq:thm_stab_affine_nominal_input_disturbance}
u_{[t,t+n-1]}=\check{u}^*_{[0,n-1]}(t)+d_{[t,t+n-1]}
\end{align}
for $t=ni$, $i\in\mathbb{I}_{\geq0}$, with some disturbance $\{d_t\}_{t=0}^{\infty}$ bounded as $\lVert d_t\rVert_{2}\leq\bar{\varepsilon}$ for all $t\in\mathbb{I}_{\geq0}$.

For any $V_{\mathrm{ROA}}>0$, there exist $\bar{\varepsilon}_{\max},c_\rml,c_\rmu>0$, $0<\check{c}_{\rmV}<1$, and $\beta_{\rmd}\in\mathcal{K}_{\infty}$ such that, if $V(\xi_0)\leq V_{\mathrm{ROA}}$, then, for all $\bar{\varepsilon}\leq\bar{\varepsilon}_{\max}$, $t=n i$, $i\in\mathbb{I}_{\geq0}$, Problem~\eqref{eq:DD_MPC_affine_nominal} is feasible and the closed loop satisfies
\begin{align}\label{eq:thm_stab_affine_nominal_lower_upper}
c_\rml\lVert\xi_t-\xi^{\rms\rmr}\rVert_2^2&\leq V(\xi_t)\leq c_\rmu\lVert\xi_t-\xi^{\rms\rmr}\rVert_2^2,\\\label{eq:thm_stab_affine_nominal}
V(\xi_{t+n})&\leq \check{c}_{\rmV}V(\xi_t)+\beta_{\rmd}(\bar{\varepsilon}).
\end{align}
\end{theorem}

The proof of Theorem~\ref{thm:stab_affine_nominal} is a straightforward adaptation of arguments from model-based and data-driven tracking MPC~\cite{koehler2020nonlinear,berberich2020tracking} and it is provided for completeness in Appendix~\ref{sec:app_B}.
The result shows that the nominal MPC scheme based on repeatedly solving Problem~\eqref{eq:DD_MPC_affine_nominal} in an $n$-step fashion (compare Algorithm~\ref{alg:MPC_n_step}) practically exponentially stabilizes the optimal reachable equilibrium $\xi^{\rms\rmr}$ in closed loop in the presence of input disturbances.
To be precise, the decay bound~\eqref{eq:thm_stab_affine_nominal} in combination with the lower and upper bounds in~\eqref{eq:thm_stab_affine_nominal_lower_upper} implies that the closed loop converges to a neighborhood of $\xi^{\rms\rmr}$, the size of which shrinks if the disturbance bound (denoted by $\bar{\varepsilon}$ with a slight abuse of notation) is small.
The guaranteed region of attraction is then given by $V(\xi_0)\leq V_{\mathrm{ROA}}$ and, in particular, for a larger size $V_{\mathrm{ROA}}$ the maximal disturbance bound $\bar{\varepsilon}$ ensuring the closed-loop properties decreases.
Moreover, the proof of Theorem~\ref{thm:stab_affine_nominal} reveals that the closed-loop robustness improves (i.e., the maximal disturbance bound $\bar{\varepsilon}$ leading to practical stability increases) if the persistence of excitation condition in Assumption~\ref{ass:affine_pe} is quantitatively stronger, i.e., the minimum singular value of the matrix in~\eqref{eq:def_pe} increases.
A similar relation was observed in~\cite{berberich2021guarantees} for linear data-driven MPC without online optimization of $(u^\rms(t),y^\rms(t))$.

In the following, we combine Proposition~\ref{prop:continuity} and Theorem~\ref{thm:stab_affine_nominal} to prove that the closed loop under the robust data-driven MPC scheme based on Problem~\eqref{eq:DD_MPC_affine} is practically exponentially stable in the presence of noisy output data and perturbed initial conditions.
To this end, we employ the Lyapunov function $V(\xi_t)$ used in Theorem~\ref{thm:stab_affine_nominal}.

\begin{corollary}\label{cor:stab_affine_robust}
Suppose $L\geq2n$, Assumptions~\ref{ass:affine_ctrb}--\ref{ass:affine_pe} hold, Problem~\eqref{eq:DD_MPC_affine_nominal} satisfies an LICQ (compare Assumption~\ref{ass:LICQ}), and $\beta_{\alpha}+2\beta_{\sigma}<2$.
Then, for any $V_{\mathrm{ROA}}>0$, there exist $\bar{\varepsilon}_{\max}>0$ and $\beta_{\rmV}\in\mathcal{K}_{\infty}$ such that, for all initial conditions with $V(\xi_0)\leq V_{\mathrm{ROA}}$ and all $\bar{\varepsilon}\leq\bar{\varepsilon}_{\max}$, the closed-loop trajectory under Algorithm~\ref{alg:MPC_n_step} satisfies
\begin{align}\label{eq:cor_stab_affine_robust_decay}
V(\xi_{t+n})&\leq \check{c}_{\rmV}V(\xi_t)+\beta_{\rmV}(\bar{\varepsilon})
\end{align}
for all $t=ni$, $i\in\mathbb{I}_{\geq0}$, with $\check{c}_{\rmV}$ as in~\eqref{eq:thm_stab_affine_nominal}.
\end{corollary}
\begin{proof}
If $\bar{\varepsilon}\leq\bar{\varepsilon}_{\max}$ is sufficiently small, then Proposition~\ref{prop:continuity} and Theorem~\ref{thm:stab_affine_nominal} imply~\eqref{eq:cor_stab_affine_robust_decay} for $t=0$ with $\beta_{\rmV}\coloneqq\beta_{\rmd}\circ\beta_{\rmu}$, where $\circ$ denotes concatenation.
Further, with $\bar{\varepsilon}$ sufficiently small, we have $V(\xi_{t+n})\leq V_{\mathrm{ROA}}$ such that the argument can be applied recursively and~\eqref{eq:cor_stab_affine_robust_decay} holds for all $t=ni$, $i\in\mathbb{I}_{\geq0}$.
\end{proof}

Corollary~\ref{cor:stab_affine_robust} shows that the closed loop under Algorithm~\ref{alg:MPC_n_step} exponentially converges to a neighborhood of the optimal reachable equilibrium whose size increases with $\bar{\varepsilon}$.
The result is a simple consequence of the facts that the noise in Problem~\eqref{eq:DD_MPC_affine} can be translated into an input disturbance for nominal data-driven MPC (Proposition~\ref{prop:continuity}) and the closed loop under the latter is practically stable w.r.t.\ the disturbance bound (Theorem~\ref{thm:stab_affine_nominal}).
Analogously to Section~\ref{sec:NL_MPC}, the analysis presented in this section reveals a separation principle of data-driven MPC with noise-free and noisy data.
That is, any data-driven MPC scheme whose nominal version is robust w.r.t.\ input disturbances will also lead to a practically stable closed loop in the presence of noisy output measurements affecting both the offline data in the Hankel matrices and the online data used to specify initial conditions.
More precisely, this analysis can also be applied to simplify the robust stability proofs for data-driven MPC schemes with~\cite{berberich2021guarantees} or without~\cite{bongard2021robust} terminal equality constraints.
Similarly, we conjecture that robustness of a one-step data-driven MPC scheme with the terminal ingredients from~\cite{berberich2021on} can be proven, given the inherent robustness results due to terminal ingredients shown in~\cite{yu2014inherent}.

The main contribution of this paper is a data-driven MPC scheme to control unknown \emph{nonlinear} systems based only on measured input-output data with closed-loop guarantees.
Nevertheless, it is worth noting that the theoretical guarantees provided by Corollary~\ref{cor:stab_affine_robust} are also a significant improvement over existing works~\cite{berberich2021guarantees,bongard2021robust} on closed-loop stability and robustness in data-driven MPC for \emph{linear} systems (cf. Remark~\ref{rk:affine_motivation}).

\subsection{Proof of Theorem~\ref{thm:stab_affine_nominal}}\label{sec:app_B}

\begin{proof}
The proof is divided into four parts.
We first show the lower and upper bounds~\eqref{eq:thm_stab_affine_nominal_lower_upper} on the Lyapunov function candidate in Part (i).
In Parts (ii) and (iii), we propose two different candidate solutions for~\eqref{eq:DD_MPC_affine_nominal} at time $t+n$ for two complementary scenarios, assuming that~\eqref{eq:DD_MPC_affine_nominal} is feasible at time $t$.
In Part (iv), we combine the bounds to prove~\eqref{eq:thm_stab_affine_nominal}.\\
\textbf{(i) Lower and upper bound on $V(\xi_t)$}\\
\textbf{(i).a Lower bound on $V(\xi_t)$}\\
Using that $(\check{u}^{\rms*}(t),\check{y}^{\rms*}(t))$ is feasible for~\eqref{eq:opt_reach_equil_affine}, we have
\begin{align}\label{eq:thm_stab_affine_nominal_proof_lower1}
&\lVert \check{y}^{\rms*}(t)-y^\rmr\rVert_S^2-\lVert y^{\rms\rmr}-y^\rmr\rVert_S^2\stackrel{\eqref{eq:strong_convexity_affine}}{\geq}\lVert\check{y}^{\rms*}(t)-y^{\rms\rmr}\rVert_S^2\\\nonumber
\stackrel{\eqref{eq:ass_affine_unique_steady_state}}{\geq}
&\frac{\lambda_{\min}(S)}{2}\left(\lVert\check{y}^{\rms*}(t)-y^{\rms\rmr}\rVert_2^2+\frac{1}{c_g}\lVert\check{u}^{\rms*}(t)-u^{\rms\rmr}\rVert_2^2\right).
\end{align}
In combination with~\eqref{eq:ab_ineq2}, this implies
\begin{align}\label{eq:thm_stab_affine_nominal_proof_lower}
V(\xi_t)\geq &\sum_{k=-n}^{-1}\lVert u_{t+k}-\check{u}^{\rms*}(t)\rVert_R^2
+\lVert y_{t+k}-\check{y}^{\rms*}(t)\rVert_Q^2\\\nonumber
&+\lVert \check{y}^{\rms*}(t)-y^\rmr\rVert_S^2-\lVert y^{\rms\rmr}-y^\rmr\rVert_S^2\\\nonumber
\stackrel{\eqref{eq:thm_stab_affine_nominal_proof_lower1}}{\geq}&c_{l}\sum_{k=-n}^{-1}(\lVert u_{t+k}-u^{\rms\rmr}\rVert_2^2+\lVert y_{t+k}-y^{\rms\rmr}\rVert_2^2)
\end{align}
and therefore, the lower bound in~\eqref{eq:thm_stab_affine_nominal_lower_upper} with 
\begin{align*}
c_\rml\coloneqq\frac{1}{2}\min\left\{\lambda_{\min}(Q,R), \frac{\lambda_{\min}(S)}{2n}\min\left\{1,\frac{1}{c_g}\right\}\right\}.
\end{align*}
\textbf{(i).b Upper bound on $V(\xi_t)$}\\
Suppose $\lVert\xi_t-\xi^{\rms\rmr}\rVert_2\leq\delta$ for a sufficiently small $\delta>0$.
We define a candidate for the artificial equilibrium by $(u^\rms(t),y^\rms(t))=(u^{\rms\rmr},y^{\rms\rmr})$.
Using controllability and $u^{\rms\rmr}\in\mathbb{U}^{\rms}\subseteq\mathrm{int}(\mathbb{U})$, there exists a feasible input-output trajectory $\{\bar{u}(t),\bar{y}(t)\}_{k=-n}^L$ for Problem~\eqref{eq:DD_MPC_affine_nominal} with $\tilde{\sigma}=0$ satisfying the terminal constraint~\eqref{eq:DD_MPC_affine_nominal_TEC} as well as
\begin{align}\label{eq:thm_stab_affine_nominal_proof_upper_bound_ctrb}
&\sum_{k=-n}^L\lVert\bar{u}_k(t)-u^{\rms\rmr}\rVert_2^2+\lVert\bar{y}_k-y^{\rms\rmr}\rVert_2^2
\leq\Gamma_{\xi}\lVert\xi_t-\xi^{\rms\rmr}\rVert_2^2
\end{align}
for some $\Gamma_{\xi}>0$.
The vector $\alpha(t)$ is chosen as
\begin{align*}
\alpha(t)=H_{ux}^\dagger\begin{bmatrix} \bar{u}(t)\\x_{t-n}\\1\end{bmatrix},
\end{align*}
where $H_{ux}$ is defined as
\begin{align}\label{eq:Hux_affine}
H_{ux}\coloneqq\begin{bmatrix}H_{L+n+1}(u^\rmd)\\H_1(x_{[0,N-L-n-1]}^\rmd)\\\mathbbm{1}_{N-L-n}^\top\end{bmatrix}.
\end{align}
This implies that all constraints of Problem~\eqref{eq:DD_MPC_affine_nominal} are satisfied (compare the proof of ``only if'' in Theorem~\ref{thm:willems_affine}) and thus, the Lyapunov function candidate $V(\xi_t)$ is upper bounded as
\begin{align}\label{eq:thm_stab_affine_nominal_proof_upper_bound1}
V(\xi_t)\leq\Gamma_{\xi}\lambda_{\max}(Q,R)\lVert\xi_t-\xi^{\rms\rmr}\rVert_2^2+\lambda_{\alpha}\bar{\varepsilon}^{\beta_{\alpha}}\lVert\alpha(t)-\alpha^{\rms\rmr}\rVert_2^2.
\end{align}
Further, using $\alpha^{\rms\rmr}=H_{ux}^\dagger\begin{bmatrix}u^{\rms\rmr}_{L+n+1}\\x^{\rms\rmr}\\1\end{bmatrix}$, we infer
\begin{align}\label{eq:thm_stab_affine_nominal_proof_upper_bound2}
\lVert\alpha(t)-\alpha^{\rms\rmr}\rVert_2^2\stackrel{\eqref{eq:thm_stab_affine_nominal_proof_upper_bound_ctrb}}{\leq}&\lVert H_{ux}^\dagger\rVert_2^2(\Gamma_{\xi}\lVert\xi_t-\xi^{\rms\rmr}\rVert_2^2+\lVert x_{t-n}-x^{\rms\rmr}\rVert_2^2).
\end{align}
Finally, using observability, there exists a matrix $M$ such that 
\begin{align}\label{eq:thm_stab_affine_nominal_proof_upper_bound3}
x_{t-n}-x^{\rms\rmr}=M(\xi_t-\xi^{\rms\rmr}).
\end{align}
Combining~\eqref{eq:thm_stab_affine_nominal_proof_upper_bound1}--\eqref{eq:thm_stab_affine_nominal_proof_upper_bound3}, we deduce that the upper bound in~\eqref{eq:thm_stab_affine_nominal_lower_upper} holds for all $\xi_t$ satisfying $\lVert\xi_t-\xi^{\rms\rmr}\rVert_2\leq\delta$ with
\begin{align*}
c_\rmu\coloneqq\Gamma_{\xi}\lambda_{\max}(Q,R)+\lambda_{\alpha}\bar{\varepsilon}^{\beta_{\alpha}}\lVert H_{ux}^\dagger\rVert_2^2\left(1+\Gamma_{\xi}+\lVert M\rVert_2^2\right).
\end{align*}
Since~\eqref{eq:DD_MPC_affine_nominal} is a multi-parametric QP, the optimal cost is piecewise quadratic and thus,~\eqref{eq:thm_stab_affine_nominal_lower_upper} holds for any feasible $\xi_t$ (with a modified constant $c_\rmu$).\\
\textbf{(ii) Candidate solution 1}\\
Assume
\begin{align}\label{eq:thm_stab_affine_nominal_proof_case1}
&\sum_{k=-n}^{-1}\lVert u_{t+k}-\check{u}^{\rms*}(t)\rVert_R^2+\lVert y_{t+k}-\check{y}^{\rms*}(t)\rVert_Q^2\\\nonumber
\geq&\gamma\lVert\check{y}^{\rms*}(t)-y^{\rms\rmr}\rVert_S^2
\end{align}
for a constant $\gamma>0$ which will be fixed later in the proof.\\
\textbf{(ii).a Definition of candidate solution}\\
We choose both the input and output equilibrium candidate as the previously optimal solution $\check{u}^{\rms}\text{$'$}(t+n)=\check{u}^{\rms*}(t)$, $\check{y}^{\rms}\text{$'$}(t+n)=\check{y}^{\rms*}(t)$.
The first $L-2n$ elements of the predicted input trajectory are a shifted version of the previously optimal trajectory, i.e., $\check{u}_k'(t+n)=\check{u}^*_{k+n}(t)$ for $k\in\mathbb{I}_{[0,L-2n-1]}$.
Over time steps $k\in\mathbb{I}_{[-n,-1]}$, we let $\check{u}_k'(t+n)=u_{[t+n+k]}$ and $\check{y}_k'(t+n)=y_{t+n+k}$.
Denote by $\{y_{k}'(t+n)\}_{k=0}^{L-n}$ the output resulting from an application of $\check{u}_{[n,L]}^*(t)$ to the system~\eqref{eq:sys_affine} initialized at $(u_{[t,t+n-1]},y_{[t,t+n-1]})$.
For $k\in\mathbb{I}_{[0,L-2n-1]}$, we let $\check{y}_k'(t+n)=y_k'(t+n)$.
We write $\check{x}_{L-2n}'(t+n)$ for the state at time $L-2n$ corresponding to $\{\check{u}_k'(t+n),\check{y}_k'(t+n)\}_{k=0}^{L-2n-1}$.
By controllability, there exists an input-output trajectory $\{\check{u}_k'(t+n),\check{y}_k'(t+n)\}_{k=L-2n}^{L-n-1}$ steering the system to the steady-state $\check{x}^{\rms*}(t)$ corresponding to $(\check{u}^{\rms*}(t),\check{y}^{\rms*}(t))$ while satisfying
\begin{align}\nonumber
&\sum_{k=L-2n}^{L-n-1}\lVert\check{u}_k'(t+n)-\check{u}^{\rms*}(t)\rVert_2^2+\lVert\check{y}_k'(t+n)-\check{y}^{\rms*}(t)\rVert_2^2\\\label{eq:thm_stab_affine_nominal_proof_case1_ctrb}
\leq&\Gamma\lVert\check{x}_{L-2n}'(t+n)-\check{x}^{\rms*}(t)\rVert_2^2
\end{align}
for some $\Gamma>0$.
In the following, we show that $\check{u}_k'(t+n)\in\mathbb{U}$, $k\in\mathbb{I}_{[L-2n,L-n-1]}$, if $\bar{\varepsilon}$ is sufficiently small.
Recall that $\check{x}^{\rms*}(t)$ is the steady-state corresponding to $(\check{u}^{\rms*}(t),\check{y}^{\rms*}(t))$, whereas the output $y_{[L-2n,L-n-1]}'(t+n)$ results from applying $\check{u}^{\rms*}_n(t)$ to the system at initial state $\check{x}_{L-2n}'(t+n)$.
Hence, using observability, there exists $c_{\rmx,1}>0$ such that
\begin{align*}
&\lVert\check{x}_{L-2n}'(t+n)-\check{x}^{\rms*}(t)\rVert_2^2\\
\leq&c_{\rmx,1}\lVert y_{[L-2n,L-n-1]}'(t+n)-\bbone_n\otimes \check{y}^{\rms*}(t)\rVert_2^2.
\end{align*}
The output trajectories $y_{[L-2n,L-n-1]}'(t+n)$ and $\check{y}_n^{\rms*}(t)$ result from applying the input $\check{u}_{[n,L]}^*(t)$ to the system~\eqref{eq:sys_affine} with initial conditions $(u_{[t,t+n-1]},y_{[t,t+n-1]})$ and $(\check{u}_{[0,n-1]}^*(t),\check{y}_{[0,n-1]}^*(t))$, respectively.
Since the difference between these initial conditions is linear in the disturbance $d_{[t,t+n-1]}$ (compare~\eqref{eq:thm_stab_affine_nominal_input_disturbance}), there exists $c_{\rmx,2}>0$ such that
\begin{align}\label{eq:thm_stab_affine_nominal_proof_case1_x_bound}
\lVert\check{x}_{L-2n}'(t+n)-\check{x}^{\rms*}(t)\rVert_2^2\leq c_{\rmx,2}\bar{\varepsilon}^2.
\end{align}
Together with~\eqref{eq:thm_stab_affine_nominal_proof_case1_ctrb} and $\check{u}^{\rms*}(t)\in\mathrm{int}(\mathbb{U})$, this shows that $\check{u}_k'(t+n)\in\mathbb{U}$ for $k\in\mathbb{I}_{[L-2n,L-n-1]}$ if $\bar{\varepsilon}$ is sufficiently small.
Finally, we let $(\check{u}_k'(t+n),\check{y}_k'(t+n))=(\check{u}^{\rms*}(t),\check{y}^{\rms*}(t))$ for $k\in\mathbb{I}_{[L-n,L]}$.
Using Assumption~\ref{ass:affine_pe}, we choose
\begin{align}\label{eq:thm_stab_affine_nominal_proof_case1_alpha_cand}
\check{\alpha}'(t+n)=H_{ux}^\dagger
\begin{bmatrix}
\check{u}'(t+n)\\x_t\\1\end{bmatrix}
\end{align}
with $H_{ux}$ as in~\eqref{eq:Hux_affine}.
This implies that~\eqref{eq:DD_MPC_affine_nominal_hankel} and thus, all constraints of~\eqref{eq:DD_MPC_affine_nominal} hold.
\\
\textbf{(ii).b Lyapunov function decay}\\
Using the above candidate solution, we have
\begin{align}\label{eq:thm_stab_affine_nominal_proof_case1_cost_decay1}
&V(\xi_{t+n})-V(\xi_t)\\\nonumber
\leq&\sum_{k=-n}^L\lVert\check{u}_k'(t+n)-\check{u}^{\rms*}(t)\rVert_R^2
+\lVert\check{y}_k'(t+n)-\check{y}^{\rms*}(t)\rVert_Q^2\\\nonumber
&+\lambda_{\alpha}\bar{\varepsilon}^{\beta_{\alpha}}(\lVert\check{\alpha}'(t+n)-\alpha^{\rms\rmr}\rVert_2^2
-\lVert\check{\alpha}^*(t)-\alpha^{\rms\rmr}\rVert_2^2)\\\nonumber
&-\sum_{k=-n}^L(\lVert\check{u}_k^*(t)-\check{u}^{\rms*}(t)\rVert_R^2+
\lVert\check{y}_k^*(t)-\check{y}^{\rms*}(t)\rVert_Q^2).
\end{align}
The terms involving the input are 
\begin{align}\label{eq:thm_stab_affine_nominal_proof_case1_input_bound}
&\sum_{k=-n}^L\lVert\check{u}_k'(t+n)-\check{u}^{\rms*}(t)\rVert_R^2-\lVert\check{u}_k^*(t)-\check{u}^{\rms*}(t)\rVert_R^2\\\nonumber
=&-\sum_{k=-n}^{-1}\lVert u_{t+k}-\check{u}^{\rms*}(t)\rVert_R^2+\sum_{k=L-2n}^{L-n-1}\lVert\check{u}_k'(t+n)-\check{u}^{\rms*}(t)\rVert_R^2\\\nonumber
&+\sum_{k=0}^{n-1}\lVert u_{t+k}-\check{u}^{\rms*}(t)\rVert_R^2-\lVert\check{u}_{k}^*(t)-\check{u}^{\rms*}(t)\rVert_R^2.
\end{align}
For $k\in\mathbb{I}_{[-n,L]}$, it holds that
\begin{align}\label{eq:thm_stab_affine_nominal_proof_case1_input_V_upper_bound}
\lVert\check{u}_k^*(t)-\check{u}^{\rms*}(t)\rVert_R^2&\leq V(\xi_t).
\end{align}
Together with the fact that $\lVert u_{t+k}-\check{u}_k^*(t)\rVert_2\leq\bar{\varepsilon}$ for $k\in\mathbb{I}_{[0,n-1]}$ (compare~\eqref{eq:thm_stab_affine_nominal_input_disturbance}) and using~\eqref{eq:ab_ineq}, this leads to
\begin{align}\label{eq:thm_stab_affine_nominal_proof_case1_input_first_steps_bound}
&\sum_{k=0}^{n-1}\lVert u_{t+k}-\check{u}^{\rms*}(t)\rVert_R^2-\lVert\check{u}_{k}^*(t)-\check{u}^{\rms*}(t)\rVert_R^2\\\nonumber
\leq&c_{\rmu,1}\bar{\varepsilon}^2+c_{\rmu,2}\bar{\varepsilon}\sqrt{V(\xi_t)}
\end{align}
for some $c_{\rmu,1},c_{\rmu,2}>0$.
Further, the second term on the right-hand side of~\eqref{eq:thm_stab_affine_nominal_proof_case1_input_bound} is bounded as
\begin{align}\label{eq:thm_stab_affine_nominal_proof_case1_input_ctrb_bound}
\sum_{k=L-2n}^{L-n-1}\lVert\check{u}_k'(t+n)-\check{u}^{\rms*}(t)\rVert_R^2\stackrel{\eqref{eq:thm_stab_affine_nominal_proof_case1_ctrb},\eqref{eq:thm_stab_affine_nominal_proof_case1_x_bound}}{\leq}&
\lambda_{\max}(R)\Gamma c_{\rmx,2}\bar{\varepsilon}^2.
\end{align}
Next, we analyze the terms in~\eqref{eq:thm_stab_affine_nominal_proof_case1_cost_decay1} depending on the output trajectory.
Inequalities~\eqref{eq:thm_stab_affine_nominal_proof_case1_ctrb} and~\eqref{eq:thm_stab_affine_nominal_proof_case1_x_bound} imply
\begin{align}\label{eq:thm_stab_affine_nominal_proof_case1_output_bound0}
\sum_{k=L-2n}^{L-n-1}\lVert\check{y}_k'(t+n)-\check{y}^{\rms*}(t)\rVert_Q^2\leq\lambda_{\max}(Q)\Gamma c_{\rmx,2}\bar{\varepsilon}^2.
\end{align}
The trajectories $\{\check{y}_k'(t+n)\}_{k=0}^{L-2n-1}$ and $\{\check{y}_{k+n}^*(t)\}_{k=0}^{L-2n-1}$ differ only in terms of their initial conditions which in turn differ linearly in terms of $d_{[t,t+n-1]}$.
Hence, following the arguments above leading to~\eqref{eq:thm_stab_affine_nominal_proof_case1_input_first_steps_bound}, there exist $c_{y,1},c_{y,2}>0$, such that for $k\in\mathbb{I}_{[-n,L-2n-1]}$
\begin{align}\label{eq:thm_stab_affine_nominal_proof_case1_output_bound}
&\lVert\check{y}_k'(t+n)-\check{y}^{\rms*}(t)\rVert_Q^2-\lVert\check{y}_{k+n}^*(t)-\check{y}^{\rms*}(t)\rVert_Q^2\\\nonumber
\leq&c_{y,1}\bar{\varepsilon}^2+c_{y,2}\bar{\varepsilon}\sqrt{V(\xi_t)}
\end{align}
Finally, by definition of $\check{\alpha}'(t+n)$ in~\eqref{eq:thm_stab_affine_nominal_proof_case1_alpha_cand}, we have
\begin{align}\label{eq:thm_stab_affine_nominal_proof_case1_alpha_bound_aux}
&\lVert\check{\alpha}'(t+n)-\alpha^{\rms\rmr}\rVert_2^2\\\nonumber
\leq&\lVert H_{ux}^\dagger\rVert_2^2\left\lVert\begin{bmatrix}\check{u}_{[-n,L]}'(t+n)-\bbone_{L+n+1}\otimes u^{\rms\rmr}\\
x_t-x^{\rms\rmr}\\0\end{bmatrix}\right\rVert_2^2\\\nonumber
=&\lVert H_{ux}^\dagger\rVert_2^2(\lVert\check{u}_{[0,n-1]}^*(t)+d_{[t,t+n-1]}-\bbone_n\otimes u^{\rms\rmr}\rVert_2^2\\\nonumber
&+\lVert\check{u}_{[n,L-n-1]}^*(t)-\bbone_{L-2n}\otimes u^{\rms\rmr}\rVert_2^2\\\nonumber
&+(n+1)\lVert\check{u}^{\rms*}(t)-u^{\rms\rmr}\rVert_2^2\\\nonumber
&+\lVert\check{u}_{[L-2n,L-n-1]}'(t+n)-\bbone_n\otimes u^{\rms\rmr}\rVert_2^2+\lVert x_t-x^{\rms\rmr}\rVert_2^2).
\end{align}
Note that
\begin{align}\label{eq:thm_stab_affine_nominal_proof_case1_us_usr_bound}
\lVert\check{u}^{\rms*}(t)-u^{\rms\rmr}\rVert_2^2&\stackrel{\eqref{eq:ass_affine_unique_steady_state}}{\leq} c_g\lVert\check{y}^{\rms*}(t)-y^{\rms\rmr}\rVert_2^2\\\nonumber
&\stackrel{\eqref{eq:ab_ineq2}}{\leq}2c_g(\lVert\check{y}^{\rms*}(t)-y^\rmr\rVert_2^2+\lVert y^{\rms\rmr}-y^\rmr\rVert_2^2)\\\nonumber
&\leq\frac{2c_g}{\lambda_{\min}(S)}
(V(\xi_t)+2J_{\mathrm{eq}}^*),
\end{align}
where we exploit $\check{J}_L^*(\xi_t)=V(\xi_t)+J_{\mathrm{eq}}^*$ for the last inequality.
Using this inequality as well as~\eqref{eq:ab_ineq2},~\eqref{eq:thm_stab_affine_nominal_lower_upper},~\eqref{eq:thm_stab_affine_nominal_proof_case1_ctrb},~\eqref{eq:thm_stab_affine_nominal_proof_case1_x_bound}, and~\eqref{eq:thm_stab_affine_nominal_proof_case1_input_V_upper_bound}, it is straightforward to verify that the input-dependent terms in~\eqref{eq:thm_stab_affine_nominal_proof_case1_alpha_bound_aux} are bounded by $c_{\alpha,1}J_{\mathrm{eq}}^*+c_{\alpha,2}\lVert\xi_t-\xi^{\rms\rmr}\rVert_2^2+c_{\alpha,3}\bar{\varepsilon}^2$ for some $c_{\alpha,i}>0$, $i\in\mathbb{I}_{[1,3]}$ i.e.,
\begin{align}\label{eq:thm_stab_affine_nominal_proof_case1_alpha_bound}
&\lVert\check{\alpha}'(t+n)-\alpha^{\rms\rmr}\rVert_2^2\\\nonumber
\leq &c_{\alpha,1}J_{\mathrm{eq}}^*+c_{\alpha,2}\lVert\xi_t-\xi^{\rms\rmr}\rVert_2^2+c_{\alpha,3}\bar{\varepsilon}^2+\lVert H_{ux}^\dagger\rVert_2^2\lVert x_t-x^{\rms\rmr}\rVert_2^2.
\end{align}
Plugging the bounds~\eqref{eq:thm_stab_affine_nominal_proof_case1_input_first_steps_bound},~\eqref{eq:thm_stab_affine_nominal_proof_case1_input_ctrb_bound} for the input,~\eqref{eq:thm_stab_affine_nominal_proof_case1_output_bound0},~\eqref{eq:thm_stab_affine_nominal_proof_case1_output_bound} for the output, and~\eqref{eq:thm_stab_affine_nominal_proof_case1_alpha_bound} for $\check{\alpha}'(t+n)$ into~\eqref{eq:thm_stab_affine_nominal_proof_case1_cost_decay1}, and using~\eqref{eq:trafo_T_extended_state}, we obtain
\begin{align}
&V(\xi_{t+n})-V(\xi_t)\\\nonumber
\leq&-\sum_{k=-n}^{-1}(\lVert u_{t+k}-\check{u}^{\rms*}(t)\rVert_R^2+\lVert y_{t+k}-\check{y}^{\rms*}(t)\rVert_Q^2)\\\nonumber
&+c_{\rmJ,1}\bar{\varepsilon}^{\beta_{\alpha}}\lVert \xi_t-\xi^{\rms\rmr}\rVert_2^2+c_{\rmJ,2}\bar{\varepsilon}\sqrt{V(\xi_t)}\\\nonumber
&+c_{\rmJ,3}\bar{\varepsilon}^{\beta_{\alpha}}J_{\mathrm{eq}}^*+c_{\rmJ,4}(\bar{\varepsilon}^2+\bar{\varepsilon}^{2+\beta_{\alpha}})
\end{align}
for some $c_{\rmJ,i}>0$, $i\in\mathbb{I}_{[1,4]}$.
Note that
\begin{align*}
&-\sum_{k=-n}^{-1}(\lVert u_{t+k}-\check{u}^{\rms*}(t)\rVert_R^2+\lVert y_{t+k}-\check{y}^{\rms*}(t)\rVert_Q^2)\\
\stackrel{\eqref{eq:ass_affine_unique_steady_state},\eqref{eq:thm_stab_affine_nominal_proof_case1}}{\leq}&-\frac{1}{2}\sum_{k=-n}^{-1}(\lVert u_{t+k}-\check{u}^{\rms*}(t)\rVert_R^2+\lVert y_{t+k}-\check{y}^{\rms*}(t)\rVert_Q^2)\\
&-\frac{\gamma}{2}\left(\frac{\lambda_{\min}(S)}{2c_g}\lVert \check{u}^{\rms*}(t)-u^{\rms\rmr}\rVert_2^2+\frac{1}{2}\lVert\check{y}^{\rms*}(t)-y^{\rms\rmr}\rVert_S^2\right)\\
\stackrel{\eqref{eq:ab_ineq2}}{\leq}&
-c_{\rmJ,6}\sum_{k=-n}^{-1}(\lVert u_{t+k}-u^{\rms\rmr}\rVert_2^2+\lVert y_{t+k}-y^{\rms\rmr}\rVert_2^2)
\end{align*}
with some $c_{\rmJ,6}>0$.
This implies
\begin{align}\label{eq:thm_stab_affine_nominal_proof_case1_cost_decay2}
&V(\xi_{t+n})-V(\xi_t)\\\nonumber
\leq&-(c_{\rmJ,6}-c_{\rmJ,1}\bar{\varepsilon}^{\beta_{\alpha}})\lVert\xi_t-\xi^{\rms\rmr}\rVert_2^2+c_{\rmJ,2}\bar{\varepsilon}\sqrt{V(\xi_t)}\\\nonumber
&+c_{\rmJ,3}\bar{\varepsilon}^{\beta_{\alpha}}J_{\mathrm{eq}}^*+c_{\rmJ,4}(\bar{\varepsilon}^2+\bar{\varepsilon}^{2+\beta_{\alpha}}).
\end{align}
\textbf{(iii) Candidate solution 2}\\
Assume
\begin{align}\label{eq:thm_stab_affine_nominal_proof_case2}
&\sum_{k=-n}^{-1}\lVert u_{t+k}-\check{u}^{\rms*}(t)\rVert_R^2+\lVert y_{t+k}-\check{y}^{\rms*}(t)\rVert_Q^2\\\nonumber
\leq&\gamma\lVert\check{y}^{\rms*}(t)-y^{\rms\rmr}\rVert_S^2.
\end{align}
\textbf{(iii).a Definition of candidate solution}\\
We choose the equilibrium candidate as a convex combination of $(\check{u}^{\rms*}(t),\check{y}^{\rms*}(t))$ and the optimal reachable equilibrium, i.e.,
\begin{align}\label{eq:thm_stab_affine_nominal_proof_case2_uys_def}
\hat{u}^\rms(t+n)&=\lambda \check{u}^{\rms*}(t)+(1-\lambda)u^{\rms\rmr},\\\nonumber
\hat{y}^\rms(t+n)&=\lambda \check{y}^{\rms*}(t)+(1-\lambda)y^{\rms\rmr}
\end{align}
for some $\lambda\in(0,1)$ which will be fixed later in the proof, and we denote the corresponding state by $\hat{x}^\rms(t+n)$.
By controllability, there exists an input steering the system from $x_{t+n}$ to $\hat{x}^\rms(t+n)$ in $L-n\geq n$ steps while satisfying
\begin{align}\nonumber
&\sum_{k=0}^{L}\lVert\hat{u}_k(t+n)-\hat{u}^\rms(t+n)\rVert_2^2+\lVert\hat{y}_k(t+n)-\hat{y}^\rms(t+n)\rVert_2^2\\\label{eq:thm_stab_affine_nominal_proof_case2_ctrb}
&\leq \Gamma\lVert x_{t+n}-\hat{x}^\rms(t+n)\rVert_2^2
\end{align}
for some $\Gamma>0$.
In the following, we show that $\hat{u}_k(t+n)\in\mathbb{U}$, $k\in\mathbb{I}_{[0,L]}$ if $\gamma$, $(1-\lambda)$, and $\bar{\varepsilon}$ are sufficiently small.
Denoting the extended state~\eqref{eq:xi_def} corresponding to $(\check{u}^{\rms*}(t),\check{y}^{\rms*}(t))$ by $\check{\xi}^{\rms*}(t)$, we have
\begin{align}\label{eq:thm_stab_affine_nominal_proof_case2_uy_tracking_cost_bound}
&\quad\sum_{k=0}^{n-1}\lVert \check{u}_k^*(t)-\check{u}^{\rms*}(t)\rVert_2^2+\lVert\check{y}_k^*(t)-\check{y}^{\rms*}(t)\rVert_2^2\\\nonumber
&\leq\sum_{k=0}^{L}\lVert \check{u}_k^*(t)-\check{u}^{\rms*}(t)\rVert_2^2+\lVert\check{y}_k^*(t)-\check{y}^{\rms*}(t)\rVert_2^2\\\nonumber
&\leq\bar{c}_{\rmu}(\lVert\xi_t-\check{\xi}^{\rms*}(t)\rVert_2^2+\lVert \check{x}^{\rms*}(t)-x^{\rms\rmr}\rVert_2^2+\lVert \check{u}^{\rms*}(t)-u^{\rms\rmr}\rVert_2^2)\\\nonumber
&\stackrel{\eqref{eq:ass_affine_unique_steady_state},\eqref{eq:thm_stab_affine_nominal_proof_case2}}{\leq}
\bar{c}_{\rmu}\left(\frac{\lambda_{\max}(S)}{\lambda_{\min}(Q,R)}\bar{c}_{\rmu}\gamma+c_g\right)\lVert\check{y}^{\rms*}(t)-y^{\rms\rmr}\rVert_2^2
\end{align}
for a suitable constant $\bar{c}_{\rmu}>0$.
The second inequality in~\eqref{eq:thm_stab_affine_nominal_proof_case2_uy_tracking_cost_bound} can be shown analogously to the upper bound in Part (i).b of the proof, using a controllability argument based on~\eqref{eq:thm_stab_affine_nominal_proof_case2} with a sufficiently small $\gamma$ and bounding $\lVert\alpha^{\rms*}(t)-\alpha^{\rms\rmr}\rVert_2^2$ with $\alpha^{\rms*}(t)\coloneqq H_{ux}^\dagger\begin{bmatrix}\bbone_{L+n+1}\otimes\check{u}^{\rms*}(t)\\\check{x}^{\rms*}(t)\\1\end{bmatrix}$ in terms of $\lVert \check{x}^{\rms*}(t)-x^{\rms\rmr}\rVert_2^2$ and $\lVert \check{u}^{\rms*}(t)-u^{\rms\rmr}\rVert_2^2$.
Moreover, for $k\in\mathbb{I}_{[0,n-1]}$, the difference
\begin{align*}
\lVert u_{t+k}-\check{u}^{\rms*}(t)\rVert_2^2-\lVert\check{u}_k^*(t)-\check{u}^{\rms*}(t)\rVert_2^2
\end{align*}
is bounded as in~\eqref{eq:thm_stab_affine_nominal_proof_case1_input_first_steps_bound}, and similarly for the output, cf.~\eqref{eq:thm_stab_affine_nominal_proof_case1_output_bound}.
Thus, adding and subtracting $\lVert\check{u}_k^*(t)-\check{u}^{\rms*}(t)\rVert_2^2$ and $\lVert\check{y}_k^*(t)-\check{y}^{\rms*}(t)\rVert_2^2$, we obtain
\begin{align}\label{eq:thm_stab_affine_nominal_proof_case2_xi_steady_state_bound}
&\lVert\xi_{t+n}-\check{\xi}^{\rms*}(t)\rVert_2^2\\\nonumber
=&\sum_{k=0}^{n-1}\lVert u_{t+k}-\check{u}^{\rms*}(t)\rVert_2^2-\lVert\check{u}_k^{*}(t)-\check{u}^{\rms*}(t)\rVert_2^2\\\nonumber
&+\sum_{k=0}^{n-1}\lVert y_{t+k}-\check{y}^{\rms*}(t)\rVert_2^2-\lVert\check{y}_k^{*}(t)-\check{y}^{\rms*}(t)\rVert_2^2\\\nonumber
&+\sum_{k=0}^{n-1}\lVert\check{u}_k^{*}(t)-\check{u}^{\rms*}(t)\rVert_2^2+\lVert\check{y}_k^{*}(t)-\check{y}^{\rms*}(t)\rVert_2^2\\\nonumber
\stackrel{\eqref{eq:thm_stab_affine_nominal_proof_case1_input_first_steps_bound},\eqref{eq:thm_stab_affine_nominal_proof_case1_output_bound},\eqref{eq:thm_stab_affine_nominal_proof_case2_uy_tracking_cost_bound}}{\leq}
&\bar{c}_{\rmx,1}\gamma\lVert\check{y}^{\rms*}(t)-y^{\rms\rmr}\rVert_2^2+\bar{c}_{\rmx,2}\bar{\varepsilon}^2+\bar{c}_{\rmx,3}\bar{\varepsilon}\sqrt{V(\xi_t)}
\end{align}
for some $\bar{c}_{\rmx,i}>0$, $i\in\mathbb{I}_{[1,3]}$.
Writing $\hat{\xi}^\rms(t+n)$ for the extended state~\eqref{eq:xi_def} corresponding to $(\hat{u}^\rms(t+n),\hat{y}^\rms(t+n))$, it holds that
\begin{align}\label{eq:thm_stab_affine_nominal_proof_case2_xi_steady_state_bound2}
&\lVert\check{\xi}^{\rms*}(t)-\hat{\xi}^\rms(t+n)\rVert_2^2\stackrel{\eqref{eq:thm_stab_affine_nominal_proof_case2_uys_def}}{=}(1-\lambda)^2\lVert\check{\xi}^{\rms*}(t)-\xi^{\rms\rmr}\rVert_2^2\\\nonumber
&\stackrel{\eqref{eq:ass_affine_unique_steady_state},\eqref{eq:xi_def}}{\leq}
(1-\lambda)^2n(1+c_g)\lVert\check{y}^{\rms*}(t)-y^{\rms\rmr}\rVert_2^2.
\end{align}
Combining these bounds, we obtain
\begin{align}\label{eq:thm_stab_affine_nominal_proof_case2_xtn_xs_bound}
&\quad\lVert x_{t+n}-\hat{x}^\rms(t+n)\rVert_2^2\stackrel{\eqref{eq:trafo_T_extended_state}}{\leq}
\lVert T_{\rmx}\rVert_2^2\lVert\xi_{t+n}-\hat{\xi}^\rms(t+n)\rVert_2^2\\\nonumber
&\leq2\lVert T_{\rmx}\rVert_2^2(\lVert\xi_{t+n}-\check{\xi}^{\rms*}(t)\rVert_2^2+\lVert\check{\xi}^{\rms*}(t)-\hat{\xi}^\rms(t+n)\rVert_2^2)\\\nonumber
&\stackrel{\eqref{eq:thm_stab_affine_nominal_proof_case2_xi_steady_state_bound},\eqref{eq:thm_stab_affine_nominal_proof_case2_xi_steady_state_bound2}}{\leq}
2\lVert T_{\rmx}\rVert_2^2\Big( (\bar{c}_{\rmx,1}\gamma+n(1+c_g)(1-\lambda)^2)\lVert\check{y}^{\rms*}(t)-y^{\rms\rmr}\rVert_2^2\\\nonumber
&\quad+\bar{c}_{\rmx,2}\bar{\varepsilon}^2+\bar{c}_{\rmx,3}\bar{\varepsilon}\sqrt{V(\xi_t)}\Big).
\end{align}
Moreover, similar to~\eqref{eq:thm_stab_affine_nominal_proof_case1_us_usr_bound}, we have
\begin{align}\label{eq:thm_stab_affine_nominal_proof_case2_ys_ysr_bound}
\lVert\check{y}^{\rms*}(t)-y^{\rms\rmr}\rVert_2^2\leq&\frac{2}{\lambda_{\min}(S)}(V(\xi_t)+2J_{\mathrm{eq}}^*).
\end{align}
Thus, using $V(\xi_t)\leq V_{\mathrm{ROA}}$, if $\gamma$, $(1-\lambda)$, and $\bar{\varepsilon}$ are sufficiently small, then $x_{t+n}$ is sufficiently close to $\hat{x}^\rms(t+n)$ such that~\eqref{eq:thm_stab_affine_nominal_proof_case2_ctrb} and $\hat{u}^\rms(t+n)\in\mathrm{int}(\mathbb{U})$ ensure $\hat{u}_k(t+n)\in\mathbb{U}$ for $k\in\mathbb{I}_{[0,L]}$.

Further, choosing the input-output candidate $(\hat{u}(t+n),\hat{y}(t+n))$ such that $(\hat{u}_k(t+n),\hat{y}_k(t+n))=(u_{t+n+k},y_{t+n+k})$ for $k\in\mathbb{I}_{[-n,-1]}$ and $(\hat{u}_k(t+n),\hat{y}_k(t+n))=(\hat{u}^\rms(t+n),\hat{y}^\rms(t+n))$ for $k\in\mathbb{I}_{[L-n,L]}$ satisfies~\eqref{eq:DD_MPC_affine_nominal_init} and~\eqref{eq:DD_MPC_affine_nominal_TEC}, respectively.
Finally, with $H_{ux}$ as in~\eqref{eq:Hux_affine}, we choose
\begin{align}
\hat{\alpha}(t+n)=H_{ux}^\dagger\begin{bmatrix}\hat{u}(t+n)\\x_t\\1\end{bmatrix},
\end{align}
which implies that all constraints of~\eqref{eq:DD_MPC_affine_nominal} are fulfilled.\\
\textbf{(iii).b Lyapunov function decay}\\
Using the above candidate solution, we obtain
\begin{align}\label{eq:thm_stab_affine_nominal_proof_case2_cost_decay1}
&V(\xi_{t+n})-V(\xi_t)\\\nonumber
\leq&\sum_{k=-n}^L\lVert\hat{u}_k(t+n)-\hat{u}^\rms(t+n)\rVert_R^2+\lVert\hat{y}_k(t+n)-\hat{y}^\rms(t+n)\rVert_Q^2
\\\nonumber
&+\lambda_{\alpha}\bar{\varepsilon}^{\beta_{\alpha}}(\lVert\hat{\alpha}(t+n)-\alpha^{\rms\rmr}\rVert_2^2-\lVert\check{\alpha}^*(t)-\alpha^{\rms\rmr}\rVert_2^2)\\\nonumber
&-\sum_{k=0}^L(\lVert\check{u}_k^*(t)-\check{u}^{\rms*}(t)\rVert_R^2+\lVert\check{y}_k^*(t)-\check{y}^{\rms*}(t)\rVert_Q^2)\\\nonumber
&+\lVert\hat{y}^\rms(t+n)-y^\rmr\rVert_S^2-\lVert\check{y}^{\rms*}(t)-y^\rmr\rVert_S^2.
\end{align}
Similar to~\cite[Inequality (19)]{koehler2020nonlinear}, strong convexity of the cost in~\eqref{eq:opt_reach_equil_affine} implies
\begin{align}\label{eq:thm_stab_affine_nominal_proof_case2_equil_bound}
&\lVert\hat{y}^\rms(t+n)-y^\rmr\rVert_S^2-\lVert\check{y}^{\rms*}(t)-y^\rmr\rVert_S^2\\\nonumber
\leq&-(1-\lambda^2)\lVert\check{y}^{\rms*}(t)-y^{\rms\rmr}\rVert_S^2.
\end{align}
The definition of $\hat{\alpha}(t+n)$ implies
\begin{align}\label{eq:thm_stab_affine_nominal_proof_case2_alpha_bound_intermediate}
&\lVert\hat{\alpha}(t+n)-\alpha^{\rms\rmr}\rVert_2^2\\\nonumber
\leq&\lVert H_{ux}^\dagger\rVert_2^2(\lVert\hat{u}(t+n)-\bbone_{L+n+1}\otimes u^{\rms\rmr}\rVert_2^2+\lVert x_t-x^{\rms\rmr}\rVert_2^2).
\end{align}
Using $\hat{u}^\rms(t+n)-u^{\rms\rmr}=\lambda(\check{u}^{\rms*}(t)-u^{\rms\rmr})$, as well as~\eqref{eq:ab_ineq2},~\eqref{eq:thm_stab_affine_nominal_lower_upper},~\eqref{eq:thm_stab_affine_nominal_proof_case1_us_usr_bound},~\eqref{eq:thm_stab_affine_nominal_proof_case2_ctrb},~\eqref{eq:thm_stab_affine_nominal_proof_case2_xtn_xs_bound}, and~\eqref{eq:thm_stab_affine_nominal_proof_case2_ys_ysr_bound} we obtain
\begin{align}\label{eq:thm_stab_affine_nominal_proof_case2_alpha_bound}
&\lVert\hat{\alpha}(t+n)-\alpha^{\rms\rmr}\rVert_2^2\\\nonumber
\leq&\bar{c}_{\alpha,1}J_{\mathrm{eq}}^*+\bar{c}_{\alpha,2}\lVert\xi_t-\xi^{\rms\rmr}\rVert_2^2+\bar{c}_{\alpha,3}\bar{\varepsilon}^2+\bar{c}_{\alpha,4}\bar{\varepsilon}\sqrt{V(\xi_t)}\\\nonumber
&+\lVert H_{ux}^\dagger\rVert_2^2\lVert x_t-x^{\rms\rmr}\rVert_2^2
\end{align}
for some $\bar{c}_{\alpha,i}>0$, $i\in\mathbb{I}_{[1,4]}$.
Plugging~\eqref{eq:thm_stab_affine_nominal_proof_case2_ctrb},~\eqref{eq:thm_stab_affine_nominal_proof_case2_equil_bound}, and~\eqref{eq:thm_stab_affine_nominal_proof_case2_alpha_bound} into~\eqref{eq:thm_stab_affine_nominal_proof_case2_cost_decay1} and using
\begin{align*}
&\sum_{k=-n}^{-1}\lVert\hat{u}_k(t+n)-\hat{u}^\rms(t+n)\rVert_R^2+\lVert\hat{y}_k(t+n)-\hat{y}^\rms(t+n)\rVert_Q^2\\
&\leq\lambda_{\max}(Q,R)\lVert\xi_{t+n}-\hat{\xi}^\rms(t+n)\rVert_2^2
\end{align*}
as well as~\eqref{eq:trafo_T_extended_state}, there exist $\bar{c}_{\rmJ,i}>0$, $i\in\mathbb{I}_{[1,7]}$ such that
\begin{align}
&\quad V(\xi_{t+n})-V(\xi_t)\\\nonumber
&\leq-\sum_{k=-n}^{-1}(\lVert u_{t+k}-\check{u}^{\rms*}(t)\rVert_R^2+\lVert y_{t+k}-\check{y}^{\rms*}(t)\rVert_Q^2)\\\nonumber
&\quad+\bar{c}_{\rmJ,1}\lVert \xi_{t+n}-\hat{\xi}^\rms(t+n)\rVert_2^2-(1-\lambda^2)\lVert\check{y}^{\rms*}(t)-y^{\rms\rmr}\rVert_S^2\\\nonumber
&\quad+\lambda_{\alpha}\bar{\varepsilon}^{\beta_{\alpha}}\Big(\bar{c}_{\alpha,1}J_{\mathrm{eq}}^*+(\bar{c}_{\alpha,2}+\lVert H_{ux}^\dagger\rVert_2^2\lVert T_{\rmx}\rVert_2^2)
\lVert \xi_t-\xi^{\rms\rmr}\rVert_2^2\\\nonumber
&\quad +\bar{c}_{\alpha,3}\bar{\varepsilon}^2+\bar{c}_{\alpha,4}\bar{\varepsilon}\sqrt{V(\xi_t)}\Big)\\\nonumber
&\stackrel{\eqref{eq:thm_stab_affine_nominal_proof_case2_xtn_xs_bound}}{\leq}
-\sum_{k=-n}^{-1}(\lVert u_{t+k}-\check{u}^{\rms*}(t)\rVert_R^2+\lVert y_{t+k}-\check{y}^{\rms*}(t)\rVert_Q^2)\\\nonumber
&+(\bar{c}_{\rmJ,2}\gamma+\bar{c}_{\rmJ,3}(1-\lambda)^2-\lambda_{\min}(S)(1-\lambda^2))\lVert \check{y}^{\rms*}(t)-y^{\rms\rmr}\rVert_2^2\\\nonumber
&+\bar{c}_{\rmJ,4}\bar{\varepsilon}^{\beta_{\alpha}}\lVert\xi_t-\xi^{\rms\rmr}\rVert_2^2+\bar{c}_{\rmJ,5}(\bar{\varepsilon}+\bar{\varepsilon}^{1+\beta_{\alpha}})\sqrt{V(\xi_t)}\\\nonumber
&+\bar{c}_{\rmJ,6}\bar{\varepsilon}^{\beta_{\alpha}}J_{\mathrm{eq}}^*+\bar{c}_{\rmJ,7}(\bar{\varepsilon}^2+\bar{\varepsilon}^{2+\beta_{\alpha}}).
\end{align}
If $\gamma$ and $(1-\lambda)$ are sufficiently small such that 
\begin{align*}
\bar{c}_{\rmJ,2}\gamma+\bar{c}_{\rmJ,3}(1-\lambda)^2-\bar{c}_{\rmJ,4}(1-\lambda^2)<0,
\end{align*}
then~\eqref{eq:ass_affine_unique_steady_state} and~\eqref{eq:ab_ineq2} lead to
\begin{align}\label{eq:thm_stab_affine_nominal_proof_case2_Lyapunov_decay}
&V(\xi_{t+n})-V(\xi_t)\\\nonumber
\leq&-(\bar{c}_{\rmJ,9}-\bar{c}_{\rmJ,4}\bar{\varepsilon}^{\beta_{\alpha}})\lVert\xi_t-\xi^{\rms\rmr}\rVert_2^2+\bar{c}_{\rmJ,6}\bar{\varepsilon}^{\beta_{\alpha}}J_{\mathrm{eq}}^*\\\nonumber
&+\bar{c}_{\rmJ,5}(\bar{\varepsilon}+\bar{\varepsilon}^{1+\beta_{\alpha}})\sqrt{V(\xi_t)}+\bar{c}_{\rmJ,7}(\bar{\varepsilon}^2+\bar{\varepsilon}^{2+\beta_{\alpha}})
\end{align}
for some $\bar{c}_{\rmJ,9}>0$.\\
\textbf{(iv) Practical stability}\\
Using~\eqref{eq:thm_stab_affine_nominal_proof_case1_cost_decay2} and~\eqref{eq:thm_stab_affine_nominal_proof_case2_Lyapunov_decay} and letting $\bar{\varepsilon}<1$, there exist $\tilde{c}_{J,i}>0$, $i\in\mathbb{I}_{[1,5]}$, such that
\begin{align}\label{eq:thm_stab_affine_nominal_proof_Lyapunov_decay1}
V(\xi_{t+n})-V(\xi_t)
\leq&-(\tilde{c}_{\rmJ,4}-\tilde{c}_{\rmJ,5}\bar{\varepsilon}^{\beta_{\alpha}})\lVert\xi_t-\xi^{\rms\rmr}\rVert_2^2\\\nonumber
&+\tilde{c}_{\rmJ,1}\bar{\varepsilon}\sqrt{V(\xi_t)}+\tilde{c}_{\rmJ,2}\bar{\varepsilon}^{\beta_{\alpha}}J_{\mathrm{eq}}^*+\tilde{c}_{\rmJ,3}\bar{\varepsilon}^2.
\end{align}
If $\bar{\varepsilon}_{\max}<\left(\frac{\tilde{c}_{\rmJ,4}}{\tilde{c}_{\rmJ,5}}\right)^{\frac{1}{\beta_{\alpha}}}$, then this together with~\eqref{eq:thm_stab_affine_nominal_lower_upper} and $V(\xi_t)\leq V_{\mathrm{ROA}}$ implies~\eqref{eq:thm_stab_affine_nominal} for $\check{c}_{\rmV}\coloneqq1-\frac{\tilde{c}_{\rmJ,4}-\tilde{c}_{\rmJ,5}\bar{\varepsilon}^{\beta_{\alpha}}}{c_{\rmu}}<1$ and some $\beta_{\rmd}\in\mathcal{K}_{\infty}$. 
If $\bar{\varepsilon}_{\max}$ is sufficiently small, then $V(\xi_{t+n})\leq V_{\mathrm{ROA}}$ such that the MPC scheme is recursively feasible and Inequalities~\eqref{eq:thm_stab_affine_nominal_lower_upper} and~\eqref{eq:thm_stab_affine_nominal} hold for all $t=ni$, $i\in\mathbb{I}_{\geq0}$.
\end{proof}

\end{document}